\documentclass[11pt,letterpaper]{article}

\usepackage[margin=1in]{geometry}

\usepackage{lineno}

\usepackage{cite}
\usepackage{hyperref}
\usepackage{colortbl}
\usepackage{paralist}
\usepackage{enumerate}
\usepackage{pdfpages}
\usepackage{float}
\usepackage{booktabs}
\usepackage[override]{cmtt}
\usepackage{color,xcolor}
\usepackage{graphicx,eso-pic}
\usepackage{boxedminipage}
\usepackage{url}
\usepackage{caption}
\usepackage{subcaption}
\usepackage{xspace}
\usepackage{multirow}
\usepackage{amsmath,amsthm,amstext,amssymb,amsfonts,latexsym}
\usepackage{wrapfig}
\usepackage{algorithm}
\usepackage{algpseudocode}
\usepackage{algorithmicx}
\usepackage{epstopdf}
\usepackage{mdframed}
\usepackage{cases}
\usepackage[compact]{titlesec}
\usepackage{bm}
\usepackage{tcolorbox}
\tcbuselibrary{skins, breakable}  

\usepackage{tikz-cd}
\usepackage{graphicx}  
\usepackage{mathrsfs}

\newcommand{\pr}{\mathrm{pr}} % or \mathrm{proj}
\newcommand{\Over}{\mathsf{Over}} 
\newcommand{\Fit}{\mathsf{Fit}} 
\newcommand{\Feas}{\mathsf{Feas}}

\renewcommand{\c}{\mathsf{c}}

\newcounter{note}[section]

\newcommand{\yuetodo}[1]{{\large\color{green}[Yue todo: #1]}}

\definecolor{yue}{rgb}{0.7, 0, 0}

\renewcommand{\yuetodo}[1]{}

\newcommand{\B}{\mathbb{B}}

\definecolor{darkgreen}{rgb}{0,0.5,0}
\definecolor{lightblue}{RGB}{0,176,240}
\definecolor{darkblue}{RGB}{0,112,192}
\definecolor{lightpurple}{RGB}{124, 66, 168}
\definecolor{grey}{RGB}{139, 137, 137}
\definecolor{maroon}{RGB}{178, 34, 34}
\definecolor{green}{RGB}{34, 139, 34}
\definecolor{types}{RGB}{72, 61, 139}
\definecolor{gold}{rgb}{0.8, 0.33, 0.0}

\definecolor{darkgray}{gray}{0.3}

\definecolor{darkred}{rgb}{0.5, 0, 0}
\definecolor{darkgreen}{rgb}{0, 0.5, 0}
\definecolor{darkblue}{rgb}{0,0,0.5}

\newcommand\markx[2]{}

\newcommand{\ignore}[1]{}

\newcounter{task}

\newtheorem{theorem}{Theorem}[section]

\newtheorem{corollary}[theorem]{Corollary}
\newtheorem{fact}[theorem]{Fact}

\newtheorem{lemma}[theorem]{Lemma}

\newtheorem{proposition}[theorem]{Proposition}

\theoremstyle{definition}
\newtheorem{definition}[theorem]{Definition}
\newtheorem{construction}[theorem]{Construction}
\newtheorem{remark}[theorem]{Remark}

\newtheorem{assumption}[theorem]{Assumption}
\theoremstyle{plain}

\newcounter{cnt:challenge}

\newcommand{\D}{\mathsf{D}}
\newcommand{\K}{\mathsf{K}}

\newcommand{\PR}{\mathsf{PR}}

\newcommand{\dom}{\ensuremath{\mathsf{dom}}}

\newcommand{\kprod}{\bowtie}

\newcommand{\Mplus}{\mathcal{M}_+}

\newcommand{\Pairi}{\Pi^{\mathcal E}_\iota \times \Pi^{\mathcal E}_{\bar\iota}}
\newcommand{\HS}{\mathsf{HS}}

\newcommand{\id}{\mathsf{id}}

\newcommand{\ac}{\mathrm{ac}}

\title{Universal Closest Refinement on Measurable Bipartite Relations}

\author{T-H. Hubert Chan\thanks{The University of Hong Kong.}}

\date{}

\begin{document}

\begin{titlepage}
    \maketitle

    \begin{abstract}

We study the universal closest refinement problem for closed measurable bipartite relations between Polish spaces equipped with finite Borel measures. Given prescribed side measures, the feasible class consists of finite refinement plans concentrated on the relation and carrying one fixed marginal. The main question is whether this highly nonunique class nevertheless contains a mathematically distinguished class of refinements. We show that the correct one-sided extremal criterion is level-optimal maximin, a levelwise maximin principle formulated through truncation and overflow profiles. We then characterize this structure through convex refinement: every minimizer of finite objective value for a strictly convex refinement criterion is level-optimal maximin, while every level-optimal maximin refinement minimizes every proper lower semicontinuous convex $\vartheta$-divergence. Proportional response identifies an opposite-side partner and yields a universally closest refinement pair. Our main theorem shows that every pair obtained in this way is universally closest among all feasible pairs for every divergence satisfying the data-processing inequality under measurable post-processing, and conversely that every universally closest pair has this structure. We also prove the corresponding structural characterization for finite-valued strictly convex closest pairs. Finally, we give an equilibrium-theoretic characterization of level-optimal maximin pairs through a naturally associated continuum economy with measure-valued commodities. The measure-theoretic theory requires tools absent from the finite setting, including disintegration, measurable selection, measurable max-flow/min-cut duality, augmentation arguments, and a symmetric density decomposition separating absolutely continuous and singular components.

\end{abstract}


    \thispagestyle{empty}
\end{titlepage}

\section{Introduction}
\label{sec:intro}

A measurable bipartite relation with prescribed side measures gives rise to a large class of feasible refinement plans: finite plans concentrated on the relation and carrying one prescribed marginal. The question addressed here is whether this highly nonunique class nevertheless contains a mathematically distinguished class of refinements. We formulate the problem on standard Borel spaces and prove the main theorem package for closed bipartite relations between Polish spaces equipped with finite nonzero Borel measures. The governing structure is a genuinely one-sided maximin principle.

The first component is \emph{level-optimal maximin}, a levelwise extremal criterion formulated through truncation and overflow profiles. It is characterized by convex refinement: every minimizer of finite objective value for a strictly convex refinement criterion is level-optimal maximin, while every level-optimal maximin refinement minimizes every proper lower semicontinuous convex \(\vartheta\)-divergence. In particular, a single bounded strictly convex criterion selects exactly this class.

Second, proportional response reconstructs an opposite-side partner from the reverse conditional structure of a one-sided refinement and lifts a level-optimal maximin refinement to a universally closest pair, optimal among all feasible pairs for every divergence satisfying channel-DPI. Conversely, every universally closest pair has this structure.

Third, pairs that are level-optimal maximin on both sides admit an equilibrium-theoretic characterization through a naturally associated continuum economy with measure-valued commodities, without requiring the two refinements to be proportional responses. Thus a single measure-theoretic structure governs convex selection, universal comparison, reciprocal reconstruction, and equilibrium.

\subsection{The closest refinement problem}
\label{subsec:intro_canonical_refinement_problem}

Let \(\Omega_0\) and \(\Omega_1\) be standard Borel spaces, let
\(\mathcal E\subseteq\Omega_0\times\Omega_1\) be a measurable bipartite relation, and let
\(\nu_0,\nu_1\) be finite base measures. A refinement from side
\(\iota\in\{0,1\}\) is a finite plan \(\pi\) on \(\Omega_0\times\Omega_1\), concentrated on
\(\mathcal E\), whose \(\iota\)-marginal is \(\nu_\iota\); its opposite marginal
\(\nu_{\bar\iota}^{(\pi)}\) records the induced payload. Such refinements are typically highly nonunique, so the problem is not merely feasibility but the structure of the feasible class.

The central question is whether a one-sided criterion selects the relevant refinements and permits an associated paired refinement to be reconstructed. Although the paired problem appears to require simultaneous optimization over both components, it is controlled by a genuinely one-sided invariant. That invariant is level-optimal maximin: at each density level it compares how much admissible mass can be packed below that level on the opposite side, equivalently determining the optimal overflow profile. 
This levelwise criterion is connected to convex refinement in both directions, while proportional response lifts it to universal paired optimality for all channel-DPI divergences.

\subsection{Why a measure-theoretic theory is needed}
\label{subsec:intro_measure_theoretic_need}

The universal closest refinement problem was introduced in the finite-space theory of Chan and
Xue~\cite{DBLP:conf/innovations/ChanX25}, which established the finite universal-closest-refinement phenomenon and related it to pointwise local maximin structure, symmetric density decomposition, and a market-equilibrium formulation. The present paper develops the corresponding theory on standard Borel spaces under the closed-Polish hypotheses needed for attainment and structural analysis.

The finite setting is insufficient because the natural objects are measures on a bipartite relation and need not be supported on finite or countable sets. The measurable extension also requires genuinely new arguments: combinatorial decompositions are replaced by disintegration and measurable selection; projections may be analytic rather than Borel; singular payload components must be separated and tracked; and attainment requires compactness and lower-semicontinuity. Closedness of \(\mathcal E\) is essential in this last point, since minimizing sequences may otherwise converge to boundary concentrations that are infeasible for the original relation.

The organizing invariant changes as well. The finite density structure is exposed by repeatedly extracting densest pieces, a mechanism with no direct analogue on general standard Borel spaces. Here it is replaced by truncation levels, capped feasible mass, and overflow profiles, leading to level-optimal maximin as the invariant relevant to convex and universal comparison.

One motivation for this extension comes from differential obliviousness and its composition theory
\cite{DBLP:conf/eurocrypt/ZhouSCM23,DBLP:conf/innovations/ZhouZCS24}. There the privacy condition for neighboring inputs can be expressed through a refinement pair for the induced view--output laws, concentrated on an output-neighbor relation and controlled by a suitable divergence; a universally closest refinement then gives a single witness for all divergences satisfying channel-DPI. The finite theory already exhibits this mechanism, but the underlying laws need not be discrete, and modern divergence formalisms accommodate continuous probability models
\cite{dong2022gaussian}. This connection motivates the extension, but the mathematical theory developed here is driven by measurable transport, convex comparison, singularity structure, and duality.

\subsection{Main results}
\label{subsec:intro_main_results}

The correct one-sided extremal criterion is \emph{level-optimal maximin}, a levelwise condition formulated through truncation on the opposite side. It governs the subsequent convex and universal comparison theory.

If a proper lower semicontinuous integrand \(\vartheta\) is strictly convex on its effective domain, every minimizer of finite \(\mathsf D_\vartheta\)-value is level-optimal maximin. Conversely, every level-optimal maximin refinement minimizes \(\mathsf D_\vartheta\) for every proper lower semicontinuous convex \(\vartheta\). Taking
\(\vartheta(t)=e^{-t}\), for which the objective is finite on the entire feasible class, shows that a single bounded strictly convex refinement criterion has precisely the level-optimal maximin refinements as its minimizers.

Proportional response then identifies an opposite-side partner and leads to the paired universal comparison theorem.

\begin{theorem}[Universal closestness]
\label{thm:intro_universal_closestness}
Let \(\Omega_0,\Omega_1\) be Polish spaces equipped with finite nonzero Borel measures, and let
\(\mathcal E\subseteq\Omega_0\times\Omega_1\) be closed with nonempty neighborhoods. Then universal closest refinement pairs exist.

For a feasible pair \((\pi_\iota,\pi_{\bar\iota})\), the following are equivalent:
\begin{enumerate}[(i)]
\item \((\pi_\iota,\pi_{\bar\iota})\) is universally closest for every divergence satisfying channel-DPI;
\item \(\pi_\iota\) is level-optimal maximin from side \(\iota\) to side \(\bar\iota\), and
      \(\pi_{\bar\iota}\) is a proportional response to \(\pi_\iota\);
\item \(\pi_{\bar\iota}\) is level-optimal maximin from side \(\bar\iota\) to side \(\iota\), and
      \(\pi_\iota\) is a proportional response to \(\pi_{\bar\iota}\).
\end{enumerate}
In particular, every universal closest refinement pair is level-optimal maximin on both sides and consists of mutual proportional responses.
\end{theorem}

Thus universal closestness is not merely a consequence of the one-sided maximin structure and proportional response; it characterizes them. The theorem is not obtained by optimizing separately for different divergences. Level-optimal maximin first gives simultaneous optimality for all hockey-stick divergences through the overflow profile, and a binary randomization criterion upgrades this to every channel-DPI divergence. Equivalently, the comparison may be expressed in power-function or receiver operating characteristic order by the classical principles of Blackwell and Torgersen
\cite{blackwell1951comparison,blackwell1953equivalent,torgersen1991comparison}.

The same underlying structure also has an equilibrium-theoretic characterization. Refinement plans become allocations of measure-valued commodities, while \(\mathcal E\) determines which opposite-side commodities contribute to utility; prices and individual optimality encode the density structure. Proportional response is not part of this formulation.

\begin{theorem}[Equilibrium characterization]
\label{thm:intro_equilibrium_characterization}
Under the hypotheses of Theorem~\ref{thm:intro_universal_closestness}, every refinement pair whose two components are level-optimal maximin gives rise to a Walrasian equilibrium of the naturally associated continuum economy. Conversely, every Walrasian equilibrium of that economy yields, through the extraction procedure, a refinement pair whose two components are level-optimal maximin.
\end{theorem}

Theorem~\ref{thm:intro_equilibrium_characterization} is therefore a structural characterization rather than a canonical bijection between refinement pairs and equilibria. Together, the results identify level-optimal maximin as the governing one-sided structure, characterize it through convex refinement, lift it through proportional response to universal paired optimality, and recover its symmetric two-sided form through Walrasian equilibrium.

\subsection{Proof architecture and governing ideas}
\label{subsec:intro_proof_architecture}

The proof begins with two reductions. First, channel-DPI under projection reduces paired comparison to a one-sided problem for the opposite marginal; proportional response then reconstructs an optimal paired refinement from a one-sided optimizer; see Section~\ref{sec:paired_reduction}. Second, on fixed-mass classes, universal comparison reduces to hockey-stick divergences, whose values coincide with overflow tails and therefore interact directly with the levelwise theory.

For each level \(t\), one compares the largest feasible shipment with opposite marginal bounded by \(t\nu_{\bar\iota}\) against the truncated absolutely continuous payload of a given refinement. Equality for every \(t\) defines level-optimal maximin, equivalently pointwise minimization of the overflow profile; see Section~\ref{sec:maximin}. A measurable max-flow/min-cut theorem identifies the capped shipment values and supplies minimizing cuts, while an augmenting-subplan lemma replaces finite exchange arguments. These tools show that failure of level-optimal maximin permits a perturbation strictly improving a finite-valued strictly convex objective, whereas overflow domination yields arbitrary convex-divergence optimality through hinge representations and approximation.

Attainment is handled separately by a compactification-and-defect argument. Closedness of \(\mathcal E\) preserves feasibility of the interior weak limit, while boundary mass is charged at the recession rate and repaired by measurable selection; the appendix shows the obstruction when closedness fails.

At the paired level, strict convexity, adjoint divergence symmetry, and uniqueness of the absolutely continuous part identify proportional response in the converse direction. For pairs level-optimal maximin on both sides, the symmetric density decomposition separates absolutely continuous and singular components and yields the reciprocal density structure used in the continuum-economy price construction. Conversely, Walrasian optimality recovers the same structure, and max-flow/min-cut then yields level-optimal maximin.

\subsection{Related work}
\label{subsec:intro_related_work}

The direct precursor is the finite-space theory of universal closest refinement developed by Chan and Xue~\cite{DBLP:conf/innovations/ChanX25}, which establishes the finite universal-closest-refinement phenomenon together with pointwise local maximin structure, symmetric density decomposition, and a market-equilibrium characterization. The present paper develops the corresponding measure-theoretic theory on standard Borel spaces. The extension requires disintegration, measurable selection, analytic projections, treatment of singular components, attainment, and relation closedness, and replaces repeated extraction of densest pieces by one-sided reduction, truncation and overflow profiles, measurable augmentation, and convex-duality arguments.

The convex-refinement aspect is adjacent to the classical
\(I\)-projection geometry of Csisz\'ar
\cite{csiszar1975idivergence}, where a single divergence is minimized over a
convex family of probability measures. Section~\ref{sec:one_side} likewise studies a fixed convex divergence over a convex feasible family of opposite marginals, but the main theorem package identifies a structural class---level-optimal maximin---that is optimal for the full convex family and, at the paired level, for every divergence satisfying channel-DPI.

The universal comparison aspect is adjacent to the classical
comparison-of-experiments literature of Blackwell
\cite{blackwell1951comparison,blackwell1953equivalent}, Torgersen, and Le Cam
\cite{torgersen1991comparison,lecam1996comparison}. Our theorem is not a comparison-of-experiments result in that tradition, but uses the same comparison criteria: universal closestness may equivalently be expressed in power-function or ROC language. Infinite-state comparison requires additional care; a modern treatment on general Polish state spaces is given by Khan, Yu, and Zhang \cite{KhanYuZhang2024}. The present paper uses the classical binary randomization criterion on standard Borel spaces in the fixed-mass reduction of Proposition~\ref{prop:HS_implies_DPI}.

On the technical side, the paper draws on the measure-theoretic marginal and duality tradition surrounding Strassen, Edwards, and Kellerer
\cite{Strassen1965,Edwards1978Marginals,Kellerer_1984}, and on Lov\'asz's measurable flow framework \cite{Lovasz_2021}. These provide background for measurable couplings, marginal duality, and max-flow/min-cut principles, while the present problem concerns refinement selected by a maximin principle and its connection to convex and universal divergence comparison.

The equilibrium characterization belongs to the literature on continuum economies and nonatomic assignment models, beginning with Aumann and Hildenbrand
\cite{Aumann1964,Aumann1966,Hildenbrand1974} and continuing through
Gretsky--Ostroy--Zame and related work
\cite{GretskyOstroyZame1992,GretskyOstroyZame1999,OstroyZame1994,Podczeck1997,TOURKY2001189}.
Here the continuum economy is not a separate application: it gives an alternative characterization of level-optimal maximin through preferences, prices, and individual optimality.

\subsection{Organization of the paper}
\label{subsec:intro_organization}

Sections~\ref{sec:prelim} and \ref{sec:paired_reduction} set up the measurable refinement framework and the reductions to one-sided and hockey-stick comparison. Section~\ref{sec:maximin} introduces level-optimal maximin and the truncation--overflow viewpoint, Section~\ref{sec:measure_bipartite} develops measurable max-flow/min-cut and augmentation, and Section~\ref{sec:one_side} proves attainment and the convex characterization. Section~\ref{sec:universal_closest_pairs} establishes the universal closest theorem, its converse, and the symmetric density decomposition. Section~\ref{sec:continuum_agents_commodities} gives the continuum-economy characterization, and the appendix records the role of relation closedness in attainment.

\section{Preliminaries}
\label{sec:prelim}

This section fixes the measurable refinement framework and the comparison language used throughout.
We first specify the measurable bipartite relation model and the standing regularity assumptions.
We then introduce the one-sided refinement plans and their induced payload measures, which will
become the basic comparison objects after the paired problem is reduced to one side in
Section~\ref{sec:paired_reduction}.  Finally, we record the divergence language needed for the
universal optimality theorem.

Throughout, write \(\mathbb B := \{0,1\}\) and \(\bar\iota := 1-\iota\).
When convenient, we formulate statements for a generic side \(\iota\in\mathbb B\) and write
\(\bar\iota:=1-\iota\) for the opposite side. In proofs, we occasionally fix \(\iota=0\) by symmetry.

For a measurable space \((\Omega,\Sigma)\), we write \(\mathcal M_+(\Omega)\) for the cone of finite
nonnegative measures on~\((\Omega,\Sigma)\).
Given \(\mu,\nu\in\mathcal M_+(\Omega)\), we write
\[
\mu = \mu^{\mathrm{ac}} + \mu^\perp = r\,\nu + \mu^\perp,
\qquad
r := \frac{d\mu^{\mathrm{ac}}}{d\nu},
\qquad
\mu^\perp \perp \nu,
\]
for the Lebesgue decomposition of \(\mu\) with respect to \(\nu\).

%, and $\mathcal P(\Omega)$ for the set of probability
%measures on $(\Omega,\Sigma)$.

\subsection{Measurable Bipartite Relation Model}
\label{subsec:model_infinite_bipartite}

We begin by fixing the ambient measurable relation and the standing assumptions under which the
existence and structure theory will be developed.

\begin{definition}[Measurable Bipartite Relation]
\label{defn:measurable_bipartite_graph}
For \(\iota\in\mathbb B\), let \((\Omega_\iota,\Sigma_\iota)\) be a measurable space.
A \emph{(measurable) bipartite relation} between the two vertex spaces \(\Omega_0\) and \(\Omega_1\)
is specified by a measurable edge relation \(\mathcal E\subseteq \Omega_0\times\Omega_1\),
measurable with respect to \(\Sigma_0\otimes\Sigma_1\). For \(x\in\Omega_0\) and \(y\in\Omega_1\), define
\[
  N_1(x):=\{\,y\in\Omega_1:(x,y)\in\mathcal E\,\},
  \qquad
  N_0(y):=\{\,x\in\Omega_0:(x,y)\in\mathcal E\,\}.
\]

For \(\iota \in \B\), \(S \subseteq \Omega_\iota\), denote
\(N_{\bar \iota}(S) := \cup_{s \in S} N_{\bar \iota}(s)\);
denote \(\mathcal{E}_\iota := \{(x_\iota, x_{\bar \iota}) \in \Omega_\iota \times \Omega_{\bar \iota}: (x_0, x_1) \in \mathcal{E}\}\).
\end{definition}

\begin{definition}[Finite Vertex Measures]
\label{defn:vertex_measures}
For each \(\iota\in\mathbb B\), let \(\nu_\iota\) be a finite measure on \((\Omega_\iota,\Sigma_\iota)\)
such that \(\nu_\iota(\Omega_\iota)>0\).
\end{definition}

\begin{assumption}[Closed Polish Relation With Nonempty Neighborhoods]
\label{assump:closed_polish_graph_nonempty_nbhd}
Throughout, we assume:
\begin{enumerate}[(i)]
\item \textbf{Polish vertex spaces.}
For each \(\iota\in\mathbb B\), \(\Omega_\iota\) is a Polish space equipped with its Borel
\mbox{\(\sigma\)-algebra~\(\Sigma_\iota\).}

\item \textbf{Closed edge relation.}
The edge set \(\mathcal E\subseteq \Omega_0\times\Omega_1\) is closed.

We record in Appendix~\ref{app:closed-support-needed} an example illustrating the role of the closedness assumption on \(\mathcal E\).

\item \textbf{Nonempty neighborhoods.}
For each \(\iota\in\mathbb B\), one has
\[
  N_{\bar\iota}(x)\neq\emptyset
  \qquad\text{for all }x\in\Omega_\iota .
\]
For the arguments below, it would in fact suffice to assume this only for \(\nu_\iota\)-a.e.\ \(x\).
\end{enumerate}
\end{assumption}

\subsection{Allocation Refinements and Induced Payloads}
\label{subsec:refinements_payloads}

We now introduce the basic one-sided refinement object and the opposite-side payload measure it
induces.  This payload will later become the true comparison variable in the one-sided reduction.

For a measurable map \(T:\mathcal{X}\to\mathcal{Y}\) and a measure \(\alpha\) on \(\mathcal{X}\), let
\((T)_\#\alpha\) denote the pushforward measure on \(\mathcal{Y}\), defined by
\(((T)_\#\alpha)(B)=\alpha(T^{-1}(B))\) for measurable \(B \subseteq \mathcal{Y}\).

\begin{definition}[Allocation Refinement as a Plan]
\label{defn:allocation_refinement_plan}
Fix \(\iota\in\mathbb B\). The collection of allocation refinements of \(\nu_\iota\) is
\[
  \Pi^{\mathcal E}_\iota(\nu_\iota)
  :=
  \Bigl\{
    \pi\in\mathcal M_+(\Omega_0\times\Omega_1)
    :
    \pi(\mathcal E^\c)=0,\ 
    (\pr_\iota)_\#\pi=\nu_\iota
  \Bigr\}.
\]
\end{definition}

Thus \(\pi\in\Pi^{\mathcal E}_\iota(\nu_\iota)\) distributes the mass of \(\nu_\iota\) across admissible edges while keeping the \(\iota\)-marginal fixed.
When there is no risk of ambiguity, we write
\(\Pi^{\mathcal E}_\iota:= \Pi^{\mathcal E}_\iota(\nu_\iota)\).  For paired refinements, we write
$
  \Pi^{\mathcal E}(\nu_0,\nu_1)
  :=
  \Pi_0^{\mathcal E}(\nu_0)
  \times
  \Pi_1^{\mathcal E}(\nu_1).
$
%and \(\Pi^{\mathcal E}_{\mathrm{pair}} := \Pi^{\mathcal E}_0 \times \Pi^{\mathcal E}_1\).

\begin{definition}[Payload Measure (Plan Form)]
\label{defn:payload}
Fix \(\iota\in\mathbb B\) and let \(\pi\in\Pi^{\mathcal E}_\iota(\nu_\iota)\).
The induced payload measure on the opposite side is the other marginal
\[
  \nu^{(\pi)}_{\bar\iota}
  :=
  (\pr_{\bar\iota})_\#\pi .
\]

Equivalently,
\[
  \nu^{(\pi)}_{\bar\iota}(B)
  =
  \pi\bigl(\pr_{\bar\iota}^{-1}(B)\bigr),
  \qquad
  B\in\Sigma_{\bar\iota}.
\]
\end{definition}

Since \((\Omega_0,\Sigma_0)\) and \((\Omega_1,\Sigma_1)\) are standard Borel, disintegration will be
available whenever needed. We defer the corresponding measure-theoretic statements to the technical
infrastructure section.

\subsection{Divergences: Discrepancy Measures Between Finite Measures}
\label{subsec:discrepancy_measures_divergences_power}

We record here the comparison language used later in the universal optimality theorem.  At this
stage we need only the abstract divergence formalism and the monotonicity property that will govern
universal comparison after the reduction to one-sided payload measures.

\begin{definition}[Divergence Between Finite Measures]
\label{defn:divergence}
A \emph{divergence} \(\D\) is a family of maps
\[
  \D_{(\Omega,\Sigma)}:\mathcal M_+(\Omega)\times\mathcal M_+(\Omega)\to\overline{\mathbb R},
  \qquad
  \overline{\mathbb R}:=\mathbb R\cup\{+\infty,-\infty\},
\]
indexed by measurable spaces \((\Omega,\Sigma)\).
When the underlying space is clear, we write \(\D(P\|Q)\) for
\(\D_{(\Omega,\Sigma)}(P,Q)\).
\end{definition}

\begin{definition}[Data-Processing Inequality (Channel-DPI)]
\label{defn:DPI_divergence}
We say that a divergence \(\D\) satisfies the \emph{data-processing inequality} if for every
pair of measurable spaces \((\Omega,\Sigma)\) and \((\Omega',\Sigma')\), every Markov kernel
\(K:(\Omega,\Sigma)\rightsquigarrow(\Omega',\Sigma')\), and all \(P,Q\in\mathcal M_+(\Omega)\), one has
\[
  \D_{(\Omega',\Sigma')}(K_\#P\|K_\#Q)
  \le
  \D_{(\Omega,\Sigma)}(P\|Q),
\]
where \((K_\#P)(B):=\int_\Omega K(x,B)\,P(dx)\) for \(B\in\Sigma'\).
We write \(\mathfrak D_{\mathrm{DPI}}\) for the class of all divergences satisfying channel-DPI.
\end{definition}

\section{Paired-to-One-Sided Reduction and Restriction of the Divergence Class}
\label{sec:paired_reduction}

This section performs the two reductions that organize the rest of the paper. We first show that the paired refinement problem is controlled entirely by a one-sided comparison problem on the opposite marginal. We then reduce universal comparison over all channel-DPI divergences to a smaller comparison class, namely hockey-stick divergences on fixed-mass families. These two reductions explain why the later sections may focus on the one-sided extremal criterion and on truncation and hockey-stick comparison.

Fix the measurable bipartite relation \(\mathcal E\) and side measures \(\nu_0,\nu_1\) as in
Section~\ref{sec:prelim}.

\begin{definition}[$\mathsf D$-Closest Paired Refinement Problem]
\label{defn:D_closest_paired_refinement_problem}
Let \(\mathsf D\) be a divergence between finite measures; fix \(\iota \in \B\).

The \emph{\(\mathsf D\)-closest paired refinement problem} is
\begin{equation} \label{eq:paired_closest}
  \inf\bigl\{
    \mathsf D(\pi_\iota\|\pi_{\bar \iota})
    :
    (\pi_\iota,\pi_{\bar \iota})\in \Pairi
  \bigr\}.
\end{equation}
A pair \((\pi_\iota^\star,\pi_{\bar\iota}^\star)\in \Pairi\) is called a \emph{\(\mathsf D\)-closest refinement pair} if it attains this infimum.
\end{definition}

\begin{definition}[Universal Closest Paired Refinement Problem]
\label{defn:universal_closest_paired_refinement_problem}
The \emph{universal closest paired refinement problem} is to find a pair \((\pi_\iota^\star,\pi_{\bar\iota}^\star)\in \Pairi\) such that, for every \(\mathsf D\in\mathfrak D_{\mathrm{DPI}}\), the pair \((\pi_\iota^\star,\pi_{\bar\iota}^\star)\) is a \(\mathsf D\)-closest refinement pair.
Any such pair is called a \emph{universal closest refinement pair}.
\end{definition}

\subsection{Reduction from Paired to One-Sided Refinement}
\label{sec:reduction_paired}

We now pass from the paired formulation to the one-sided problem that will govern the rest of the
theory.  Although universal optimality is naturally stated for refinement pairs, the reduction below
shows that the relevant comparison is already determined by the opposite marginal of a single
one-sided refinement.  Proportional response then serves as the reconstruction mechanism that turns
a one-sided optimizer back into a paired one.

\begin{definition}[$\mathsf D$-Closest One-Sided Refinement Problem]
\label{defn:D_closest_one_sided_refinement_problem}
Fix \(\iota\in\mathbb B\), and let \(\mathsf D\) be a divergence between finite measures.

The \emph{\(\mathsf D\)-closest one-sided refinement problem on side \(\iota\)} is
\begin{equation} \label{eq:one_closest}
  \inf\bigl\{
    \mathsf D\bigl(\nu^{(\pi)}_{\bar\iota}\,\|\,\nu_{\bar\iota}\bigr)
    :
    \pi\in \Pi^{\mathcal E}_\iota
  \bigr\},
\end{equation}
where \(\nu^{(\pi)}_{\bar\iota}\) is the \({\bar \iota}\)-marginal
of \(\pi\) as in Definition~\ref{defn:payload}.

A plan \(\pi^\star\in \Pi^{\mathcal E}_\iota\) is called a \emph{\(\mathsf D\)-closest one-sided refinement on side \(\iota\)} if it attains this infimum.
\end{definition}

To formulate the reconstruction step precisely, we record the kernel notation and the reverse
disintegration formalism used in the definition of proportional response.

\paragraph{Kernel product.}
If \(\mu\in\Mplus(\Omega_0)\) and \(K:\Omega_0\times\Sigma_1\to[0,\infty)\) is a kernel, write
\(\mu\kprod K\in\Mplus(\Omega_0\times\Omega_1)\) for the finite measure determined by
\[
  (\mu\kprod K)(A\times B):=\int_A K(x,B)\,\mu(dx),
  \qquad A\in\Sigma_0,\ B\in\Sigma_1.
\]
Likewise, if \(\eta\in\Mplus(\Omega_1)\) and \(L:\Omega_1\times\Sigma_0\to[0,\infty)\) is a kernel, write
\(\eta\kprod L\) for the finite measure on \(\Omega_0\times\Omega_1\) determined by
\[
  (\eta\kprod L)(A\times B):=\int_B L(y,A)\,\eta(dy),
  \qquad A\in\Sigma_0,\ B\in\Sigma_1.
\]

\begin{definition}[Disintegration Kernel]
\label{defn:disintegration}
Assume \((\Omega_0,\Sigma_0)\) and \((\Omega_1,\Sigma_1)\) are standard Borel measurable spaces, and let
\(\pi\) be a finite measure on \((\Omega_0\times\Omega_1,\Sigma_0\otimes\Sigma_1)\).
Fix \(\iota\in\{0,1\}\), write \(\bar\iota:=1-\iota\), and let
\(\nu^{(\pi)}_\iota := (\pr_\iota)_\#\pi\) be the \(\iota\)-marginal of \(\pi\).

Let \(\mu_\iota\) be a \(\sigma\)-finite measure on \((\Omega_\iota,\Sigma_\iota)\) such that
\(\nu^{(\pi)}_\iota\ll \mu_\iota\).
A \emph{(regular) disintegration of \(\pi\) with respect to \(\mu_\iota\)} is a kernel
\[
  \Gamma^{(\pi)}_{\bar\iota}:\Omega_\iota\times\Sigma_{\bar\iota}\to[0,\infty)
\]
such that:
\begin{enumerate}\setlength{\itemsep}{0.2em}
\item for each \(x\in\Omega_\iota\), \(B\mapsto \Gamma^{(\pi)}_{\bar\iota}(x,B)\) is a finite measure on
\((\Omega_{\bar\iota},\Sigma_{\bar\iota})\);
\item for each \(B\in\Sigma_{\bar\iota}\), \(x\mapsto \Gamma^{(\pi)}_{\bar\iota}(x,B)\) is \(\Sigma_\iota\)-measurable;
\item \(\pi=\mu_\iota\kprod \Gamma^{(\pi)}_{\bar\iota}\).
\end{enumerate}
\end{definition}

\begin{remark}[Disintegration on standard Borel spaces]
\label{rem:disintegration_exists}
The existence and \(\mu_\iota\)-a.e.\ uniqueness of disintegrations in the sense of
Definition~\ref{defn:disintegration} are standard consequences of the disintegration theorem
on standard Borel spaces; see, for example,
Kallenberg~\cite[Chapter~7, ``Conditioning and Disintegration'']{Kallenberg2021Foundations}.
\end{remark}

\begin{definition}[Proportional Response to Refinement Plan]
\label{defn:proportional_response_plan}
Fix \(\iota\in\mathbb B\), write \(\bar\iota:=1-\iota\), and let
\(\pi\in\Pi^{\mathcal E}_\iota\). Set
\[
  P:=(\pr_{\bar\iota})_\#\pi,
\]
and let \(\K^{(\pi)}_\iota\) be a disintegration of \(\pi\) with
respect to \(P\) in the sense of Definition~\ref{defn:disintegration};
thus
\[
  \pi=P\kprod\K^{(\pi)}_\iota.
\]
This kernel is unique \(P\)-a.e. and, since
\(\pi(\mathcal E^\c)=0\), satisfies
\[
  \K^{(\pi)}_\iota\bigl(y,N_\iota(y)\bigr)=1
  \qquad
  \text{for \(P\)-a.e.\ }y.
\]

Write the mutual Lebesgue decompositions as
\[
  \nu_{\bar\iota}
  =
  (\nu_{\bar\iota})_{\mathrm{ac}}
  +
  (\nu_{\bar\iota})_{\perp},
  \qquad
  P=P_{\mathrm{ac}}+P_{\perp},
\]
where \((\nu_{\bar\iota})_{\perp}\),
\((\nu_{\bar\iota})_{\mathrm{ac}}\sim P_{\mathrm{ac}}\), and
\(P_{\perp}\) are mutually singular.

By measurable selection and passage to an
\((\nu_{\bar\iota})_{\perp}\)-a.e.\ equal Borel version, there exists
a measurable probability kernel
\[
  \K^{\mathrm{fb}}_\iota:
  \Omega_{\bar\iota}\times\Sigma_\iota\to[0,1]
\]
such that
\[
  \K^{\mathrm{fb}}_\iota\bigl(y,N_\iota(y)\bigr)=1
  \qquad
  \text{for }(\nu_{\bar\iota})_{\perp}\text{-a.e.\ }y.
\]
Fix any such kernel and define
\[
  \PR_{\bar\iota}(\pi)
  :=
  (\nu_{\bar\iota})_{\mathrm{ac}}
    \kprod\K^{(\pi)}_\iota
  +
  (\nu_{\bar\iota})_{\perp}
    \kprod\K^{\mathrm{fb}}_\iota.
\]
Then
\[
  \PR_{\bar\iota}(\pi)\in\Pi^{\mathcal E}_{\bar\iota},
  \qquad
  (\pr_{\bar\iota})_\#\PR_{\bar\iota}(\pi)
  =
  \nu_{\bar\iota}.
\]
\end{definition}

\begin{lemma}[Reduction of the paired problem to the one-sided problem]
\label{lem:paired_to_one_sided_reduction}
Fix \(\iota\in\mathbb B\), write \(\bar\iota:=1-\iota\), and let \(\mathsf D\) be a divergence
between finite measures satisfying channel-DPI.

Then
\[
  \inf\bigl\{
    \mathsf D(\pi_\iota\|\pi_{\bar\iota})
    :
    (\pi_\iota,\pi_{\bar\iota})\in \Pairi
  \bigr\}
  =
  \inf\bigl\{
    \mathsf D\bigl(\nu^{(\pi)}_{\bar\iota}\,\|\,\nu_{\bar\iota}\bigr)
    :
    \pi\in \Pi^{\mathcal E}_\iota
  \bigr\}.
\]

Moreover:
\begin{enumerate}
\item if \(\pi^\star\in\Pi^{\mathcal E}_\iota\) is a \(\mathsf D\)-closest one-sided refinement on side \(\iota\),
then \((\pi^\star,\PR_{\bar\iota}(\pi^\star))\) is a \(\mathsf D\)-closest refinement pair;

\item if \((\pi_\iota^\star,\pi_{\bar\iota}^\star)\in \Pairi\) is a
\(\mathsf D\)-closest refinement pair, then \(\pi_\iota^\star\) is a \(\mathsf D\)-closest one-sided refinement on side \(\iota\).
\end{enumerate}
\end{lemma}

\begin{proof}
Write
\[
  I_{\mathrm{pair}}
  :=
  \inf\bigl\{
    \mathsf D(\pi_\iota\|\pi_{\bar\iota})
    :
    (\pi_\iota,\pi_{\bar\iota})\in \Pairi
  \bigr\},
  \qquad
  I_{\mathrm{one}}
  :=
  \inf\bigl\{
    \mathsf D\bigl(\nu^{(\pi)}_{\bar\iota}\,\|\,\nu_{\bar\iota}\bigr)
    :
    \pi\in \Pi^{\mathcal E}_\iota
  \bigr\}.
\]

\paragraph{Projection bound.}
The inequality \(I_{\mathrm{one}}\le I_{\mathrm{pair}}\) is immediate from channel-DPI applied to the
projection \(\pr_{\bar\iota}\). Indeed, if \((\pi_\iota,\pi_{\bar\iota})\in\Pairi\), then
\((\pr_{\bar\iota})_\#\pi_\iota=\nu^{(\pi_\iota)}_{\bar\iota}\) and
\((\pr_{\bar\iota})_\#\pi_{\bar\iota}=\nu_{\bar\iota}\), so
\[
  \mathsf D\bigl(\nu^{(\pi_\iota)}_{\bar\iota}\,\|\,\nu_{\bar\iota}\bigr)
  \le
  \mathsf D(\pi_\iota\|\pi_{\bar\iota}).
\]
Taking the infimum over \((\pi_\iota,\pi_{\bar\iota})\in\Pairi\) gives \(I_{\mathrm{one}}\le I_{\mathrm{pair}}\).

\paragraph{Proportional-response comparison.}
For the reverse inequality, fix \(\pi\in\Pi^{\mathcal E}_\iota\) and set
\[
  P:=\nu^{(\pi)}_{\bar\iota}
  =
  (\pr_{\bar\iota})_\#\pi.
\]
Choose the disintegration \(\K^{(\pi)}_\iota\) and fallback kernel
\(\K^{\mathrm{fb}}_\iota\) from
Definition~\ref{defn:proportional_response_plan}. After modifying
\(\K^{(\pi)}_\iota\) on a \(P\)-null set, which changes neither
\(P\kprod\K^{(\pi)}_\iota\) nor
\((\nu_{\bar\iota})_{\mathrm{ac}}\kprod\K^{(\pi)}_\iota\), we may
assume that \(\K^{(\pi)}_\iota\) is a probability kernel everywhere.

Since \((\nu_{\bar\iota})_{\perp}\perp P\), there exists
\(Z\in\Sigma_{\bar\iota}\) such that
\[
  P(Z)=0,
  \qquad
  (\nu_{\bar\iota})_{\perp}(Z^\c)=0.
\]
Redefine \(\K^{(\pi)}_\iota\) on \(Z\) by
\[
  \widehat{\K}(y,\cdot)
  :=
  \begin{cases}
    \K^{\mathrm{fb}}_\iota(y,\cdot), & y\in Z,\\
    \K^{(\pi)}_\iota(y,\cdot), & y\notin Z.
  \end{cases}
\]

Since \(P(Z)=0\), this modification does not change the disintegration of \(\pi\), and hence
\[
  \pi = P \kprod \widehat{\K}.
\]
On the other hand, \((\nu_{\bar\iota})_{\mathrm{ac}}\ll P\), so
\((\nu_{\bar\iota})_{\mathrm{ac}}(Z)=0\), while \((\nu_{\bar\iota})_{\perp}\) is carried by \(Z\).
Therefore
\[
  \nu_{\bar\iota}\kprod \widehat{\K}
  =
  (\nu_{\bar\iota})_{\mathrm{ac}}\kprod \K^{(\pi)}_\iota
  +
  (\nu_{\bar\iota})_{\perp}\kprod \K^{\mathrm{fb}}_\iota
  =
  \PR_{\bar\iota}(\pi).
\]

Let \(T\) denote the Markov kernel from \(\Omega_{\bar\iota}\) to
\(\Omega_0\times\Omega_1\) associated with \(\widehat{\K}\). Then
\(T_\#P=\pi\) and \(T_\#\nu_{\bar\iota}=\PR_{\bar\iota}(\pi)\). By channel-DPI,
\[
  \mathsf D\bigl(\pi\,\|\,\PR_{\bar\iota}(\pi)\bigr)
  \le
  \mathsf D(P\|\nu_{\bar\iota})
  =
  \mathsf D\bigl(\nu^{(\pi)}_{\bar\iota}\,\|\,\nu_{\bar\iota}\bigr).
\]
Since \(\PR_{\bar\iota}(\pi)\in\Pi^{\mathcal E}_{\bar\iota}\), this yields
\[
  I_{\mathrm{pair}}
  \le
  \mathsf D\bigl(\pi\,\|\,\PR_{\bar\iota}(\pi)\bigr)
  \le
  \mathsf D\bigl(\nu^{(\pi)}_{\bar\iota}\,\|\,\nu_{\bar\iota}\bigr).
\]
Taking the infimum over \(\pi\in\Pi^{\mathcal E}_\iota\) gives \(I_{\mathrm{pair}}\le I_{\mathrm{one}}\).
Hence \(I_{\mathrm{pair}}=I_{\mathrm{one}}\).

\paragraph{Attainment.}
For the first claim, let \(\pi^\star\in\Pi^{\mathcal E}_\iota\) be a \(\mathsf D\)-closest one-sided
refinement on side \(\iota\). Then
\[
  I_{\mathrm{pair}}
  \le
  \mathsf D\bigl(\pi^\star\|\PR_{\bar\iota}(\pi^\star)\bigr)
  \le
  \mathsf D\bigl(\nu^{(\pi^\star)}_{\bar\iota}\,\|\,\nu_{\bar\iota}\bigr)
  =
  I_{\mathrm{one}}
  =
  I_{\mathrm{pair}},
\]
so \((\pi^\star,\PR_{\bar\iota}(\pi^\star))\) attains the paired infimum.

For the second claim, let \((\pi_\iota^\star,\pi_{\bar\iota}^\star)\in\Pairi\) attain the paired
infimum. By the projection argument from the first paragraph,
\[
  \mathsf D\bigl(\nu^{(\pi_\iota^\star)}_{\bar\iota}\,\|\,\nu_{\bar\iota}\bigr)
  \le
  \mathsf D(\pi_\iota^\star\|\pi_{\bar\iota}^\star)
  =
  I_{\mathrm{pair}}
  =
  I_{\mathrm{one}}.
\]
Since \(I_{\mathrm{one}}\) is the infimum of the one-sided objective, equality must hold, and
\(\pi_\iota^\star\) is a \(\mathsf D\)-closest one-sided refinement on side \(\iota\).
\end{proof}

Lemma~\ref{lem:paired_to_one_sided_reduction} identifies the true core of the theory.  From this
point onward, paired optimality may be studied entirely through the one-sided opposite marginal.
The remaining problem is therefore to determine which one-sided refinements minimize the relevant
comparison functionals, a question that will be addressed through level-optimal maximin in
Section~\ref{sec:maximin}.

\subsection{Reduction from DPI-Divergences to a Restricted Class}
\label{sec:reduction_HS}

After the paired problem has been reduced to one-sided payload measures, the comparison class is
still too broad if left in the form of all channel-DPI divergences.  The next step is therefore to
reduce universal comparison to the hockey-stick subclass.  This is the form in which the comparison
theory will later connect to truncation, overflow, and convex-duality arguments.

We first record the convex divergence family on finite measures, then isolate the hockey-stick
subclass, and finally state the reduction from hockey-stick universality to
\(\mathfrak D_{\mathrm{DPI}}\)-universality on fixed-mass classes.

\begin{definition}[$\vartheta$-Divergence]
\label{defn:div_measure}
Let \((\Omega,\Sigma)\) be a measurable space, let \(P,Q\in\Mplus(\Omega)\), and let
\(\vartheta:[0,\infty)\to \mathbb R\cup\{+\infty\}\) be convex and proper.

Write the Lebesgue decomposition of \(P\) with respect to \(Q\) as
\(P=P_{\mathrm{ac}}+P_\perp\), where \(P_{\mathrm{ac}}\ll Q\) and \(P_\perp\perp Q\), and let
\(r:=dP_{\mathrm{ac}}/dQ\).
Define the recession slope
\[
  \vartheta'_\infty:=\lim_{t\to\infty}\frac{\vartheta(t)}{t}\in\mathbb R\cup\{+\infty\}.
\]
The \emph{\(\vartheta\)-divergence} of \(P\) from \(Q\) is
\[
  \D_\vartheta(P\|Q)
  :=
  \int_\Omega \vartheta(r)\,dQ
  +
  \vartheta'_\infty\,P_\perp(\Omega),
\]
with the convention that \(\D_\vartheta(P\|Q)=+\infty\) if \(\vartheta'_\infty=+\infty\) and
\(P_\perp(\Omega)>0\).
\end{definition}

The data-processing inequality for \(\vartheta\)-divergences under measurable
channels is classical; see Csisz\'ar~\cite{Csiszar1967}.
The finite-measure formulation above follows from the
probability-measure version by normalization and rescaling.

\begin{definition}[Hockey-Stick Divergences]
\label{defn:hockey_stick_divergence}
For \(\gamma>0\), define the convex function
\[
  \vartheta_\gamma(t):=(t-\gamma)_+,\qquad t\ge 0.
\]
The corresponding \emph{hockey-stick divergence} is
\[
  \HS_\gamma(P\|Q):=\D_{\vartheta_\gamma}(P\|Q),
  \qquad P,Q\in\Mplus(\Omega).
\]
We write
\[
  \mathfrak D_{\mathrm{HS}}:=\{\HS_\gamma:\gamma>0\}
\]
for the hockey-stick subclass of \(\vartheta\)-divergences.
\end{definition}

For probability measures, hockey-stick domination is equivalent to
domination in the binary testing, or ROC, order. Indeed,
\[
  \HS_\eta(P\|Q)
  =
  \sup_{0\le \phi\le 1}
  \left\{
    \int \phi\,dP
    -
    \eta\int \phi\,dQ
  \right\},
  \qquad \eta>0,
\]
so the family of hockey-stick inequalities determines the binary
testing region through its supporting lines. By Blackwell's binary
randomization theorem, domination of this testing region is equivalent
to measurable randomized post-processing; see
Dong--Roth--Su~\cite[Theorem~2]{dong2022gaussian}, and compare
Blackwell~\cite{blackwell1951comparison,blackwell1953equivalent} and
Torgersen~\cite{torgersen1991comparison}.

For convenience, we record the fixed-mass finite-measure consequence
needed here. The reduction is by normalization and threshold rescaling,
followed by the probability-measure randomization criterion and
channel-DPI.

\begin{proposition}[Hockey-stick universality implies \(\mathfrak D_{\mathrm{DPI}}\)-universality]
\label{prop:HS_implies_DPI}
Let \((\Omega,\Sigma)\) be a standard Borel space, and let
\(\mathcal A,\mathcal B\subseteq \Mplus(\Omega)\) be collections of finite measures on \((\Omega,\Sigma)\).
Assume that there exist constants \(a,b>0\) such that
\(x(\Omega)=a\) for all \(x\in\mathcal A\) and \(y(\Omega)=b\) for all \(y\in\mathcal B\).

Suppose that \((x^\star,y^\star)\in\mathcal A\times\mathcal B\) is universally closest with respect
to hockey-stick divergences, in the sense that for every \(\gamma>0\),
\[
  \HS_\gamma(x^\star\|y^\star)
  =
  \inf\bigl\{
    \HS_\gamma(x\|y)
    :
    (x,y)\in\mathcal A\times\mathcal B
  \bigr\}.
\]
Then \((x^\star,y^\star)\) is universally closest with respect to
\(\mathfrak D_{\mathrm{DPI}}\), i.e. for every \(\mathsf D\in\mathfrak D_{\mathrm{DPI}}\),
\[
  \mathsf D(x^\star\|y^\star)
  =
  \inf\bigl\{
    \mathsf D(x\|y)
    :
    (x,y)\in\mathcal A\times\mathcal B
  \bigr\}.
\]
\end{proposition}

\begin{proof}
Write
\[
  \mathcal A_1:=\{x/a:x\in\mathcal A\},
  \qquad
  \mathcal B_1:=\{y/b:y\in\mathcal B\},
\]
so that \(\mathcal A_1\) and \(\mathcal B_1\) are collections of probability measures on \((\Omega,\Sigma)\).
Set \(P^\star:=x^\star/a\) and \(Q^\star:=y^\star/b\).

For \(x=aP\) and \(y=bQ\) with \(P,Q\) probability measures,
\[
  \HS_\gamma(x\|y)
  = a\,\HS_{\gamma b/a}(P\|Q),
  \qquad \gamma>0.
\]
Hence the hypothesis implies that for every \(\eta>0\),
\[
  \HS_\eta(P^\star\|Q^\star)
  =
  \inf\bigl\{
    \HS_\eta(P\|Q)
    :
    P\in\mathcal A_1,\ Q\in\mathcal B_1
  \bigr\},
\]
by taking \(\gamma=a\eta/b\).

Now fix \(\mathsf D\in\mathfrak D_{\mathrm{DPI}}\), and define a divergence on probability measures by
\[
  \widetilde{\mathsf D}(P\|Q):=\mathsf D(aP\|bQ).
\]
Since pushforward commutes with multiplication by a positive scalar,
\(\widetilde{\mathsf D}\) again satisfies channel-DPI.

Let \(P\in\mathcal A_1\) and \(Q\in\mathcal B_1\). By the binary randomization criterion for
comparison of experiments
(\cite{blackwell1951comparison,blackwell1953equivalent,torgersen1991comparison}; see also
\cite{dong2022gaussian}),
the inequalities
\[
  \HS_\eta(P^\star\|Q^\star)\le \HS_\eta(P\|Q)
  \qquad\text{for all }\eta>0
\]
imply that there exists a Markov kernel \(K\) such that
\[
  P^\star = K_\# P,
  \qquad
  Q^\star = K_\# Q.
\]
Applying channel-DPI for \(\widetilde{\mathsf D}\) yields
\[
  \widetilde{\mathsf D}(P^\star\|Q^\star)
  \le
  \widetilde{\mathsf D}(P\|Q).
\]
Equivalently,
\[
  \mathsf D(x^\star\|y^\star)\le \mathsf D(x\|y)
\]
for every \(x=aP\in\mathcal A\) and \(y=bQ\in\mathcal B\). Therefore
\[
  \mathsf D(x^\star\|y^\star)
  \le
  \inf\bigl\{
    \mathsf D(x\|y)
    :
    (x,y)\in\mathcal A\times\mathcal B
  \bigr\}.
\]
Since \((x^\star,y^\star)\in\mathcal A\times\mathcal B\), the reverse inequality is immediate. Hence
\[
  \mathsf D(x^\star\|y^\star)
  =
  \inf\bigl\{
    \mathsf D(x\|y)
    :
    (x,y)\in\mathcal A\times\mathcal B
  \bigr\}.
\]
Thus \((x^\star,y^\star)\) is universally closest with respect to \(\mathfrak D_{\mathrm{DPI}}\).
\end{proof}

Proposition~\ref{prop:HS_implies_DPI} is what allows the later universal theorem to be proved through a
restricted comparison family rather than through the full class of channel-DPI divergences one by
one.  In particular, once one-sided optimality has been identified at the hockey-stick level, the
passage back to \(\mathfrak D_{\mathrm{DPI}}\)-universality is automatic.

\section{Level-Optimal Maximin as the One-Sided Extremal Criterion}
\label{sec:maximin}

By Lemma~\ref{lem:paired_to_one_sided_reduction}, the paired problem is already
reduced to a one-sided comparison problem on opposite marginals. The remaining
task is therefore to identify the governing one-sided invariant. The appropriate
formulation is levelwise: at each density cap \(t\), one asks how much admissible
mass can be packed below that level. This leads to level-optimal maximin, which
will later be related in Section~\ref{sec:one_side} to the variational structure of
convex refinement.

Before stating the definition, we isolate the quantities that must be compared at each level.
Given a refinement \(\pi\), its opposite marginal determines, for every \(t\ge 0\), a truncated
payload \(P_t^{(\pi)}\).  On the other hand, the feasible class \(\Feas_{\bar\iota}(t)\) records all
subplans whose shipment remains below the same density cap \(t\,\nu_{\bar\iota}\), and
\(\Fit_{\bar\iota}(t)\) is the largest total mass attainable under that constraint.  
The defining condition is that these two quantities agree at every level.

\begin{definition}[Level-optimal maximin]
\label{defn:level_opt_maximin}
Fix \(\iota\in\mathbb B\), and write \(\bar\iota:=1-\iota\).
Let \(\pi\in\Pi^{\mathcal E}_\iota\), and write
\[
  P:=\nu^{(\pi)}_{\bar\iota}=(\pr_{\bar\iota})_\#\pi \in \mathcal M_+(\Omega_{\bar\iota})
\]
for its opposite marginal.

Write the Lebesgue decomposition of \(P\) with respect to \(\nu_{\bar\iota}\) as
\[
  P = r\,\nu_{\bar\iota} + P^\perp,
  \qquad
  P^\perp\perp \nu_{\bar\iota}.
\]
For \(t\ge 0\), define the truncated absolutely continuous payload at level \(t\) by
\[
  P_t^{(\pi)} := (r\wedge t)\,\nu_{\bar\iota}\in\mathcal M_+(\Omega_{\bar\iota}).
\]

Define
\begin{equation} \label{eq:Feas_t}
  \Feas_{\bar\iota}(t)
  :=
  \Bigl\{
    \sigma\in\mathcal M_+(\Omega_0\times\Omega_1)
    :
    \sigma(\mathcal E^\c)=0,\ 
    (\pr_\iota)_\#\sigma \le \nu_\iota,\ 
    (\pr_{\bar\iota})_\#\sigma \le t\,\nu_{\bar\iota}
  \Bigr\},
\end{equation}
and
\begin{equation} \label{eq:Fit_t}
  \Fit_{\bar\iota}(t)
  :=
  \sup\bigl\{
    \sigma(\Omega_0\times\Omega_1)
    :
    \sigma\in\Feas_{\bar\iota}(t)
  \bigr\}.
\end{equation}

We say that \(\pi\) is \emph{level-optimal maximin} if, for every \(t>0\),
\begin{equation}
\label{eq:global_tail_optimality_exact}
  P_t^{(\pi)}(\Omega_{\bar\iota}) = \Fit_{\bar\iota}(t).
\end{equation}
Equivalently, the same identity holds for every \(t\ge 0\), since both sides vanish at \(t=0\).
\end{definition}

An equivalent language is provided by the overflow profile, which records the complementary mass that
cannot be packed below level \(t\).  This is the tail functional that will later connect the present
levelwise criterion to the hockey-stick reduction from Section~\ref{sec:reduction_HS} and to the
convex comparison arguments below.

\paragraph{Overflow profile.}
For \(\pi\in\Pi^\mathcal E_\iota\), let
\(P=\nu^{(\pi)}_{\bar\iota}=r\,\nu_{\bar\iota}+P^\perp\)
be the Lebesgue decomposition with respect to \(\nu_{\bar\iota}\). Define
\begin{equation}
\label{eq:overflow_profile}
  \Over^\pi(t)
  :=
  P^\perp(\Omega_{\bar\iota})
  +
  \int_{\Omega_{\bar\iota}}(r-t)_+\,d\nu_{\bar\iota},
  \qquad t\ge 0.
\end{equation}

\begin{lemma}[Overflow decomposition with uniquely determined opposite marginals]
\label{lem:overflow_decomposition}
Fix \(\iota\in\mathbb B\), write \(\bar\iota:=1-\iota\), and let \(\pi\in\Pi^{\mathcal E}_\iota\).
Write
\[
  P := \nu^{(\pi)}_{\bar\iota}=r\,\nu_{\bar\iota}+P^\perp,
  \qquad
  P^\perp\perp \nu_{\bar\iota},
\]
and fix \(t\ge 0\).

Then there exist measures
\[
  \sigma_t\in\Feas_{\bar\iota}(t),
  \qquad
  \pi_t^{\mathrm{ov}}\in\mathcal M_+(\Omega_0\times\Omega_1),
\]
such that
\[
  \pi=\sigma_t+\pi_t^{\mathrm{ov}},
  \qquad
  \sigma_t(\mathcal E^\c)=\pi_t^{\mathrm{ov}}(\mathcal E^\c)=0,
  \qquad
  \pi_t^{\mathrm{ov}}(\Omega_0\times\Omega_1)=\Over^\pi(t).
\]

Moreover, although the decomposition \(\pi=\sigma_t+\pi_t^{\mathrm{ov}}\) need not be unique, the
\(\bar\iota\)-marginals are uniquely determined:
\[
  (\pr_{\bar\iota})_\#\sigma_t=(r\wedge t)\,\nu_{\bar\iota},
  \qquad
  (\pr_{\bar\iota})_\#\pi_t^{\mathrm{ov}}=P^\perp+(r-t)_+\,\nu_{\bar\iota}.
\]
\end{lemma}

\begin{proof}
Fix \(t\ge 0\). Set \(Q:=(r\wedge t)\,\nu_{\bar\iota}\) and \(R:=P-Q\). Since
\(Q\le r\,\nu_{\bar\iota}\le P\), one has \(Q\ll P\). Let \(f:=dQ/dP\), modified on a \(P\)-null set
so that \(0\le f\le 1\) everywhere. Also,
\(R=P^\perp + \bigl(r-(r\wedge t)\bigr)\,\nu_{\bar\iota}
  = P^\perp + (r-t)_+\,\nu_{\bar\iota}\).

\paragraph{Construction from the reverse disintegration.}
Choose a disintegration of \(\pi\) with respect to its \(\bar\iota\)-marginal \(P\), say
\(\pi=P\kprod K\), where \(K:\Omega_{\bar\iota}\times\Sigma_\iota\to[0,1]\) is a probability kernel
\(P\)-a.e. Define
\[
  \sigma_t:=(fP)\kprod K,
  \qquad
  \pi_t^{\mathrm{ov}}:=((1-f)P)\kprod K.
\]
Then \(\pi=\sigma_t+\pi_t^{\mathrm{ov}}\).

For every measurable rectangle \(A\times B\subseteq \Omega_\iota\times\Omega_{\bar\iota}\),
\[
  \sigma_t(A\times B)
  =
  \int_B K(y,A)\,f(y)\,P(dy)
  =
  \int_{A\times B} f\circ \pr_{\bar\iota}\,d\pi.
\]
By the monotone class theorem, \(d\sigma_t = (f\circ \pr_{\bar\iota})\,d\pi\). Likewise,
\(d\pi_t^{\mathrm{ov}} = \bigl(1-f\circ \pr_{\bar\iota}\bigr)\,d\pi\). Since \(0\le f\le 1\), it follows
that \(\sigma_t\le \pi\) and \(\pi_t^{\mathrm{ov}}\le \pi\). As \(\pi(\mathcal E^\c)=0\), one obtains
\(\sigma_t(\mathcal E^\c)=\pi_t^{\mathrm{ov}}(\mathcal E^\c)=0\).

\paragraph{Identification of the \(\bar\iota\)-marginals.}
Because \(fP\ll P\) and \((1-f)P\ll P\), the kernel \(K\) is a probability kernel
\((fP)\)-a.e.\ and \(((1-f)P)\)-a.e. Hence
\[
  (\pr_{\bar\iota})_\#\sigma_t = fP = Q = (r\wedge t)\,\nu_{\bar\iota},
  \qquad
  (\pr_{\bar\iota})_\#\pi_t^{\mathrm{ov}}
  = (1-f)P
  = P^\perp + (r-t)_+\,\nu_{\bar\iota}.
\]
In particular, \((\pr_{\bar\iota})_\#\sigma_t=(r\wedge t)\,\nu_{\bar\iota}\le t\,\nu_{\bar\iota}\).
Since \(\sigma_t\le \pi\) and \((\pr_\iota)_\#\pi=\nu_\iota\), also
\((\pr_\iota)_\#\sigma_t\le \nu_\iota\). Thus \(\sigma_t\in \Feas_{\bar\iota}(t)\).

Finally,
\[
  \pi_t^{\mathrm{ov}}(\Omega_0\times\Omega_1)
  =
  \bigl((\pr_{\bar\iota})_\#\pi_t^{\mathrm{ov}}\bigr)(\Omega_{\bar\iota})
  =
  P^\perp(\Omega_{\bar\iota})
  +
  \int_{\Omega_{\bar\iota}} (r-t)_+\,d\nu_{\bar\iota}
  =
  \Over^\pi(t).
\]
This proves existence and the stated formulas for the \(\bar\iota\)-marginals.

\paragraph{Uniqueness of the \(\bar\iota\)-marginals.}
Let \(\sigma,\pi^{\mathrm{ov}}\in\mathcal M_+(\Omega_0\times\Omega_1)\) satisfy
\[
  \pi=\sigma+\pi^{\mathrm{ov}},
  \qquad
  \sigma\in\Feas_{\bar\iota}(t),
  \qquad
  \pi^{\mathrm{ov}}(\Omega_0\times\Omega_1)=\Over^\pi(t).
\]
Set \(S:=(\pr_{\bar\iota})_\#\sigma\). Since \(S\le t\,\nu_{\bar\iota}\), one has
\(S\ll \nu_{\bar\iota}\), so \(S=q\,\nu_{\bar\iota}\) for some \(q\in L^1(\nu_{\bar\iota})\) with
\(0\le q\le t\) \(\nu_{\bar\iota}\)-a.e.

Also \(\sigma\le\pi\) implies \(S\le P\). Choose \(Z\in\Sigma_{\bar\iota}\) such that
\(\nu_{\bar\iota}(Z)=0\) and \(P^\perp(Z^\c)=0\). Then for every measurable
\(A\subseteq\Omega_{\bar\iota}\),
\[
  \int_{A\cap Z^\c} q\,d\nu_{\bar\iota}
  =
  S(A\cap Z^\c)
  \le
  P(A\cap Z^\c)
  =
  \int_{A\cap Z^\c} r\,d\nu_{\bar\iota}.
\]
Hence \(q\le r\) \(\nu_{\bar\iota}\)-a.e., so \(0\le q\le r\wedge t\) \(\nu_{\bar\iota}\)-a.e.

Now
\[
  (\pr_{\bar\iota})_\#\pi^{\mathrm{ov}}
  =
  P-S
  =
  P^\perp + (r-q)\,\nu_{\bar\iota},
\]
and therefore
\[
  \pi^{\mathrm{ov}}(\Omega_0\times\Omega_1)
  =
  P^\perp(\Omega_{\bar\iota})
  +
  \int_{\Omega_{\bar\iota}} (r-q)\,d\nu_{\bar\iota}.
\]
Comparing with \(\pi^{\mathrm{ov}}(\Omega_0\times\Omega_1)=\Over^\pi(t)\), one gets
\[
  \int_{\Omega_{\bar\iota}} \bigl((r\wedge t)-q\bigr)\,d\nu_{\bar\iota}=0.
\]
Since \((r\wedge t)-q\ge 0\) \(\nu_{\bar\iota}\)-a.e., it follows that
\(q=r\wedge t\) \(\nu_{\bar\iota}\)-a.e. Thus
\[
  (\pr_{\bar\iota})_\#\sigma=(r\wedge t)\,\nu_{\bar\iota},
  \qquad
  (\pr_{\bar\iota})_\#\pi^{\mathrm{ov}}
  =
  P-(r\wedge t)\,\nu_{\bar\iota}
  =
  P^\perp+(r-t)_+\,\nu_{\bar\iota}.
\]
This proves the claimed uniqueness.
\end{proof}

\medskip
\noindent
Lemma~\ref{lem:overflow_decomposition} provides, for each level \(t\), a decomposition of a plan
into a subplan whose shipment stays below density level \(t\) with respect to
\(\nu_{\bar\iota}\), and a complementary \emph{overflow} subplan. Although this decomposition
need not be unique as a measure on \(\Omega_0\times\Omega_1\), its
\(\bar\iota\)-marginals are uniquely determined. The total mass of the overflow part is precisely
\(\Over^\pi(t)\), namely the amount of shipment exceeding level \(t\), together with any
\(\nu_{\bar\iota}\)-singular mass. Equivalently, \(t\mapsto\Over^\pi(t)\) is the tail functional
appearing in the hockey-stick divergence between the opposite marginal
\((\pr_{\bar\iota})_\#\pi\) and \(\nu_{\bar\iota}\).

The next lemma records the corresponding hinge representation when both endpoint slopes are
finite. This is the form needed later for the bounded-slope Fenchel approximants in
Theorem~\ref{thm:level_opt_maximin_implies_theta_opt}.

\begin{lemma}[Hinge representation under finite endpoint slopes]
\label{lem:convex_hinge_representation_for_divergence}
Fix \(\iota\in\mathbb B\), write \(\bar\iota:=1-\iota\), and let
\(\phi:[0,\infty)\to\mathbb R\) be convex and lower semicontinuous. Assume that
\[
  \phi'_+(0)\in\mathbb R
  \qquad\text{and}\qquad
  \phi'_\infty\in\mathbb R.
\]
Let \(\mu_\phi\) be the Borel measure on \((0,\infty)\) determined by
\[
  \mu_\phi((a,b])
  :=
  \phi'_+(b)-\phi'_+(a),
  \qquad
  0<a<b<\infty.
\]
Then \(\mu_\phi\) is finite, with
\[
  \mu_\phi((0,\infty))
  =
  \phi'_\infty-\phi'_+(0),
\]
and, for every \(\pi\in\Pi^{\mathcal E}_\iota\),
\begin{equation}
\label{eq:Dphi_overflow_integral}
  \D_\phi\bigl(\nu^{(\pi)}_{\bar\iota}\,\|\,\nu_{\bar\iota}\bigr)
  =
  \phi(0)\,\nu_{\bar\iota}(\Omega_{\bar\iota})
  +
  \phi'_+(0)\,\nu_\iota(\Omega_\iota)
  +
  \int_{(0,\infty)}\Over^\pi(t)\,\mu_\phi(dt).
\end{equation}
In particular, the first two terms are independent of \(\pi\).
\end{lemma}

\begin{proof}
Write
\[
  P:=\nu^{(\pi)}_{\bar\iota}
  =
  r\,\nu_{\bar\iota}+P^\perp,
  \qquad
  P^\perp\perp\nu_{\bar\iota}.
\]

Since \(\phi\) is finite and convex on \([0,\infty)\), its right derivative
\(\phi'_+\) exists and is finite on \((0,\infty)\), is nondecreasing and
right-continuous, and satisfies
\(\phi'_+(u)\uparrow\phi'_\infty\) as \(u\to\infty\). By hypothesis,
\(\phi'_+(0)\) is also finite. Hence, by continuity from below and
right-continuity at \(0\),
\[
  \mu_\phi((0,s])
  =
  \phi'_+(s)-\phi'_+(0),
  \qquad s\ge0,
\]
and therefore
\[
  \mu_\phi((0,\infty))
  =
  \phi'_\infty-\phi'_+(0)
  <
  \infty.
\]

For \(u\ge0\), the fundamental theorem for finite convex functions and
Tonelli's theorem give
\begin{align}
  \phi(u)
  &=
  \phi(0)+\int_0^u\phi'_+(s)\,ds \notag\\
  &=
  \phi(0)+\phi'_+(0)\,u
  +\int_0^u\mu_\phi((0,s])\,ds \notag\\
  &=
  \phi(0)+\phi'_+(0)\,u
  +\int_{(0,\infty)}(u-t)_+\,\mu_\phi(dt).
  \label{eq:phi_hinge_clean}
\end{align}

Integrating \eqref{eq:phi_hinge_clean} with \(u=r\) against
\(\nu_{\bar\iota}\) and applying Tonelli once more yields
\begin{align*}
  \int_{\Omega_{\bar\iota}}\phi(r)\,d\nu_{\bar\iota}
  &=
  \phi(0)\,\nu_{\bar\iota}(\Omega_{\bar\iota})
  +
  \phi'_+(0)
  \int_{\Omega_{\bar\iota}}r\,d\nu_{\bar\iota} \\
  &\quad
  +
  \int_{(0,\infty)}
  \left(
    \int_{\Omega_{\bar\iota}}(r-t)_+\,d\nu_{\bar\iota}
  \right)
  \mu_\phi(dt).
\end{align*}
Since
\[
  P(\Omega_{\bar\iota})
  =
  \int_{\Omega_{\bar\iota}}r\,d\nu_{\bar\iota}
  +
  P^\perp(\Omega_{\bar\iota})
  =
  \nu_\iota(\Omega_\iota),
\]
the definition of \(\D_\phi\) gives
\begin{align*}
  \D_\phi(P\|\nu_{\bar\iota})
  &=
  \phi(0)\,\nu_{\bar\iota}(\Omega_{\bar\iota})
  +
  \phi'_+(0)\,\nu_\iota(\Omega_\iota) \\
  &\quad
  +
  \int_{(0,\infty)}
  \left(
    \int_{\Omega_{\bar\iota}}(r-t)_+\,d\nu_{\bar\iota}
  \right)
  \mu_\phi(dt) \\
  &\quad
  +
  \bigl(\phi'_\infty-\phi'_+(0)\bigr)
  P^\perp(\Omega_{\bar\iota}).
\end{align*}
Because
\(\phi'_\infty-\phi'_+(0)=\mu_\phi((0,\infty))\), the last term equals
\[
  \int_{(0,\infty)}
  P^\perp(\Omega_{\bar\iota})\,\mu_\phi(dt).
\]
The definition of \(\Over^\pi\) therefore gives
\[
  \D_\phi(P\|\nu_{\bar\iota})
  =
  \phi(0)\,\nu_{\bar\iota}(\Omega_{\bar\iota})
  +
  \phi'_+(0)\,\nu_\iota(\Omega_\iota)
  +
  \int_{(0,\infty)}\Over^\pi(t)\,\mu_\phi(dt),
\]
which is \eqref{eq:Dphi_overflow_integral}.
\end{proof}

\section{Measure-Theoretic Bipartite Tools}
\label{sec:measure_bipartite}

This section gathers the measure-theoretic infrastructure used later in the analysis of
level-optimal maximin and convex refinement.  The first tool is a measurable max-flow/min-cut
formula for constrained shipment on a bipartite relation; in particular, it is the device that
will later identify capped feasible mass in the sense of Section~\ref{sec:maximin}.  The second
tool is an augmenting-subplan lemma, which plays the role of a measure-theoretic replacement for
finite exchange arguments and will supply the local perturbations used later in
Section~\ref{sec:one_side}.

\subsection{Measure-Theoretic Max-Flow/Min-Cut}
\label{sec:measure_maxflow_mincut}

The first ingredient is a measurable max-flow/min-cut formula for constrained shipment on a bipartite relation, with standard Borel vertex spaces and finite measures in place of finite capacities. The key duality input is Kellerer's measurable marginal duality theorem~\cite[Corollary~(2.18)]{Kellerer_1984}. After an elementary augmentation that converts the constrained shipment problem into a fixed-marginal transport problem with bounded Borel payoff \(\mathbf 1_{\mathcal E}\), Kellerer's theorem yields the relevant dual formulation. A layer-cake argument then identifies the dual value with the cut expression used later in the paper. We subsequently show that this cut infimum is attained; that additional conclusion follows from the submodularity and monotone-continuity properties of the cut functional.

\begin{proposition}[Measurable max-flow/min-cut duality on a bipartite relation]
\label{prop:measurable_maxflow_mincut_bipartite}
Let \((\Omega_0,\Sigma_0)\) and \((\Omega_1,\Sigma_1)\) be standard Borel spaces, let
\(\mathcal E\in\Sigma_0\otimes\Sigma_1\) be a measurable bipartite relation, and let
\(a\in\mathcal M_+(\Omega_0)\) and \(b\in\mathcal M_+(\Omega_1)\) be finite measures.

Define
\[
  \mathcal F(a,b)
  :=
  \sup\Bigl\{
    \sigma(\Omega_0\times\Omega_1):
    \sigma\in\mathcal M_+(\Omega_0\times\Omega_1),\
    \sigma(\mathcal E^\c)=0,\
    (\pr_0)_\#\sigma\le a,\
    (\pr_1)_\#\sigma\le b
  \Bigr\}.
\]
For \(C\in\Sigma_0\), write
\[
  N_1(C):=\{y\in\Omega_1:\exists x\in C \text{ with } (x,y)\in\mathcal E\}.
\]
Then \(N_1(C)\) is analytic, hence universally measurable, so \(b(N_1(C))\) is
well defined after completion. Moreover,
\begin{equation}\label{eq:mfmc_correct}
  \mathcal F(a,b)
  =
  \inf_{C\in\Sigma_0}\bigl\{a(C)+b\bigl(N_1(C^\c)\bigr)\bigr\}.
\end{equation}
\end{proposition}

\begin{proof}
The argument consists of a cut bound, followed by a reduction to a fixed-marginal transport
problem on an augmented space. The duality step is then supplied by Kellerer's measurable marginal
duality theorem~\cite[Corollary~(2.18)]{Kellerer_1984}, and the resulting dual problem is identified with the cut
expression by a layer-cake argument.

Write
\[
  \Phi(a,b):=\inf_{C\in\Sigma_0}\bigl\{a(C)+b\bigl(N_1(C^\c)\bigr)\bigr\}.
\]

\paragraph{The cut bound.}
Let \(\sigma\) be feasible in the definition of \(\mathcal F(a,b)\) and fix
\(C\in\Sigma_0\). Decompose \(\sigma\) into its restrictions to
\(C\times\Omega_1\) and \(C^\c\times\Omega_1\). The first part has mass at most
\(a(C)\) since \((\pr_0)_\#\sigma\le a\). For the second part, the relation
constraint \(\sigma(\mathcal E^\c)=0\) implies that its \(\Omega_1\)-marginal is
carried by \(N_1(C^\c)\); hence its mass is at most \(b(N_1(C^\c))\). Therefore
\[
  \sigma(\Omega_0\times\Omega_1)\le a(C)+b\bigl(N_1(C^\c)\bigr).
\]
Taking the supremum over feasible \(\sigma\) and then the infimum over \(C\)
gives
\[
  \mathcal F(a,b)\le \Phi(a,b).
\]

\paragraph{Reduction to a fixed-marginal problem.}
Adjoin isolated points \(\partial_0,\partial_1\) and set
\(\overline\Omega_0:=\Omega_0\sqcup\{\partial_0\}\) and
\(\overline\Omega_1:=\Omega_1\sqcup\{\partial_1\}\). Define
\[
  \overline a:=a+b(\Omega_1)\delta_{\partial_0},
  \qquad
  \overline b:=b+a(\Omega_0)\delta_{\partial_1},
\]
so that \(\overline a(\overline\Omega_0)=\overline b(\overline\Omega_1)\). Let
\(\Pi(\overline a,\overline b)\) denote the set of couplings of
\((\overline a,\overline b)\), and let \(h:=\mathbf 1_{\mathcal E}\) on
\(\overline\Omega_0\times\overline\Omega_1\), where \(\mathcal E\) is viewed as a
subset of \(\overline\Omega_0\times\overline\Omega_1\).

Then
\begin{equation}\label{eq:augmented_primal}
  \mathcal F(a,b)=\sup_{\gamma\in\Pi(\overline a,\overline b)}\int h\,d\gamma .
\end{equation}
Indeed, if \(\sigma\) is feasible for \(\mathcal F(a,b)\), then
\[
  \gamma
  :=
  \sigma
  +(a-(\pr_0)_\#\sigma)\otimes\delta_{\partial_1}
  +\delta_{\partial_0}\otimes(b-(\pr_1)_\#\sigma)
  +\sigma(\Omega_0\times\Omega_1)\delta_{(\partial_0,\partial_1)}
\]
belongs to \(\Pi(\overline a,\overline b)\) and satisfies
\(\int h\,d\gamma=\sigma(\Omega_0\times\Omega_1)\). Conversely, if
\(\gamma\in\Pi(\overline a,\overline b)\), then
\(\sigma:=\gamma\restriction\mathcal E\) satisfies
\(\sigma(\mathcal E^\c)=0\), \((\pr_0)_\#\sigma\le a\), \((\pr_1)_\#\sigma\le b\),
and \(\sigma(\Omega_0\times\Omega_1)=\int h\,d\gamma\). This proves
\eqref{eq:augmented_primal}.

Since \(\overline\Omega_0\) and \(\overline\Omega_1\) are standard Borel, we may
equip them with Polish topologies generating the given Borel structures.
Thus each factor is a Suslin space, and
\(h=\mathbf 1_{\mathcal E}\) is bounded and Borel measurable on
\(\overline\Omega_0\times\overline\Omega_1\).
Kellerer's duality theorem for measurable marginal problems therefore applies to \eqref{eq:augmented_primal};
specifically, this is
\cite[Corollary~(2.18)]{Kellerer_1984}.
Thus
\[
  \mathcal F(a,b)
  =
  \inf\Bigl\{
    \int_{\overline\Omega_0}u\,d\overline a
    +\int_{\overline\Omega_1}v\,d\overline b:
    u(\bar x)+v(\bar y)\ge \mathbf 1_{\mathcal E}(\bar x,\bar y)
    \ \text{for all }(\bar x,\bar y)
  \Bigr\}.
\]
For such a dual-feasible pair \((u,v)\), set
\[
  f(x):=u(x)+v(\partial_1), \qquad x\in\Omega_0,
\]
and
\[
  g(y):=v(y)+u(\partial_0), \qquad y\in\Omega_1.
\]
Because \(h(x,\partial_1)=h(\partial_0,y)=0\), feasibility implies
\(f\ge 0\) on \(\Omega_0\) and \(g\ge 0\) on \(\Omega_1\); and since
\(h(x,y)=1\) on \(\mathcal E\), one has \(f(x)+g(y)\ge 1\) for all
\((x,y)\in\mathcal E\). Moreover,
\[
  \int_{\overline\Omega_0}u\,d\overline a
  +\int_{\overline\Omega_1}v\,d\overline b
  =
  \int_{\Omega_0} f\,da+\int_{\Omega_1} g\,db.
\]
Conversely, any measurable \(f,g\ge 0\) with \(f(x)+g(y)\ge 1\) on \(\mathcal E\)
arises in this way by taking \(u=f\) on \(\Omega_0\), \(u(\partial_0)=0\),
\(v=g\) on \(\Omega_1\), and \(v(\partial_1)=0\). Hence
\[
  \mathcal F(a,b)
  =
  \inf\Bigl\{
    \int_{\Omega_0} f\,da+\int_{\Omega_1} g\,db:
    f,g\ge 0 \text{ measurable},\
    f(x)+g(y)\ge 1 \ \text{on }\mathcal E
  \Bigr\}.
\]
Truncating above by \(1\) preserves feasibility and does not increase the
objective, so
\begin{equation}\label{eq:dual_value}
  \mathcal F(a,b)=\mathcal D(a,b),
\end{equation}
where
\[
  \mathcal D(a,b)
  :=
  \inf\Bigl\{
    \int_{\Omega_0} f\,da+\int_{\Omega_1} g\,db:
    f:\Omega_0\to[0,1],\ g:\Omega_1\to[0,1]\ \text{measurable},\
    f(x)+g(y)\ge 1 \ \text{on }\mathcal E
  \Bigr\}.
\]

\paragraph{Identification of the dual value.}
It remains to show that \(\mathcal D(a,b)=\Phi(a,b)\). For the inequality
\(\mathcal D(a,b)\le \Phi(a,b)\), fix \(C\in\Sigma_0\). Since \(N_1(C^\c)\) is
analytic, hence \(b\)-measurable after completion, there exists
\(B\in\Sigma_1\) with \(N_1(C^\c)\subseteq B\) and \(b(B)=b(N_1(C^\c))\).
Set \(f=\mathbf 1_C\) and \(g=\mathbf 1_B\). Then \(f\) and \(g\) are measurable, and
for \((x,y)\in\mathcal E\) one has either \(x\in C\), hence \(f(x)=1\), or
\(x\in C^\c\), in which case \(y\in N_1(C^\c)\subseteq B\), hence \(g(y)=1\).
Therefore \(f(x)+g(y)\ge 1\) on \(\mathcal E\), so
\[
  \mathcal D(a,b)\le a(C)+b(B)=a(C)+b\bigl(N_1(C^\c)\bigr).
\]
Taking the infimum over \(C\) yields \(\mathcal D(a,b)\le \Phi(a,b)\).

For the reverse inequality, let \(f,g\) be dual-feasible with values in \([0,1]\),
and set \(C_s:=\{x\in\Omega_0:f(x)\ge s\}\) for \(s\in[0,1]\). If
\(y\in N_1(C_s^\c)\), then there exists \(x\in C_s^\c\) with \((x,y)\in\mathcal E\),
hence \(f(x)<s\) and therefore \(g(y)\ge 1-f(x)>1-s\). Thus
\[
  N_1(C_s^\c)\subseteq \{y\in\Omega_1:g(y)>1-s\}.
\]
It follows that for every \(s\in[0,1]\),
\[
  a(C_s)+b(\{g>1-s\})\ge a(C_s)+b\bigl(N_1(C_s^\c)\bigr)\ge \Phi(a,b).
\]
Using the layer-cake formula,
\[
  \int_{\Omega_0} f\,da \ge \int_0^1 a(C_s)\,ds,
  \qquad
  \int_{\Omega_1} g\,db
  = \int_0^1 b(\{g>1-s\})\,ds.
\]
Hence
\[
  \int_{\Omega_0} f\,da+\int_{\Omega_1} g\,db
  \ge
  \int_0^1 \bigl(a(C_s)+b(\{g>1-s\})\bigr)\,ds
  \ge \Phi(a,b).
\]
Taking the infimum over all dual-feasible \((f,g)\) gives
\(\mathcal D(a,b)\ge \Phi(a,b)\).

Combining this with \eqref{eq:dual_value}, we obtain
\(\mathcal F(a,b)=\mathcal D(a,b)=\Phi(a,b)\), which is exactly
\eqref{eq:mfmc_correct}.
\end{proof}

Proposition~\ref{prop:measurable_maxflow_mincut_bipartite} identifies the maximal feasible shipment with the corresponding cut value. For the carrier arguments used later, we also need an actual minimizing cut rather than only the equality of the two infimal values. The next lemma supplies this attainment statement from the lattice structure of the cut functional.

\begin{lemma}[Attainment of measurable minimum cuts]
\label{lem:measurable_min_cut_attainment}
Under the hypotheses and notation of
Proposition~\ref{prop:measurable_maxflow_mincut_bipartite}, there exists
\(C^\star\in\Sigma_0\) such that
\[
  \mathcal F(a,b)
  =
  a(C^\star)+b\bigl(N_1((C^\star)^\c)\bigr).
\]
Moreover, there exists \(B^\star\in\Sigma_1\) such that
\[
  N_1((C^\star)^\c)\subseteq B^\star,
  \qquad
  b(B^\star)=b\bigl(N_1((C^\star)^\c)\bigr).
\]
\end{lemma}

\begin{proof}
Write \(m:=\mathcal F(a,b)\) and
\[
  \Psi(C):=a(C)+b\bigl(N_1(C^\c)\bigr),
  \qquad C\in\Sigma_0.
\]
As in Proposition~\ref{prop:measurable_maxflow_mincut_bipartite},
the measure of each analytic neighborhood set is understood in the
\(b\)-completion. By that proposition,
\(m=\inf_{C\in\Sigma_0}\Psi(C)\).

\noindent\textbf{Submodularity.}
For \(A,C\in\Sigma_0\), one has
\[
  N_1((A\cap C)^\c)
  =
  N_1(A^\c)\cup N_1(C^\c),
  \qquad
  N_1((A\cup C)^\c)
  \subseteq
  N_1(A^\c)\cap N_1(C^\c).
\]
Since \(a\) is modular, it follows that
\[
  \Psi(A\cap C)+\Psi(A\cup C)
  \le
  \Psi(A)+\Psi(C).
\]

\noindent\textbf{Tail intersections of approximate minimizers.}
Choose \(C_n\in\Sigma_0\) such that
\(\Psi(C_n)<m+2^{-n}\) for every \(n\ge1\). For \(1\le n\le N\), set
\[
  D_{n,N}:=\bigcap_{k=n}^N C_k.
\]
We claim that
\[
  \Psi(D_{n,N})
  <
  m+\sum_{k=n}^N2^{-k}.
\]
The case \(N=n\) follows from the choice of \(C_n\). For \(N>n\),
submodularity and the lower bound \(\Psi(E)\ge m\) give
\[
\begin{aligned}
  \Psi(D_{n,N})
  &=
  \Psi(D_{n,N-1}\cap C_N) \\
  &\le
  \Psi(D_{n,N-1})+\Psi(C_N)
  -\Psi(D_{n,N-1}\cup C_N) \\
  &\le
  \Psi(D_{n,N-1})+\Psi(C_N)-m,
\end{aligned}
\]
so the claim follows by induction.

Now set \(D_n:=\bigcap_{k=n}^\infty C_k\). For fixed \(n\),
the sets \(D_{n,N}\) decrease to \(D_n\), while
\(D_{n,N}^\c\uparrow D_n^\c\). Since the neighborhood operator
preserves unions,
\[
  N_1(D_n^\c)
  =
  \bigcup_{N\ge n}N_1(D_{n,N}^\c).
\]
Continuity from above for the finite measure \(a\) and continuity from
below for \(b\) therefore yield
\[
  \Psi(D_n)
  =
  \lim_{N\to\infty}\Psi(D_{n,N})
  \le
  m+\sum_{k=n}^\infty2^{-k}.
\]
Since also \(\Psi(D_n)\ge m\), it follows that
\(\Psi(D_n)\to m\).

\noindent\textbf{Passage to a minimizing cut.}
The sets \(D_n\) are increasing. Define
\[
  C^\star:=\bigcup_{n=1}^\infty D_n
  =
  \liminf_{n\to\infty}C_n.
\]
Then \(a(D_n)\uparrow a(C^\star)\), whereas \(D_n^\c\downarrow
(C^\star)^\c\) and hence \(N_1(D_n^\c)\) decreases. Moreover,
\[
  N_1((C^\star)^\c)
  =
  N_1\left(\bigcap_{n=1}^\infty D_n^\c\right)
  \subseteq
  \bigcap_{n=1}^\infty N_1(D_n^\c).
\]
By continuity from above for the finite measure \(b\),
\[
\begin{aligned}
  \Psi(C^\star)
  &\le
  a(C^\star)
  +
  b\left(\bigcap_{n=1}^\infty N_1(D_n^\c)\right) \\
  &=
  \lim_{n\to\infty}
  \left[
    a(D_n)+b\bigl(N_1(D_n^\c)\bigr)
  \right]
  =
  m.
\end{aligned}
\]
The reverse inequality follows from the definition of \(m\). Thus
\[
  \Psi(C^\star)=m=\mathcal F(a,b).
\]

\noindent\textbf{Borel hull.}
The set \(N_1((C^\star)^\c)\) is analytic and hence measurable in the
\(b\)-completion. By the definition of the completed measure, there
exists \(B^\star\in\Sigma_1\) such that
\[
  N_1((C^\star)^\c)\subseteq B^\star,
  \qquad
  b(B^\star)=b\bigl(N_1((C^\star)^\c)\bigr).
\]
\end{proof}

\subsection{Measure-Theoretic Augmenting Subplan}
\label{sec:measure_augment_subplan}

The second ingredient is a measure-theoretic augmenting-subplan statement in the spirit of the
classical theory of probability measures with prescribed marginals and relation constraints, going
back to Strassen \cite{Strassen1965} and its extensions by Edwards
\cite{Edwards1978Marginals} and Kellerer \cite{Kellerer_1984}.  The proof given here is not quoted
from a single source in this exact form, but follows the same general paradigm of reducing to a
finite max-flow/min-cut problem on a measurable discretization and then passing back to the measure
setting by compactness; compare also Lovász’s measurable-space flow framework
\cite{Lovasz_2021}.  Its role in the sequel is to replace finite exchange or augmenting-path
arguments by a measure-theoretic perturbation principle, which will be used later in
Section~\ref{sec:one_side}.

\begin{lemma}[Measure-theoretic augmenting subplan]
\label{lem:measure_theoretic_augmenting_subplan}
Let \((\Omega_0,\Sigma_0)\) and \((\Omega_1,\Sigma_1)\) be measurable spaces. Fix
\(\iota\in\B\) and write \(\bar\iota:=1-\iota\). Let
\(b\in \mathcal M_+(\Omega_{\bar\iota})\) be finite, and let
\(\sigma,\sigma_0\in \mathcal M_+(\Omega_0\times\Omega_1)\).

Write
\[
  \xi := (\pr_\iota)_\#\sigma,\qquad
  \xi_0 := (\pr_\iota)_\#\sigma_0,\qquad
  \psi := (\pr_{\bar\iota})_\#\sigma,\qquad
  \psi_0 := (\pr_{\bar\iota})_\#\sigma_0.
\]
Assume that \(\psi\le b\), \(\psi_0\le b\), and
\(\sigma(\Omega_0\times\Omega_1)>\sigma_0(\Omega_0\times\Omega_1)\). Define
\[
  L:=\Bigl\{y\in\Omega_{\bar\iota}:\frac{d\psi_0}{db}(y)<1\Bigr\}.
\]

Then there exist measures \(\gamma,\gamma_0\in\mathcal M_+(\Omega_0\times\Omega_1)\),
with \(\gamma\le\sigma\) and \(\gamma_0\le\sigma_0\), such that, writing
\[
  \alpha := (\pr_\iota)_\#\gamma,\qquad
  \alpha_0 := (\pr_\iota)_\#\gamma_0,\qquad
  \beta := (\pr_{\bar\iota})_\#\gamma,\qquad
  \beta_0 := (\pr_{\bar\iota})_\#\gamma_0,
\]
one has:
\begin{enumerate}[(i)]
\item \(\alpha_0\le \alpha\) and
      \(0\neq \alpha-\alpha_0\le (\xi-\xi_0)_+\);

\item \(\beta_0\le \beta\), and the nonzero measure \(\beta-\beta_0\) is carried by \(L\),
      i.e. \((\beta-\beta_0)(L^\c)=0\);

\item \(\psi_0-\beta_0+\beta\le b\).
\end{enumerate}
\end{lemma}

\begin{proof}
By symmetry it is enough to treat the case \(\iota=0\), so \(\bar\iota=1\). Write
\[
  \xi := (\pr_0)_\#\sigma,\qquad
  \xi_0 := (\pr_0)_\#\sigma_0,\qquad
  \psi := (\pr_1)_\#\sigma,\qquad
  \psi_0 := (\pr_1)_\#\sigma_0.
\]
Let \(f_0\) be a measurable version of \(d\psi_0/db\), and set
\[
  L:=\{y\in\Omega_1:f_0(y)<1\}.
\]
Since \(\psi_0\le b\), one has \(\psi_0\ll b\) and \(f_0\le 1\) \(b\)-a.e. Hence
\[
  (b-\psi_0)(L^\c)=\int_{L^\c}(1-f_0)\,db=0,
\]
so \(b-\psi_0\) is carried by \(L\).

Set \(\nu:=\xi-\xi_0\), let \(\nu=\nu^+-\nu^-\) be its Jordan decomposition, and write
\[
  \Delta:=\sigma(\Omega_0\times\Omega_1)-\sigma_0(\Omega_0\times\Omega_1)
  =\nu(\Omega_0)>0.
\]
It is enough to construct measures \(\gamma\le \sigma\) and \(\gamma_0\le \sigma_0\) such that, with
\[
  \alpha := (\pr_0)_\#\gamma,\qquad
  \alpha_0 := (\pr_0)_\#\gamma_0,\qquad
  \beta := (\pr_1)_\#\gamma,\qquad
  \beta_0 := (\pr_1)_\#\gamma_0,
\]
one has
\begin{equation}\label{eq:target_bounds_aug_subplan}
  0\le \alpha-\alpha_0\le \nu^+,\qquad
  0\le \beta-\beta_0\le b-\psi_0,\qquad
  (\beta-\beta_0)(\Omega_1)\ge \Delta.
\end{equation}
Indeed, these inequalities imply \(\alpha_0\le \alpha\) and
\(0\neq \alpha-\alpha_0\le (\xi-\xi_0)_+\), since
\[
  (\alpha-\alpha_0)(\Omega_0)
  =\gamma(\Omega_0\times\Omega_1)-\gamma_0(\Omega_0\times\Omega_1)
  =(\beta-\beta_0)(\Omega_1)\ge \Delta>0.
\]
Likewise \(\beta_0\le \beta\), the measure \(\beta-\beta_0\) is nonzero and carried by \(L\), because it is dominated by \(b-\psi_0\), and
\[
  \psi_0-\beta_0+\beta=\psi_0+(\beta-\beta_0)\le \psi_0+(b-\psi_0)=b.
\]
Thus it remains to prove \eqref{eq:target_bounds_aug_subplan}.

\paragraph{Finite-dimensional model.}
Fix finite measurable partitions \(\mathcal P=\{P_1,\dots,P_m\}\subset\Sigma_0\) and
\(\mathcal Q=\{Q_1,\dots,Q_n\}\subset\Sigma_1\). Write
\[
  s_{ij}:=\sigma(P_i\times Q_j),\qquad
  t_{ij}:=\sigma_0(P_i\times Q_j),\qquad
  u_i:=\nu^+(P_i),\qquad
  r_j:=(b-\psi_0)(Q_j).
\]
We claim that there exist nonnegative numbers \(g_{ij}\le s_{ij}\) and \(h_{ij}\le t_{ij}\) such that, with
\[
  \eta_i:=\sum_{j=1}^n(g_{ij}-h_{ij}),\qquad
  \delta_j:=\sum_{i=1}^m(g_{ij}-h_{ij}),
\]
one has
\begin{equation}\label{eq:finite_goal_aug_subplan}
  0\le \eta_i\le u_i\quad(1\le i\le m),\qquad
  0\le \delta_j\le r_j\quad(1\le j\le n),
\end{equation}
and
\begin{equation}\label{eq:finite_value_aug_subplan}
  \sum_{i=1}^m \eta_i=\sum_{j=1}^n \delta_j\ge \Delta.
\end{equation}

Consider the finite directed network with source \(s\), sink \(t\), row vertices \(x_1,\dots,x_m\), column vertices \(y_1,\dots,y_n\), and capacities
\[
  s\to x_i:\ u_i,\qquad
  x_i\to y_j:\ s_{ij},\qquad
  y_j\to x_i:\ t_{ij},\qquad
  y_j\to t:\ r_j.
\]
By the finite max-flow/min-cut theorem, it is enough to show that every \(s\)-\(t\) cut has capacity at least \(\Delta\). Let \(U\) be the source side of such a cut, and set
\[
  A:=\bigcup_{x_i\in U}P_i,\qquad
  B:=\bigcup_{y_j\in U}Q_j.
\]
The capacity of the cut is
\[
  \nu^+(A^\c)+\sigma(A\times B^\c)+\sigma_0(A^\c\times B)+(b-\psi_0)(B).
\]
Since \(\Delta=\nu(\Omega_0)=\nu(A)+\nu(A^\c)\le \nu(A)+\nu^+(A^\c)\), it suffices to prove
\begin{equation}\label{eq:need_nuA_bound_aug_subplan}
  \nu(A)\le \sigma(A\times B^\c)+\sigma_0(A^\c\times B)+(b-\psi_0)(B).
\end{equation}
Now
\[
  \nu(A)=\xi(A)-\xi_0(A)
  =\bigl[\sigma(A\times B)-\sigma_0(A\times B)\bigr]
   +\sigma(A\times B^\c)-\sigma_0(A\times B^\c),
\]
while
\[
  \sigma(A\times B)-\sigma_0(A\times B)
  =\psi(B)-\psi_0(B)-\sigma(A^\c\times B)+\sigma_0(A^\c\times B)
  \le \psi(B)-\psi_0(B)+\sigma_0(A^\c\times B).
\]
Using also \(-\sigma_0(A\times B^\c)\le 0\), we obtain
\[
  \nu(A)\le \psi(B)-\psi_0(B)+\sigma_0(A^\c\times B)+\sigma(A\times B^\c).
\]
Finally \(\psi\le b\) gives \(\psi(B)-\psi_0(B)\le (b-\psi_0)(B)\), and \eqref{eq:need_nuA_bound_aug_subplan} follows.

Hence there is a flow of value at least \(\Delta\). Let \(g_{ij}\) and \(h_{ij}\) be the flows on
\(x_i\to y_j\) and \(y_j\to x_i\), and let \(\eta_i\) and \(\delta_j\) be the flows on
\(s\to x_i\) and \(y_j\to t\). Flow conservation gives the displayed formulas for \(\eta_i\) and \(\delta_j\), the edge capacities yield \eqref{eq:finite_goal_aug_subplan}, and the value of the flow gives \eqref{eq:finite_value_aug_subplan}.

\paragraph{Realization on finitely many test sets.}
Fix finite families \(\mathscr A\subset\Sigma_0\) and \(\mathscr B\subset\Sigma_1\). Let \(\mathcal P\) and \(\mathcal Q\) be the finite partitions generated by \(\mathscr A\) and \(\mathscr B\), and apply the preceding paragraph. For each cell \(P_i\times Q_j\), set
\[
  c_{ij}:=
  \begin{cases}
    g_{ij}/\sigma(P_i\times Q_j),& \sigma(P_i\times Q_j)>0,\\
    0,& \sigma(P_i\times Q_j)=0,
  \end{cases}
  \qquad
  d_{ij}:=
  \begin{cases}
    h_{ij}/\sigma_0(P_i\times Q_j),& \sigma_0(P_i\times Q_j)>0,\\
    0,& \sigma_0(P_i\times Q_j)=0.
  \end{cases}
\]
Then \(0\le c_{ij},d_{ij}\le 1\), so the measures
\[
  \gamma_{\mathscr A,\mathscr B}
  :=\sum_{i,j} c_{ij}\,\sigma\!\restriction(P_i\times Q_j),
  \qquad
  \gamma_{0,\mathscr A,\mathscr B}
  :=\sum_{i,j} d_{ij}\,\sigma_0\!\restriction(P_i\times Q_j)
\]
satisfy \(\gamma_{\mathscr A,\mathscr B}\le \sigma\) and
\(\gamma_{0,\mathscr A,\mathscr B}\le \sigma_0\). Let
\[
  \alpha_{\mathscr A,\mathscr B}:=(\pr_0)_\#\gamma_{\mathscr A,\mathscr B},\qquad
  \alpha_{0,\mathscr A,\mathscr B}:=(\pr_0)_\#\gamma_{0,\mathscr A,\mathscr B},
\]
\[
  \beta_{\mathscr A,\mathscr B}:=(\pr_1)_\#\gamma_{\mathscr A,\mathscr B},\qquad
  \beta_{0,\mathscr A,\mathscr B}:=(\pr_1)_\#\gamma_{0,\mathscr A,\mathscr B}.
\]
By construction,
\[
  \alpha_{\mathscr A,\mathscr B}(P_i)-\alpha_{0,\mathscr A,\mathscr B}(P_i)=\eta_i,\qquad
  \beta_{\mathscr A,\mathscr B}(Q_j)-\beta_{0,\mathscr A,\mathscr B}(Q_j)=\delta_j.
\]
Hence, for every \(A\in\mathscr A\),
\[
  0\le \alpha_{\mathscr A,\mathscr B}(A)-\alpha_{0,\mathscr A,\mathscr B}(A)\le \nu^+(A),
\]
and for every \(B\in\mathscr B\),
\[
  0\le \beta_{\mathscr A,\mathscr B}(B)-\beta_{0,\mathscr A,\mathscr B}(B)\le (b-\psi_0)(B).
\]
Moreover,
\[
  \bigl(\beta_{\mathscr A,\mathscr B}-\beta_{0,\mathscr A,\mathscr B}\bigr)(\Omega_1)
  =\sum_{j=1}^n\delta_j\ge \Delta.
\]

\paragraph{Compactness.}
Let \(\lambda:=\sigma+\sigma_0\), and write \(a:=d\sigma/d\lambda\) and \(a_0:=d\sigma_0/d\lambda\). Then \(0\le a,a_0\le 1\) \(\lambda\)-a.e. Every \(\gamma\le \sigma\) has the form \(\gamma=p\,\lambda\) with \(0\le p\le a\) \(\lambda\)-a.e., and every \(\gamma_0\le \sigma_0\) has the form \(\gamma_0=q\,\lambda\) with \(0\le q\le a_0\) \(\lambda\)-a.e. Set
\[
  \mathcal K:=\{(p,q)\in L^\infty(\lambda)\times L^\infty(\lambda): 0\le p\le a,\ 0\le q\le a_0
  \ \lambda\text{-a.e.}\}.
\]
Equip \(L^\infty(\lambda)\times L^\infty(\lambda)\) with the product weak-* topology. Since \(\lambda\) is finite, \(L^\infty(\lambda)\) is the dual of \(L^1(\lambda)\), and Banach--Alaoglu yields weak-* compactness of the unit ball. The set \(\mathcal K\) is weak-* closed, because the order constraints are equivalent to weak-* closed inequalities against nonnegative \(L^1(\lambda)\)-functions. Thus \(\mathcal K\) is weak-* compact.

For \((p,q)\in\mathcal K\), let \(\gamma:=p\,\lambda\), \(\gamma_0:=q\,\lambda\), and define signed measures
\[
  \eta_{p,q}:=(\pr_0)_\#\gamma-(\pr_0)_\#\gamma_0,\qquad
  \delta_{p,q}:=(\pr_1)_\#\gamma-(\pr_1)_\#\gamma_0.
\]
For each \(A\in\Sigma_0\) and \(B\in\Sigma_1\),
\[
  \eta_{p,q}(A)=\int_{A\times\Omega_1}(p-q)\,d\lambda,\qquad
  \delta_{p,q}(B)=\int_{\Omega_0\times B}(p-q)\,d\lambda,
\]
so these evaluations are weak-* continuous on \(\mathcal K\). Consequently the sets
\[
  C_A^-:=\{(p,q)\in\mathcal K:\eta_{p,q}(A)\ge 0\},\qquad
  C_A^+:=\{(p,q)\in\mathcal K:\eta_{p,q}(A)\le \nu^+(A)\},
\]
\[
  D_B^-:=\{(p,q)\in\mathcal K:\delta_{p,q}(B)\ge 0\},\qquad
  D_B^+:=\{(p,q)\in\mathcal K:\delta_{p,q}(B)\le (b-\psi_0)(B)\},
\]
and
\[
  T:=\{(p,q)\in\mathcal K:\delta_{p,q}(\Omega_1)\ge \Delta\}
\]
are weak-* closed.

It remains to verify the finite intersection property. Given finitely many such constraints, collect the corresponding sets from \(\Sigma_0\) and \(\Sigma_1\) into finite families \(\mathscr A\) and \(\mathscr B\). The preceding paragraph then yields measures
\(\gamma_{\mathscr A,\mathscr B}\le \sigma\) and \(\gamma_{0,\mathscr A,\mathscr B}\le \sigma_0\), hence a pair \((p,q)\in\mathcal K\), satisfying all those finitely many constraints and also
\(\delta_{p,q}(\Omega_1)\ge \Delta\). Thus every finite subfamily has nonempty intersection. By compactness, the whole family has nonempty intersection. Choose \((p,q)\) in that intersection, set
\(\gamma:=p\,\lambda\) and \(\gamma_0:=q\,\lambda\), and let \(\alpha,\alpha_0,\beta,\beta_0\) be the corresponding marginals.

Then for every \(A\in\Sigma_0\),
\[
  0\le (\alpha-\alpha_0)(A)\le \nu^+(A),
\]
hence \(0\le \alpha-\alpha_0\le \nu^+\). Likewise, for every \(B\in\Sigma_1\),
\[
  0\le (\beta-\beta_0)(B)\le (b-\psi_0)(B),
\]
so \(0\le \beta-\beta_0\le b-\psi_0\). Finally,
\[
  (\beta-\beta_0)(\Omega_1)\ge \Delta.
\]
This is \eqref{eq:target_bounds_aug_subplan}, and the proof is complete.
\end{proof}

\section{Convex Characterizations of Level-Optimal Maximin Refinements}
\label{sec:one_side}

The preceding sections identified level-optimal maximin as the candidate
one-sided invariant and developed the measurable tools needed to analyze it.
The purpose of the present section is to relate that invariant to convex
refinement. We first prove existence of one-sided minimizers for general convex
\(\vartheta\)-divergence objectives. We then show that every minimizer of finite
objective value for a strictly convex integrand is level-optimal maximin, and
conversely that every level-optimal maximin refinement is optimal for every
proper convex lower semicontinuous objective. Thus level-optimal maximin is
contained in the minimizer class of the entire convex family, while strict
convexity recovers it from the finite-valued minimizers. In this sense, the
maximin structure from Section~\ref{sec:maximin} and the variational structure
of convex refinement are two complementary descriptions of the same one-sided
refinement phenomenon, although an individual convex objective need not
characterize the level-optimal maximin class uniquely. This variational
viewpoint is adjacent to the classical divergence-projection framework of
Csisz\'ar~\cite{csiszar1975idivergence}, but the later universal comparison
results go beyond a fixed-divergence projection statement: the point is not
merely that one chosen divergence is minimized, but that the same structural
class is optimal throughout the convex family and ultimately for every
divergence satisfying channel-DPI.

\subsection{Existence of one-sided convex minimizers}
\label{sec:one_sided_convex_existence}

We begin with attainment. Before comparing convex minimizers with level-optimal maximin refinements, one must know that the one-sided convex program attains its infimum in the measurable setting. The main compactness obstruction is escape of mass in the opposite coordinate. The proof resolves this by passing to a global weak limit on a compactification of the opposite vertex space. Closedness of $\mathcal E$ ensures that the portion of the limit remaining in the original product space is still carried by $\mathcal E$; the boundary mass appears as a marginal defect and is charged at the recession rate of the integrand. Joint narrow lower semicontinuity of the finite-measure divergence gives the required lower bound, and a measurable selector repairs the defect at no greater cost than this boundary penalty.

\begin{proposition}[Attainment of the $\D_\vartheta$-closest one-sided refinement problem]
\label{prop:vartheta_one_sided_attainment}
Under Assumption~\ref{assump:closed_polish_graph_nonempty_nbhd}, fix $\iota\in\mathbb B$ and write $\bar\iota:=1-\iota$.
Let $\vartheta:[0,\infty)\to\mathbb R\cup\{+\infty\}$ be proper, convex, and lower semicontinuous.
Then there exists $\pi^\star\in\Pi_\iota^{\mathcal E}$ such that
\[
\D_\vartheta\bigl(\nu_{\bar\iota}^{(\pi^\star)} \,\|\, \nu_{\bar\iota}\bigr)
=
\inf_{\pi\in\Pi_\iota^{\mathcal E}}
\D_\vartheta\bigl(\nu_{\bar\iota}^{(\pi)} \,\|\, \nu_{\bar\iota}\bigr).
\]
In other words, the $\D_\vartheta$-closest one-sided refinement problem on side $\iota$
defined in~\eqref{eq:one_closest} attains its infimum.
\end{proposition}

\noindent \textbf{Remark.}
An example showing that the closedness of $\mathcal{E}$
is genuinely needed for the existence theory is given in Appendix~\ref{app:closed-support-needed}.

\begin{proof}
By symmetry, it suffices to treat the case \(\iota=0\). Set
\[
J(\pi):=\D_\vartheta\bigl((\pr_1)_\#\pi\|\nu_1\bigr),
\qquad
\pi\in\Pi_0^{\mathcal E},
\]
and write \(w_0:=\nu_0(\Omega_0)>0\).

\paragraph{Normalization and a feasible selector.}
Since \(\vartheta\) is proper, convex, and lower semicontinuous on
\([0,\infty)\), it admits an affine minorant: there exist \(a,b\in\mathbb R\)
such that \(\vartheta(t)\ge at+b\) for every \(t\ge0\). Define
\[
\widetilde\vartheta(t):=\vartheta(t)-at-b.
\]
Then \(\widetilde\vartheta\) is nonnegative, proper, convex, and lower
semicontinuous, and
\[
\D_{\widetilde\vartheta}(P\|\nu_1)
=
\D_\vartheta(P\|\nu_1)
-aP(\Omega_1)-b\nu_1(\Omega_1)
\]
for every \(P\in\mathcal M_+(\Omega_1)\); here
\(\widetilde\vartheta'_\infty=\vartheta'_\infty-a\). Since every feasible
payload has total mass \(w_0\), the two objectives differ by a constant on
\(\Pi_0^{\mathcal E}\). Relabelling \(\widetilde\vartheta\) and the
corresponding objective as \(\vartheta\) and \(J\), respectively, we may
therefore assume that \(\vartheta\ge0\). In particular,
\[
c:=\vartheta'_\infty\in[0,+\infty].
\]

We first verify that \(\Pi_0^{\mathcal E}\neq\varnothing\). Let
\(\overline{\Sigma}_0^{\,\nu_0}\) be the \(\nu_0\)-completion of
\(\Sigma_0\), and define
\[
F(x):=N_1(x)=\{y\in\Omega_1:(x,y)\in\mathcal E\}.
\]
Each \(F(x)\) is nonempty and closed. Moreover, for every open
\(V\subseteq\Omega_1\),
\[
\{x\in\Omega_0:F(x)\cap V\neq\varnothing\}
=
\pr_0\bigl(\mathcal E\cap(\Omega_0\times V)\bigr)
\]
is analytic, hence \(\overline{\Sigma}_0^{\,\nu_0}\)-measurable. The
Kuratowski--Ryll-Nardzewski theorem therefore yields an
\(\overline{\Sigma}_0^{\,\nu_0}\)-measurable selector
\(\mathfrak s:\Omega_0\to\Omega_1\) such that
\((x,\mathfrak s(x))\in\mathcal E\) for every \(x\). By the standard
Borel-version theorem, \(\mathfrak s\) has a \(\Sigma_0\)-measurable version
\(s\). Hence
\[
(x,s(x))\in\mathcal E
\qquad
\text{for }\nu_0\text{-a.e. }x,
\]
and consequently
\[
\pi^0:=(\id_{\Omega_0},s)_\#\nu_0
\]
belongs to \(\Pi_0^{\mathcal E}\).

Let
\[
J^\star:=\inf_{\pi\in\Pi_0^{\mathcal E}}J(\pi).
\]
If \(J^\star=+\infty\), then \(\pi^0\) is a minimizer. Assume henceforth that
\(J^\star<+\infty\), and choose \((\pi_n)\subseteq\Pi_0^{\mathcal E}\) such
that
\[
J(\pi_n)<+\infty
\quad\text{for every }n,
\qquad
J(\pi_n)\longrightarrow J^\star.
\]
Set \(P_n:=(\pr_1)_\#\pi_n\). Then \(P_n(\Omega_1)=w_0\) for every \(n\).
Moreover,
\[
\dom\vartheta\cap(0,\infty)\neq\varnothing.
\]
Indeed, otherwise \(\dom\vartheta=\{0\}\), in which case every measure of
positive mass has infinite \(\vartheta\)-divergence, contrary to
\(J(\pi_n)<+\infty\).

\paragraph{Compactified limit and boundary defect.}
Choose a compact metric space \(K\) in which \(\Omega_1\) is homeomorphic
to a \(G_\delta\)-subset, that is, a countable intersection of open subsets
of \(K\), and identify \(\Omega_1\) with that subset.
Let
\(\overline\nu_1\in\mathcal M_+(K)\) be the extension of \(\nu_1\) by zero
to \(K\setminus\Omega_1\).

Regard each \(\pi_n\) as a measure on \(\Omega_0\times K\). The family
\((\pi_n)\) is tight: for every \(\varepsilon>0\), choose a compact
\(C_\varepsilon\subseteq\Omega_0\) with
\(\nu_0(C_\varepsilon^\c)<\varepsilon\); then
\[
\pi_n\bigl((C_\varepsilon\times K)^\c\bigr)
=
\nu_0(C_\varepsilon^\c)
<
\varepsilon
\qquad
\text{for every }n.
\]
By Prokhorov's theorem, after passing to a subsequence,
\[
\pi_n\Rightarrow\widehat\pi
\qquad
\text{in }\mathcal M_+(\Omega_0\times K)
\]
for some finite measure \(\widehat\pi\). Continuity of the first projection
gives
\[
(\pr_0)_\#\widehat\pi=\nu_0.
\]

Let
\[
\widehat{\mathcal E}
:=
\overline{\mathcal E}^{\,\Omega_0\times K}.
\]
Since \(\widehat{\mathcal E}\) is closed and
\(\pi_n(\widehat{\mathcal E})=w_0\), the Portmanteau theorem yields
\[
\widehat\pi(\widehat{\mathcal E})
\ge
\limsup_{n\to\infty}\pi_n(\widehat{\mathcal E})
=
w_0.
\]
As \(\widehat\pi(\Omega_0\times K)=w_0\), it follows that
\(\widehat\pi(\widehat{\mathcal E}^{\,\c})=0\). Furthermore,
\[
\widehat{\mathcal E}\cap(\Omega_0\times\Omega_1)
=
\overline{\mathcal E}^{\,\Omega_0\times\Omega_1}
=
\mathcal E,
\]
because \(\mathcal E\) is closed in \(\Omega_0\times\Omega_1\).

Set \(K_\partial:=K\setminus\Omega_1\), and define
\[
\pi_\infty
:=
\widehat\pi\!\restriction(\Omega_0\times\Omega_1),
\qquad
\pi^\partial
:=
\widehat\pi\!\restriction(\Omega_0\times K_\partial).
\]
We view these restrictions by zero extension as measures on
\(\Omega_0\times K\); when using \(\pr_1\), we regard \(\pi_\infty\) as a
measure on \(\Omega_0\times\Omega_1\). Then
\[
\widehat\pi=\pi_\infty+\pi^\partial,
\qquad
\pi_\infty(\mathcal E^\c)=0.
\]
Define
\[
\mu_\infty:=(\pr_0)_\#\pi_\infty,
\qquad
\nu_{\mathrm{def}}:=(\pr_0)_\#\pi^\partial.
\]
Since \((\pr_0)_\#\widehat\pi=\nu_0\),
\[
\nu_0=\mu_\infty+\nu_{\mathrm{def}}.
\]

\paragraph{Lower semicontinuity.}
Let
\[
P_\infty:=(\pr_1)_\#\pi_\infty\in\mathcal M_+(\Omega_1),
\]
and denote by \(\overline P_\infty\) its extension by zero to \(K\). Let
\(\pr_K:\Omega_0\times K\to K\) be the second projection, and set
\[
\widehat P:=(\pr_K)_\#\widehat\pi,
\qquad
P^\partial:=(\pr_K)_\#\pi^\partial.
\]
Identifying each \(P_n\) with its zero extension to \(K\), continuity of
\(\pr_K\) gives
\[
P_n\Rightarrow\widehat P
\qquad
\text{in }\mathcal M_+(K).
\]
Moreover,
\[
\widehat P=\overline P_\infty+P^\partial,
\qquad
P^\partial(K)=\nu_{\mathrm{def}}(\Omega_0),
\]
and \(P^\partial\perp\overline\nu_1\).

Write the Lebesgue decomposition of \(P_\infty\) with respect to \(\nu_1\)
as
\[
P_\infty=r\,\nu_1+P_\infty^\perp.
\]
Then the Lebesgue decomposition of \(\widehat P\) with respect to
\(\overline\nu_1\) is
\[
\widehat P
=
r\,\overline\nu_1
+
\overline{P_\infty^\perp}
+
P^\partial,
\]
where \(\overline{P_\infty^\perp}\) denotes the zero extension of
\(P_\infty^\perp\) to \(K\). Hence
\[
\D_\vartheta(\widehat P\|\overline\nu_1)
=
\D_\vartheta(P_\infty\|\nu_1)
+
c\,\nu_{\mathrm{def}}(\Omega_0),
\]
with the convention \(+\infty\cdot0=0\).

The normalized integrand \(\vartheta\) is nonnegative, convex, and lower
semicontinuous, and its effective domain meets \((0,\infty)\). Thus the
joint narrow lower semicontinuity theorem for finite-measure relative
entropy functionals applies; see
\cite[Corollary~2.9]{LieroMielkeSavare2018}. Applying it to
\[
(P_n,\overline\nu_1)\longrightarrow(\widehat P,\overline\nu_1)
\]
gives
\[
\liminf_{n\to\infty}\D_\vartheta(P_n\|\overline\nu_1)
\ge
\D_\vartheta(\widehat P\|\overline\nu_1).
\]
By invariance under zero extension,
\[
\D_\vartheta(P_n\|\overline\nu_1)
=
\D_\vartheta(P_n\|\nu_1).
\]
Therefore
\begin{equation}
\label{eq:mass_loss_liminf}
\liminf_{n\to\infty}\D_\vartheta(P_n\|\nu_1)
\ge
\D_\vartheta(P_\infty\|\nu_1)
+
c\,\nu_{\mathrm{def}}(\Omega_0).
\end{equation}

\paragraph{Repair and conclusion.}
Since \(\nu_{\mathrm{def}}\le\nu_0\), the selector \(s\) satisfies
\[
(x,s(x))\in\mathcal E
\qquad
\text{for }\nu_{\mathrm{def}}\text{-a.e. }x.
\]
Define
\[
\pi_{\mathrm{rep}}
:=
(\id_{\Omega_0},s)_\#\nu_{\mathrm{def}},
\qquad
\pi^\star:=\pi_\infty+\pi_{\mathrm{rep}}.
\]
Both measures are concentrated on \(\mathcal E\), and
\[
(\pr_0)_\#\pi^\star
=
\mu_\infty+\nu_{\mathrm{def}}
=
\nu_0.
\]
Thus \(\pi^\star\in\Pi_0^{\mathcal E}\).

If \(c=+\infty\), then \eqref{eq:mass_loss_liminf} and
\(J(\pi_n)\to J^\star<+\infty\) imply
\(\nu_{\mathrm{def}}(\Omega_0)=0\). Hence
\(\pi^\star=\pi_\infty\), and
\[
J(\pi^\star)
\le
\liminf_{n\to\infty}J(\pi_n)
=
J^\star.
\]
Feasibility gives the reverse inequality, so \(J(\pi^\star)=J^\star\).

Assume now that \(c<+\infty\). For every
\(P,M\in\mathcal M_+(\Omega_1)\),
\begin{equation}
\label{eq:add_mass_bound_clean}
\D_\vartheta(P+M\|\nu_1)
\le
\D_\vartheta(P\|\nu_1)+cM(\Omega_1).
\end{equation}
Indeed, write
\[
P=r\,\nu_1+P^\perp,
\qquad
M=u\,\nu_1+M^\perp.
\]
Convexity and the definition of the recession slope give
\[
\vartheta(s+t)\le\vartheta(s)+ct
\qquad
(s,t\ge0)
\]
in the extended-real sense. Integrating this inequality and adding the
singular contribution proves \eqref{eq:add_mass_bound_clean}.

Set \(M:=(\pr_1)_\#\pi_{\mathrm{rep}}\). Then
\[
M(\Omega_1)=\nu_{\mathrm{def}}(\Omega_0),
\qquad
(\pr_1)_\#\pi^\star=P_\infty+M.
\]
Using \eqref{eq:add_mass_bound_clean} and
\eqref{eq:mass_loss_liminf}, we obtain
\[
\begin{aligned}
J(\pi^\star)
&=
\D_\vartheta(P_\infty+M\|\nu_1) \\
&\le
\D_\vartheta(P_\infty\|\nu_1)
+
c\,\nu_{\mathrm{def}}(\Omega_0) \\
&\le
\liminf_{n\to\infty}J(\pi_n)
=
J^\star.
\end{aligned}
\]
Again, feasibility gives the reverse inequality. Hence
\(J(\pi^\star)=J^\star\), completing the proof.
\end{proof}

\subsection{Finite-valued strictly convex minimizers are level-optimal maximin}
\label{sec:strict_convex_implies_level_opt_maximin}

We now prove the converse direction from convex optimization back to the
one-sided extremal criterion of Section~\ref{sec:maximin}. The converse
requires a genuine perturbative argument. Starting from a finite-valued
minimizer for a strictly convex integrand, we use the overflow decomposition
from Lemma~\ref{lem:overflow_decomposition} and the measure-theoretic
augmenting-subplan lemma,
Lemma~\ref{lem:measure_theoretic_augmenting_subplan}, to show that any failure
of level-optimal maximin at some threshold \(t\) would yield a feasible spliced
perturbation with strictly smaller \(\D_\vartheta\)-value. This is where the
maximin structure and the convex variational structure come together.

\begin{theorem}[Finite-valued strictly convex minimizers are level-optimal maximin]
\label{thm:strict_convex_implies_level_opt_maximin}
Work under Assumption~\ref{assump:closed_polish_graph_nonempty_nbhd}.
Fix \(\iota\in\mathbb B\), and write \(\bar\iota:=1-\iota\).
Let \(\vartheta:[0,\infty)\to\mathbb R\cup\{+\infty\}\) be proper, lower semicontinuous, and
strictly convex on \(\dom\vartheta\).

Let \(\pi^\star\in\Pi_\iota^{\mathcal E}\) be a minimizer of the
\(\D_\vartheta\)-closest one-sided refinement problem, and assume that
\[
\D_\vartheta\bigl(\nu_{\bar\iota}^{(\pi^\star)}\,\|\,\nu_{\bar\iota}\bigr)
=
\inf_{\pi\in\Pi_\iota^{\mathcal E}}
\D_\vartheta\bigl(\nu_{\bar\iota}^{(\pi)}\,\|\,\nu_{\bar\iota}\bigr)
<+\infty.
\]
Then \(\pi^\star\) is level-optimal maximin, in the sense of
Definition~\ref{defn:level_opt_maximin}.

In particular, level-optimal maximin refinements exist.
\end{theorem}

\noindent \textbf{Remark.}
The assumption that the optimal value is finite is needed because the proof
uses strict convexity to obtain a strict decrease in \(\D_\vartheta\); this
mechanism is meaningful only on the effective domain of the divergence, not
when all values under comparison are \(+\infty\).

\begin{proof}
Set \(P:=(\pr_{\bar\iota})_\#\pi^\star\), and write the Lebesgue decomposition of \(P\) with respect to \(\nu_{\bar\iota}\) as
\[
  P=r\,\nu_{\bar\iota}+P^\perp,
  \qquad
  r\in L^1(\nu_{\bar\iota}),
  \qquad
  P^\perp\perp \nu_{\bar\iota}.
\]
We must prove that
\[
  P_t^{(\pi^\star)}(\Omega_{\bar\iota})=\Fit_{\bar\iota}(t)
  \qquad\text{for every }t\ge 0,
\]
where \(P_t^{(\pi^\star)}:=(r\wedge t)\,\nu_{\bar\iota}\).

The case \(t=0\) is immediate, since \(\Feas_{\bar\iota}(0)=\{0\}\). Fix \(t>0\).

\paragraph{Truncation at level \(t\).}
Apply Lemma~\ref{lem:overflow_decomposition} at level \(t\). Then
\(\pi^\star=\sigma_0+\pi_{\mathrm{ov}}\) with \(\sigma_0\in\Feas_{\bar\iota}(t)\), and
\begin{equation}
\label{eq:strict_convex_split}
  (\pr_{\bar\iota})_\#\sigma_0=(r\wedge t)\,\nu_{\bar\iota},
  \qquad
  (\pr_{\bar\iota})_\#\pi_{\mathrm{ov}}
  =
  P^\perp+(r-t)_+\,\nu_{\bar\iota}.
\end{equation}
Hence \(\sigma_0(\Omega_0\times\Omega_1)=P_t^{(\pi^\star)}(\Omega_{\bar\iota})\), so
\begin{equation}
\label{eq:strict_convex_easy_dir}
  \Fit_{\bar\iota}(t)\ge P_t^{(\pi^\star)}(\Omega_{\bar\iota}).
\end{equation}

We prove the reverse inequality. Assume that it fails. Then there exists \(\sigma\in \Feas_{\bar\iota}(t)\) such that
\(\sigma(\Omega_0\times\Omega_1)>\sigma_0(\Omega_0\times\Omega_1)\). Write
\[
  \xi:=(\pr_\iota)_\#\sigma,
  \qquad
  \xi_0:=(\pr_\iota)_\#\sigma_0,
  \qquad
  \psi:=(\pr_{\bar\iota})_\#\sigma,
  \qquad
  \psi_0:=(\pr_{\bar\iota})_\#\sigma_0=(r\wedge t)\,\nu_{\bar\iota}.
\]

\paragraph{Boundary levels of \(\dom\vartheta\).}
These require no slope argument. Since \(\D_\vartheta(P\|\nu_{\bar\iota})<\infty\), one has \(r\in\dom\vartheta\) for
\(\nu_{\bar\iota}\)-a.e.\ point.

If \(t\le \inf\dom\vartheta\), then \(r\ge t\) \(\nu_{\bar\iota}\)-a.e., so \(\psi_0=t\,\nu_{\bar\iota}\). Applying
Lemma~\ref{lem:measure_theoretic_augmenting_subplan} with \(b=t\,\nu_{\bar\iota}\) would then produce a nonzero measure
\(\beta-\beta_0\) dominated by \(b-\psi_0=0\), which is impossible.

If \(\sup\dom\vartheta<\infty\) and \(t\ge \sup\dom\vartheta\), then \(\vartheta'_\infty=+\infty\). Since
\(\D_\vartheta(P\|\nu_{\bar\iota})<\infty\), we have \(P^\perp=0\) and \(r\le \sup\dom\vartheta\le t\)
\(\nu_{\bar\iota}\)-a.e. Hence \((r-t)_+=0\), so \(\pi_{\mathrm{ov}}=0\) in \eqref{eq:strict_convex_split}. Thus
\(\sigma_0=\pi^\star\), and every \(\sigma\in\Feas_{\bar\iota}(t)\) has total mass at most
\(\nu_\iota(\Omega_\iota)=\sigma_0(\Omega_0\times\Omega_1)\), again a contradiction.

We may therefore assume \(t\in \operatorname{int}(\dom\vartheta)\).

\paragraph{An augmenting exchange.}
Apply Lemma~\ref{lem:measure_theoretic_augmenting_subplan} with \(b=t\,\nu_{\bar\iota}\). Since
\[
  \frac{d\psi_0}{d(t\,\nu_{\bar\iota})}=\frac{r\wedge t}{t}
  \qquad t\,\nu_{\bar\iota}\text{-a.e.},
\]
the set \(L\) from that lemma agrees \(t\,\nu_{\bar\iota}\)-a.e. with \(\{r<t\}\). Hence there exist
\(\gamma,\gamma_0\in\mathcal M_+(\Omega_0\times\Omega_1)\) such that \(\gamma\le \sigma\) and \(\gamma_0\le \sigma_0\), and,
writing
\[
  \alpha:=(\pr_\iota)_\#\gamma,
  \qquad
  \alpha_0:=(\pr_\iota)_\#\gamma_0,
  \qquad
  \beta:=(\pr_{\bar\iota})_\#\gamma,
  \qquad
  \beta_0:=(\pr_{\bar\iota})_\#\gamma_0,
\]
one has
\begin{align}
\label{eq:strict_convex_alpha}
  &\alpha_0\le \alpha,
  \qquad
  0\neq \alpha-\alpha_0\le (\xi-\xi_0)_+,
  \\
\label{eq:strict_convex_beta}
  &\beta_0\le \beta,
  \qquad
  R:=\beta-\beta_0\neq 0,
  \qquad
  R \text{ is carried by } \{r<t\},
  \\
\label{eq:strict_convex_residual_feas}
  &\psi_0-\beta_0+\beta\le t\,\nu_{\bar\iota}.
\end{align}
Since \(R\le \beta\le t\,\nu_{\bar\iota}\), we may write \(R=s\,\nu_{\bar\iota}\) with
\(s\in L^1_+(\nu_{\bar\iota})\), carried by \(\{r<t\}\). Set
\[
  \delta:=R(\Omega_{\bar\iota})=\int s\,d\nu_{\bar\iota}>0.
\]

\paragraph{Using the overflow to supply the first marginal.}
Let \(\rho:=(\pr_\iota)_\#\pi_{\mathrm{ov}}\). Since \((\pr_\iota)_\#\pi^\star=\nu_\iota\) and
\(\pi^\star=\sigma_0+\pi_{\mathrm{ov}}\), we have \(\rho=\nu_\iota-\xi_0\). By \eqref{eq:strict_convex_alpha},
\[
  0\le \alpha-\alpha_0\le (\xi-\xi_0)_+\le \nu_\iota-\xi_0=\rho.
\]

Choose a disintegration \(\pi_{\mathrm{ov}}=\rho\kprod G_{\mathrm{ov}}\), and define
\(\lambda:=(\alpha-\alpha_0)\kprod G_{\mathrm{ov}}\). Then \(\lambda\le \pi_{\mathrm{ov}}\) and
\((\pr_\iota)_\#\lambda=\alpha-\alpha_0\). Let \(Q:=(\pr_{\bar\iota})_\#\lambda\). From
\(\lambda\le \pi_{\mathrm{ov}}\) and \eqref{eq:strict_convex_split},
\begin{equation}
\label{eq:strict_convex_Q_bound}
  Q\le P^\perp+(r-t)_+\,\nu_{\bar\iota}.
\end{equation}
Write
\[
  Q=q\,\nu_{\bar\iota}+Q^\perp,
  \qquad
  Q^\perp\perp \nu_{\bar\iota}.
\]
Then \(0\le q\le (r-t)_+\) \(\nu_{\bar\iota}\)-a.e., so \(q\) is carried by \(\{r>t\}\), and
\begin{equation}
\label{eq:strict_convex_Qperp}
  Q^\perp\le P^\perp.
\end{equation}
Moreover,
\[
  Q(\Omega_{\bar\iota})
  =
  \lambda(\Omega_0\times\Omega_1)
  =
  (\alpha-\alpha_0)(\Omega_\iota)
  =
  (\beta-\beta_0)(\Omega_{\bar\iota})
  =
  \delta.
\]

\paragraph{The spliced competitor.}
Define \(\widehat\pi:=\pi^\star-\lambda+\gamma-\gamma_0\). Because \(\gamma_0\le \sigma_0\),
\(\lambda\le \pi_{\mathrm{ov}}\), and \(\pi^\star=\sigma_0+\pi_{\mathrm{ov}}\), this is a nonnegative measure
carried by \(\mathcal E\). Its \(\iota\)-marginal is
\[
  (\pr_\iota)_\#\widehat\pi
  =
  \nu_\iota-(\alpha-\alpha_0)+\alpha-\alpha_0
  =
  \nu_\iota,
\]
so \(\widehat\pi\in\Pi_\iota^{\mathcal E}\). Writing \(\widehat P:=(\pr_{\bar\iota})_\#\widehat\pi\), we have
\[
  \widehat P=P-Q+R=(r-q+s)\,\nu_{\bar\iota}+(P^\perp-Q^\perp).
\]

\paragraph{Location of the moved mass.}
Since \(\beta\le t\,\nu_{\bar\iota}\) and \(\beta_0\le (r\wedge t)\,\nu_{\bar\iota}\), both are absolutely
continuous with respect to \(\nu_{\bar\iota}\); write
\[
  \beta=u\,\nu_{\bar\iota},
  \qquad
  \beta_0=u_0\,\nu_{\bar\iota},
\]
so \(s=u-u_0\) \(\nu_{\bar\iota}\)-a.e. From \eqref{eq:strict_convex_residual_feas},
\[
  (r\wedge t)-u_0+u\le t
  \qquad \nu_{\bar\iota}\text{-a.e.}
\]
Hence, on \(\{s>0\}\subseteq\{r<t\}\),
\begin{equation}
\label{eq:strict_convex_added_side}
  r+s\le t.
\end{equation}
On the other hand, \(0\le q\le (r-t)_+\) gives, on \(\{q>0\}\),
\begin{equation}
\label{eq:strict_convex_removed_side}
  r-q\ge t.
\end{equation}
Therefore \(\{s>0\}\cap\{q>0\}=\varnothing\).

\paragraph{Strict decrease of the objective.}
Set
\[
  m_-:=\vartheta'_-(t),
  \qquad
  m_+:=\vartheta'_+(t),
  \qquad
  c:=\vartheta'_\infty\in(-\infty,+\infty].
\]
Since \(t\in\operatorname{int}(\dom\vartheta)\), both \(m_-\) and \(m_+\) are finite. Strict convexity implies
that, whenever \(u<t\), \(h>0\), and \(u+h\le t\),
\begin{equation}
\label{eq:strict_convex_add_est}
  \vartheta(u+h)-\vartheta(u)<m_-\,h,
\end{equation}
and whenever \(v>t\), \(h>0\), and \(v-h\ge t\),
\begin{equation}
\label{eq:strict_convex_remove_est}
  \vartheta(v-h)-\vartheta(v)<-\,m_+\,h.
\end{equation}
Using \eqref{eq:strict_convex_added_side} and \eqref{eq:strict_convex_add_est},
\begin{equation}
\label{eq:strict_convex_add_integral}
  \int \bigl[\vartheta(r+s)-\vartheta(r)\bigr]\,d\nu_{\bar\iota}
  <m_-\,\delta.
\end{equation}
Using \eqref{eq:strict_convex_removed_side} and \eqref{eq:strict_convex_remove_est},
\begin{equation}
\label{eq:strict_convex_remove_integral}
  \int \bigl[\vartheta(r-q)-\vartheta(r)\bigr]\,d\nu_{\bar\iota}
  \le -\,m_+\,q(\Omega_{\bar\iota}),
\end{equation}
with strict inequality if \(q(\Omega_{\bar\iota})>0\). Since \(\{s>0\}\cap\{q>0\}=\varnothing\),
\[
  \vartheta(r-q+s)-\vartheta(r)
  =
  \bigl[\vartheta(r+s)-\vartheta(r)\bigr]
  +
  \bigl[\vartheta(r-q)-\vartheta(r)\bigr]
  \qquad \nu_{\bar\iota}\text{-a.e.}
\]

If \(c=+\infty\), then \(\D_\vartheta(P\|\nu_{\bar\iota})<\infty\) forces \(P^\perp=0\), hence
\(Q^\perp=0\) by \eqref{eq:strict_convex_Qperp}. Therefore \(\delta=q(\Omega_{\bar\iota})\), and
\[
  \D_\vartheta(\widehat P\|\nu_{\bar\iota})-\D_\vartheta(P\|\nu_{\bar\iota})
  =
  \int \bigl[\vartheta(r-q+s)-\vartheta(r)\bigr]\,d\nu_{\bar\iota}
  <
  (m_--m_+)\delta
  \le 0.
\]
Hence \(\D_\vartheta(\widehat P\|\nu_{\bar\iota})<\D_\vartheta(P\|\nu_{\bar\iota})\).

Assume now that \(c<+\infty\). Since \(c<\infty\), \(\dom\vartheta\) contains a ray \([T,\infty)\): otherwise
\(\vartheta(u)=+\infty\) along some sequence \(u\to\infty\), which would force \(\vartheta(u)/u\to+\infty\),
contrary to the definition of \(c\). On that ray the map \(u\mapsto \vartheta'_+(u)\) is nondecreasing and
converges to \(c\). If \(m_+=c\), then monotonicity gives \(\vartheta'_+(u)=c\) for every \(u\ge t\), forcing
\(\vartheta\) to be affine on \([t,\infty)\), contrary to strict convexity on \(\dom\vartheta\). Thus
\(m_+<c\), and therefore \(m_-<c\) as well.

Using the singular-part formula for \(\D_\vartheta\),
\[
  \D_\vartheta(M\|\nu_{\bar\iota})
  =
  \int \vartheta(u)\,d\nu_{\bar\iota}+c\,M^\perp(\Omega_{\bar\iota}),
  \qquad
  M=u\,\nu_{\bar\iota}+M^\perp,
\]
we obtain
\[
  \D_\vartheta(\widehat P\|\nu_{\bar\iota})-\D_\vartheta(P\|\nu_{\bar\iota})
  <
  m_-\,\delta-m_+\,q(\Omega_{\bar\iota})-c\,Q^\perp(\Omega_{\bar\iota}).
\]
Since \(\delta=q(\Omega_{\bar\iota})+Q^\perp(\Omega_{\bar\iota})\),
\[
  \D_\vartheta(\widehat P\|\nu_{\bar\iota})-\D_\vartheta(P\|\nu_{\bar\iota})
  <
  (m_--m_+)\,q(\Omega_{\bar\iota})+(m_--c)\,Q^\perp(\Omega_{\bar\iota}).
\]

Since \(m_-\le m_+\) and \(m_-<c\), the right-hand side is nonpositive.
The preceding strict inequality therefore gives
\[
  \D_\vartheta(\widehat P\|\nu_{\bar\iota})
  <
  \D_\vartheta(P\|\nu_{\bar\iota}).
\]

In both cases \(\widehat\pi\in\Pi_\iota^{\mathcal E}\) is feasible and has strictly smaller
\(\D_\vartheta\)-value than \(\pi^\star\), contradicting optimality.

Therefore no \(\sigma\in\Feas_{\bar\iota}(t)\) can have mass exceeding
\[
  \sigma_0(\Omega_0\times\Omega_1)=P_t^{(\pi^\star)}(\Omega_{\bar\iota}),
\]
so \(\Fit_{\bar\iota}(t)\le P_t^{(\pi^\star)}(\Omega_{\bar\iota})\). Together with
\eqref{eq:strict_convex_easy_dir}, this yields
\[
  \Fit_{\bar\iota}(t)=P_t^{(\pi^\star)}(\Omega_{\bar\iota})
  \qquad\text{for every }t\ge 0.
\]
Thus \(\pi^\star\) is level-optimal maximin.

For existence, apply the theorem with \(\vartheta(s)=e^{-s}\).
This integrand is bounded and strictly convex, so
Proposition~\ref{prop:vartheta_one_sided_attainment} provides a minimizer of
finite objective value.
\end{proof}

\subsection{Level-optimal maximin is optimal for every convex objective}
\label{sec:level_opt_maximin_implies_universal_convex_opt}

We now prove the converse implication.  Once the one-sided overflow order from
Section~\ref{sec:maximin} is available, the comparison of convex objectives becomes essentially
formal: level-optimal maximin gives pointwise domination of overflow profiles, and the hinge
representation from Lemma~\ref{lem:convex_hinge_representation_for_divergence} integrates that
order into optimality for every proper convex lower semicontinuous \(\vartheta\).  This is the step
that shows the entire convex family selects the same one-sided object.

\begin{theorem}[Level-optimal maximin implies
\texorpdfstring{$\vartheta$}{theta}-optimality]
\label{thm:level_opt_maximin_implies_theta_opt}
Work under Assumption~\ref{assump:closed_polish_graph_nonempty_nbhd}.
Fix \(\iota\in\mathbb B\), and write \(\bar\iota:=1-\iota\).
Let
\(\vartheta:[0,\infty)\to\mathbb R\cup\{+\infty\}\)
be proper, convex, and lower semicontinuous.

If \(\pi^\star\in\Pi_\iota^{\mathcal E}\) is level-optimal maximin from side
\(\iota\) to side \(\bar\iota\), then
\[
  \D_\vartheta
  \bigl(
    \nu_{\bar\iota}^{(\pi^\star)}
    \,\|\,
    \nu_{\bar\iota}
  \bigr)
  =
  \inf_{\pi\in\Pi_\iota^{\mathcal E}}
  \D_\vartheta
  \bigl(
    \nu_{\bar\iota}^{(\pi)}
    \,\|\,
    \nu_{\bar\iota}
  \bigr).
\]
Thus \(\pi^\star\) is optimal for the
\(\D_\vartheta\)-closest one-sided refinement problem.
\end{theorem}

\begin{remark}
No strict convexity is required in
Theorem~\ref{thm:level_opt_maximin_implies_theta_opt}: level-optimal maximin
implies optimality for every proper lower semicontinuous convex integrand.
Strict convexity enters only in the converse direction, where it is used to
show that a finite-valued \(\D_\vartheta\)-minimizer must be level-optimal
maximin.

The theorem may also be viewed as a continuous-majorization statement for
the feasible opposite marginals relative to \(\nu_{\bar\iota}\). For a
feasible refinement \(\pi\), the quantity \(\Over^\pi(t)\) is the analogue
of the integrated tail, or stop-loss, transform of the density
\(r=dP^{\ac}/d\nu_{\bar\iota}\), together with the singular mass
\(P^\perp(\Omega_{\bar\iota})\). Thus
\[
  \Over^{\pi^\star}(t)\le \Over^\pi(t),
  \qquad t\ge0,
\]
places the opposite marginal of a level-optimal maximin refinement below
every other feasible opposite marginal in the corresponding relative
convex order. The passage from this pointwise overflow order to arbitrary
convex integrands is obtained by applying the hinge representation to
bounded-slope Fenchel approximants and then passing to the limit. This is
closely related to continuous majorization and convex-order comparison on
measure spaces; see, for example,
\cite{marshall2011inequalities,ryff1970measure}.
\end{remark}

\begin{proof}
Let \(\pi\in\Pi_\iota^{\mathcal E}\) be arbitrary, and write
\[
  P^\star:=(\pr_{\bar\iota})_\#\pi^\star,
  \qquad
  P:=(\pr_{\bar\iota})_\#\pi.
\]
Let
\[
  P^\star=r^\star\,\nu_{\bar\iota}+P^\star_\perp,
  \qquad
  P=r\,\nu_{\bar\iota}+P_\perp
\]
be their Lebesgue decompositions with respect to
\(\nu_{\bar\iota}\), and set
\[
  P_t^{(\pi^\star)}
  :=
  (r^\star\wedge t)\,\nu_{\bar\iota},
  \qquad
  P_t^{(\pi)}
  :=
  (r\wedge t)\,\nu_{\bar\iota},
  \qquad t\ge0.
\]
Since
\((\pr_\iota)_\#\pi^\star=(\pr_\iota)_\#\pi=\nu_\iota\),
the opposite marginals have the same total mass:
\[
  P^\star(\Omega_{\bar\iota})
  =
  P(\Omega_{\bar\iota})
  =
  \nu_\iota(\Omega_\iota).
\]

\paragraph{Overflow domination.}
Fix \(t\ge0\). By
Lemma~\ref{lem:overflow_decomposition}, applied to \(\pi\), there exists
\(\sigma_t\in\Feas_{\bar\iota}(t)\) such that
\[
  (\pr_{\bar\iota})_\#\sigma_t=P_t^{(\pi)}.
\]
Consequently,
\[
  P_t^{(\pi)}(\Omega_{\bar\iota})
  =
  \sigma_t(\Omega_0\times\Omega_1)
  \le
  \Fit_{\bar\iota}(t).
\]
Since \(\pi^\star\) is level-optimal maximin,
\[
  P_t^{(\pi^\star)}(\Omega_{\bar\iota})
  =
  \Fit_{\bar\iota}(t),
\]
and hence
\begin{equation}
\label{eq:trunc_order_clean}
  P_t^{(\pi^\star)}(\Omega_{\bar\iota})
  \ge
  P_t^{(\pi)}(\Omega_{\bar\iota}),
  \qquad t\ge0.
\end{equation}

Moreover,
\[
  \Over^{\pi^\star}(t)
  =
  P^\star(\Omega_{\bar\iota})
  -
  P_t^{(\pi^\star)}(\Omega_{\bar\iota}),
\]
and likewise
\[
  \Over^\pi(t)
  =
  P(\Omega_{\bar\iota})
  -
  P_t^{(\pi)}(\Omega_{\bar\iota}).
\]
Since \(P^\star\) and \(P\) have the same total mass,
\eqref{eq:trunc_order_clean} yields
\begin{equation}
\label{eq:overflow_order_clean}
  \Over^{\pi^\star}(t)
  \le
  \Over^\pi(t),
  \qquad t\ge0.
\end{equation}

\paragraph{Bounded-slope approximation.}
Extend \(\vartheta\) to \(\mathbb R\) by setting
\(\vartheta(u):=+\infty\) for \(u<0\), and let
\[
  \vartheta^*(s)
  :=
  \sup_{u\in\mathbb R}\{su-\vartheta(u)\},
  \qquad s\in\mathbb R.
\]
Choose \(a_0\in\dom\vartheta^*\). For \(n\in\mathbb N\), set
\[
  I_n:=[a_0-n,a_0+n]
\]
and define
\[
  \vartheta_n(u)
  :=
  \sup_{s\in I_n}
  \{su-\vartheta^*(s)\},
  \qquad u\ge0.
\]

Each \(\vartheta_n\) is finite-valued, convex, and lower semicontinuous on
\([0,\infty)\). Indeed, \(\vartheta^*\) is bounded below on the compact set
\(I_n\), while \(a_0\in I_n\cap\dom\vartheta^*\). Moreover,
\(\vartheta_n\) is Lipschitz, since it is a supremum of affine functions
whose slopes belong to \(I_n\). In particular,
\[
  (\vartheta_n)'_+(0)\in\mathbb R,
  \qquad
  c_n:=(\vartheta_n)'_\infty\in\mathbb R.
\]
Thus
Lemma~\ref{lem:convex_hinge_representation_for_divergence}
applies to \(\vartheta_n\). Writing \(\mu_n\) for the corresponding finite
nonnegative Stieltjes measure, we obtain
\[
  \D_{\vartheta_n}(P^\star\|\nu_{\bar\iota})
  =
  C_n
  +
  \int_{(0,\infty)}
  \Over^{\pi^\star}(t)\,\mu_n(dt)
\]
and
\[
  \D_{\vartheta_n}(P\|\nu_{\bar\iota})
  =
  C_n
  +
  \int_{(0,\infty)}
  \Over^\pi(t)\,\mu_n(dt),
\]
where
\[
  C_n
  :=
  \vartheta_n(0)\,\nu_{\bar\iota}(\Omega_{\bar\iota})
  +
  (\vartheta_n)'_+(0)\,\nu_\iota(\Omega_\iota)
\]
is independent of the plan. Since \(\mu_n\ge0\),
\eqref{eq:overflow_order_clean} gives
\begin{equation}
\label{eq:theta_n_divergence_order}
  \D_{\vartheta_n}(P^\star\|\nu_{\bar\iota})
  \le
  \D_{\vartheta_n}(P\|\nu_{\bar\iota}),
  \qquad n\in\mathbb N.
\end{equation}

Because \(I_n\uparrow\mathbb R\), the Fenchel--Moreau theorem yields
\begin{equation}
\label{eq:theta_n_up_clean}
  \vartheta_n(u)\uparrow\vartheta(u),
  \qquad u\ge0.
\end{equation}

\paragraph{Passage to the limit.}
Set
\[
  \ell(u):=a_0u-\vartheta^*(a_0).
\]
Since \(a_0\in I_n\), one has
\[
  \ell(u)\le\vartheta_n(u)\le\vartheta(u),
  \qquad u\ge0,\quad n\in\mathbb N.
\]
Because \(r,r^\star\in L^1(\nu_{\bar\iota})\), the functions
\(\ell(r)\) and \(\ell(r^\star)\) are integrable. By
\eqref{eq:theta_n_up_clean}, monotone convergence applied after subtracting
\(\ell\) gives
\[
  \int_{\Omega_{\bar\iota}}
  \vartheta_n(r)\,d\nu_{\bar\iota}
  \uparrow
  \int_{\Omega_{\bar\iota}}
  \vartheta(r)\,d\nu_{\bar\iota}
\]
and
\[
  \int_{\Omega_{\bar\iota}}
  \vartheta_n(r^\star)\,d\nu_{\bar\iota}
  \uparrow
  \int_{\Omega_{\bar\iota}}
  \vartheta(r^\star)\,d\nu_{\bar\iota}.
\]

For each \(n\),
\[
  c_n
  =
  \sup\bigl(I_n\cap\dom\vartheta^*\bigr).
\]
Since \(I_n\uparrow\mathbb R\),
\[
  c_n
  \uparrow
  \sup\dom\vartheta^*
  =
  \vartheta'_\infty,
\]
where the last identity is the standard dual characterization of the
recession slope; see \cite{rockafellar_convex_1970}. It follows that
\[
  c_nP_\perp(\Omega_{\bar\iota})
  \uparrow
  \vartheta'_\infty P_\perp(\Omega_{\bar\iota})
\]
and
\[
  c_nP^\star_\perp(\Omega_{\bar\iota})
  \uparrow
  \vartheta'_\infty P^\star_\perp(\Omega_{\bar\iota}),
\]
with the usual convention that
\(+\infty\cdot m=+\infty\) for \(m>0\). Therefore
\[
  \D_{\vartheta_n}(P\|\nu_{\bar\iota})
  \uparrow
  \D_\vartheta(P\|\nu_{\bar\iota})
\]
and
\[
  \D_{\vartheta_n}(P^\star\|\nu_{\bar\iota})
  \uparrow
  \D_\vartheta(P^\star\|\nu_{\bar\iota}).
\]
Passing to the limit in
\eqref{eq:theta_n_divergence_order} gives
\[
  \D_\vartheta(P^\star\|\nu_{\bar\iota})
  \le
  \D_\vartheta(P\|\nu_{\bar\iota}).
\]

Since \(\pi\in\Pi_\iota^{\mathcal E}\) was arbitrary,
\(\pi^\star\) is optimal for the
\(\D_\vartheta\)-closest one-sided refinement problem.
\end{proof}

\subsection{Consequences of level-optimal maximin}
\label{sec:level_opt_maximin_consequences}

At this point the convex characterization is complete: finite-valued
minimization for a strictly convex integrand yields level-optimal maximin, and
level-optimal maximin is optimal for every convex objective. The first
consequence is the promised one-sided universal closestness statement, obtained
by combining Theorem~\ref{thm:level_opt_maximin_implies_theta_opt} with the
hockey-stick reduction from Section~\ref{sec:reduction_HS}. The same convex
viewpoint also yields a rigidity statement: for a strictly convex objective,
all minimizers of finite objective value have the same absolutely continuous
part, although their singular parts need not coincide.

\begin{corollary}[Universal closestness of level-optimal maximin refinements]
\label{cor:level_opt_maximin_universal_closest_one_sided}
Work under Assumption~\ref{assump:closed_polish_graph_nonempty_nbhd}.
Fix \(\iota\in\mathbb B\), and write \(\bar\iota:=1-\iota\).
Let \(\pi^\star\in\Pi_\iota^{\mathcal E}\) be level-optimal maximin from side \(\iota\) to side \(\bar\iota\).

Then \(\pi^\star\) is universally closest for the one-sided refinement problem: for every
\(\mathsf D\in\mathfrak D_{\mathrm{DPI}}\), the plan \(\pi^\star\) is a \(\mathsf D\)-closest
one-sided refinement in the sense of
Definition~\ref{defn:D_closest_one_sided_refinement_problem}.
\end{corollary}

\begin{proof}
Fix \(\mathsf D\in\mathfrak D_{\mathrm{DPI}}\). To prove that \(\pi^\star\) is a
\(\mathsf D\)-closest one-sided refinement, it is enough, by
Proposition~\ref{prop:HS_implies_DPI}, to verify the corresponding universal statement for
hockey-stick divergences.

For \(\gamma>0\), let \(\vartheta_\gamma(t):=(t-\gamma)_+\). Then \(\vartheta_\gamma\) is proper,
convex, and lower semicontinuous on \([0,\infty)\), and
\[
  \HS_\gamma(P\|Q)=\D_{\vartheta_\gamma}(P\|Q)
\]
by Definition~\ref{defn:hockey_stick_divergence}. Hence
Theorem~\ref{thm:level_opt_maximin_implies_theta_opt}, applied with
\(\vartheta=\vartheta_\gamma\), yields
\[
  \HS_\gamma\bigl(\nu_{\bar\iota}^{(\pi^\star)}\,\|\,\nu_{\bar\iota}\bigr)
  =
  \inf_{\pi\in\Pi_\iota^{\mathcal E}}
  \HS_\gamma\bigl(\nu_{\bar\iota}^{(\pi)}\,\|\,\nu_{\bar\iota}\bigr)
  \qquad \forall\,\gamma>0.
\]

Now apply Proposition~\ref{prop:HS_implies_DPI} with
\[
  \mathcal A:=\bigl\{\nu_{\bar\iota}^{(\pi)}:\pi\in\Pi_\iota^{\mathcal E}\bigr\},
  \qquad
  \mathcal B:=\{\nu_{\bar\iota}\}.
\]
Every measure in \(\mathcal A\) has total mass \(\nu_\iota(\Omega_\iota)\), and every measure in
\(\mathcal B\) has total mass \(\nu_{\bar\iota}(\Omega_{\bar\iota})\). Therefore
\(\nu_{\bar\iota}^{(\pi^\star)}\) is universally closest with respect to
\(\mathfrak D_{\mathrm{DPI}}\) among opposite marginals of feasible one-sided refinements.
Equivalently, for every \(\mathsf D\in\mathfrak D_{\mathrm{DPI}}\), the plan \(\pi^\star\) is a
\(\mathsf D\)-closest one-sided refinement in the sense of
Definition~\ref{defn:D_closest_one_sided_refinement_problem}.
\end{proof}

\begin{lemma}[Uniqueness of the \(Q\)-absolutely continuous part among finite-valued strictly convex minimizers]
\label{lem:strict_convex_unique_ac_part}
Let \((\Omega,\Sigma)\) be a measurable space, let \(\mathcal A\subseteq \mathcal M_+(\Omega)\) be convex, let
\(Q\in\mathcal M_+(\Omega)\), and let
\(\vartheta:[0,\infty)\to\mathbb R\cup\{+\infty\}\) be proper, convex, lower semicontinuous, and strictly convex on
\(\dom\vartheta\).

Suppose \(P_0,P_1\in\mathcal A\) satisfy
\[
  \D_\vartheta(P_0\|Q)=\D_\vartheta(P_1\|Q)
  =
  \inf_{P\in\mathcal A}\D_\vartheta(P\|Q)
  <\infty.
\]
Write their Lebesgue decompositions with respect to \(Q\) as
\(P_j=r_j\,Q+P_j^\perp\), with \(P_j^\perp\perp Q\), \(j\in\{0,1\}\).
Then \(r_0=r_1\) \(Q\)-a.e. Equivalently, \(P_0^\ac=P_1^\ac\).

If, in addition, \(\vartheta'_\infty=+\infty\), then \(P_0^\perp=P_1^\perp=0\).
In general, no uniqueness of the \(Q\)-singular part is asserted.
\end{lemma}

\begin{proof}
Set
\[
  m:=\inf_{P\in\mathcal A}\D_\vartheta(P\|Q)
  =
  \D_\vartheta(P_0\|Q)
  =
  \D_\vartheta(P_1\|Q)
  <\infty.
\]
Since \(\vartheta\) is proper, convex, and lower semicontinuous on \([0,\infty)\), it admits an affine lower
bound \(au+b\). As \(Q\) and \(P_j\) are finite, this shows that \(\int \vartheta(r_j)\,dQ\) is well defined in
\((-\infty,\infty]\). The finiteness of \(\D_\vartheta(P_j\|Q)\) therefore implies
\(\vartheta(r_j)<\infty\) for \(Q\)-a.e.\ point, hence \(r_j\in\dom\vartheta\) \(Q\)-a.e., \(j\in\{0,1\}\).

\paragraph{Averaging the minimizers.}
Set \(\bar P:=\tfrac12(P_0+P_1)\). Since \(\mathcal A\) is convex, \(\bar P\in\mathcal A\), so
\[
  \D_\vartheta(\bar P\|Q)\ge m.
\]
Moreover,
\[
  \bar P=\bar r\,Q+\bar P^\perp,
  \qquad
  \bar r:=\frac{r_0+r_1}{2},
  \qquad
  \bar P^\perp:=\frac{P_0^\perp+P_1^\perp}{2},
\]
and \(\bar P^\perp\perp Q\). By convexity of \(\vartheta\),
\[
  \vartheta(\bar r)\le \frac{\vartheta(r_0)+\vartheta(r_1)}{2}
  \qquad Q\text{-a.e.}
\]
Hence
\[
  \D_\vartheta(\bar P\|Q)
  =
  \int \vartheta(\bar r)\,dQ+\vartheta'_\infty\,\bar P^\perp(\Omega)
  \le
  \frac12\D_\vartheta(P_0\|Q)+\frac12\D_\vartheta(P_1\|Q)
  =
  m.
\]
Therefore \(\D_\vartheta(\bar P\|Q)=m\). Since
\[
  \vartheta'_\infty\,\bar P^\perp(\Omega)
  =
  \frac12\,\vartheta'_\infty P_0^\perp(\Omega)
  +
  \frac12\,\vartheta'_\infty P_1^\perp(\Omega),
\]
the only source of inequality in the preceding display is the convexity inequality for the density part. It follows that
\[
  \int
  \left[
    \frac{\vartheta(r_0)+\vartheta(r_1)}{2}
    -
    \vartheta\!\left(\frac{r_0+r_1}{2}\right)
  \right] dQ
  =0.
\]

\paragraph{Equality in the convexity inequality.}
Define
\[
  g:=
  \frac{\vartheta(r_0)+\vartheta(r_1)}{2}
  -
  \vartheta\!\left(\frac{r_0+r_1}{2}\right).
\]
Then \(g\ge 0\) \(Q\)-a.e., and the previous paragraph gives \(\int g\,dQ=0\). Hence \(g=0\) \(Q\)-a.e.

Since \(r_0,r_1\in\dom\vartheta\) \(Q\)-a.e.\ and \(\vartheta\) is strictly convex on \(\dom\vartheta\),
\[
  g(\omega)>0
  \qquad\text{whenever } r_0(\omega)\neq r_1(\omega).
\]
Thus \(Q(\{r_0\neq r_1\})=0\), so \(r_0=r_1\) \(Q\)-a.e. Equivalently,
\[
  P_0^\ac=P_1^\ac.
\]

\paragraph{The case \(\vartheta'_\infty=+\infty\).}
If \(\vartheta'_\infty=+\infty\), then the finiteness of \(\D_\vartheta(P_j\|Q)\) forces
\(P_j^\perp(\Omega)=0\), hence \(P_j^\perp=0\), for \(j\in\{0,1\}\).

No uniqueness of the \(Q\)-singular part is asserted when \(\vartheta'_\infty<\infty\), and none follows from
the argument above, since the singular contribution enters \(\D_\vartheta(\,\cdot\,\|Q)\) only through the total
mass \(P^\perp(\Omega)\).
\end{proof}

\begin{corollary}[Uniqueness of the absolutely continuous opposite marginal]
\label{cor:unique_ac_opposite_marginal_level_opt_maximin}
Work under Assumption~\ref{assump:closed_polish_graph_nonempty_nbhd}.
Fix \(\iota\in\mathbb B\), and write \(\bar\iota:=1-\iota\).

If \(\pi_0,\pi_1\in\Pi_\iota^{\mathcal E}\) are level-optimal maximin from side \(\iota\) to side \(\bar\iota\), then
\[
  \bigl(\nu_{\bar\iota}^{(\pi_0)}\bigr)^\ac
  =
  \bigl(\nu_{\bar\iota}^{(\pi_1)}\bigr)^\ac
\]
as measures on \(\Omega_{\bar\iota}\), where absolute continuity is taken with respect to \(\nu_{\bar\iota}\).

Equivalently, the \(\nu_{\bar\iota}\)-absolutely continuous part of \(\nu_{\bar\iota}^{(\pi)}\) is the same for every
level-optimal maximin refinement \(\pi\in\Pi_\iota^{\mathcal E}\).
\end{corollary}

\begin{proof}
Set \(\vartheta(t):=e^{-t}\), \(t\ge 0\), and
\[
  P_0:=\nu_{\bar\iota}^{(\pi_0)},
  \qquad
  P_1:=\nu_{\bar\iota}^{(\pi_1)},
\]
and define
\[
  \mathcal A:=\bigl\{\nu_{\bar\iota}^{(\pi)}:\pi\in\Pi_\iota^{\mathcal E}\bigr\}
  \subseteq \mathcal M_+(\Omega_{\bar\iota}).
\]
The set \(\mathcal A\) is convex: if \(P_k=\nu_{\bar\iota}^{(\pi_k)}\in\mathcal A\) for \(k=0,1\) and
\(\lambda\in[0,1]\), then \(\lambda\pi_0+(1-\lambda)\pi_1\in\Pi_\iota^{\mathcal E}\), and its
\(\bar\iota\)-marginal is \(\lambda P_0+(1-\lambda)P_1\).

Since \(\pi_0\) and \(\pi_1\) are level-optimal maximin, Theorem~\ref{thm:level_opt_maximin_implies_theta_opt},
applied with \(\vartheta(t)=e^{-t}\), shows that \(\pi_0\) and \(\pi_1\) are optimal for the
\(\D_\vartheta\)-closest one-sided refinement problem. Equivalently, \(P_0\) and \(P_1\) minimize
\(P\mapsto \D_\vartheta(P\|\nu_{\bar\iota})\) over \(\mathcal A\), since every \(P\in\mathcal A\) is of the
form \(P=\nu_{\bar\iota}^{(\pi)}\) for some \(\pi\in\Pi_\iota^{\mathcal E}\).

Moreover, \(\vartheta\) is proper, convex, lower semicontinuous, and strictly convex on \([0,\infty)\), and
\(\D_\vartheta(P_j\|\nu_{\bar\iota})<\infty\) for \(j=0,1\), because \(\vartheta\) is bounded and
\(\vartheta'_\infty=0\).

Applying Lemma~\ref{lem:strict_convex_unique_ac_part} on \(\Omega_{\bar\iota}\) with
\(Q=\nu_{\bar\iota}\), we obtain \(P_0^\ac=P_1^\ac\). That is,
\[
  \bigl(\nu_{\bar\iota}^{(\pi_0)}\bigr)^\ac
  =
  \bigl(\nu_{\bar\iota}^{(\pi_1)}\bigr)^\ac.
\]

The final sentence is the reformulation of this pairwise uniqueness statement.
\end{proof}

\section{Universal Closest Refinement Pairs}
\label{sec:universal_closest_pairs}

This section completes the passage from one-sided level-optimal maximin
refinements to paired refinements. We first show that proportional response
upgrades a one-sided level-optimal maximin refinement to a universally closest
pair. We then prove the converse for every closest pair of finite value under a
strictly convex \(\vartheta\)-divergence. These two results yield an exact
characterization of universal closest refinement pairs, as well as their
existence under the standing assumptions. The final subsection no longer
assumes proportional response: it develops the symmetric density decomposition
and the absolutely continuous/singular transport structure associated with an
arbitrary pair that is level-optimal maximin on both sides. This structure will
be used in the equilibrium characterization of Section~\ref{sec:continuum_agents_commodities}.

\subsection{Universal closestness via proportional response}
\label{sec:paired_universal_closest}

After the one-sided theory of Section~\ref{sec:one_side}, the universal paired theorem follows by
combining one-sided universal closestness with the paired-to-one-sided reduction from
Section~\ref{sec:paired_reduction}.  The point is that proportional response reconstructs the
paired object from the one-sided refinement without any further optimization.

\begin{theorem}[Universal closestness via proportional response]
\label{thm:universal_closest_refinement_pair}
Work under Assumption~\ref{assump:closed_polish_graph_nonempty_nbhd}.
Fix \(\iota\in\mathbb B\), and write \(\bar\iota:=1-\iota\).

Let \(\pi_\iota\in\Pi_\iota^{\mathcal E}\) be level-optimal maximin from side \(\iota\) to side \(\bar\iota\), and let
\(\pi_{\bar\iota}\in\Pi_{\bar\iota}^{\mathcal E}\) be a proportional response to \(\pi_\iota\), in the sense of Definition~\ref{defn:proportional_response_plan}.
Then \((\pi_\iota,\pi_{\bar\iota})\in\Pairi\) is a universal closest refinement pair, in the sense of Definition~\ref{defn:universal_closest_paired_refinement_problem}.
\end{theorem}

\begin{proof}
By Definition~\ref{defn:proportional_response_plan}, \(\pi_{\bar\iota}\in\Pi_{\bar\iota}^{\mathcal E}\), so
\((\pi_\iota,\pi_{\bar\iota})\in\Pairi\).

Let \(\mathsf D\in\mathfrak D_{\mathrm{DPI}}\). Since \(\pi_\iota\) is level-optimal maximin,
Corollary~\ref{cor:level_opt_maximin_universal_closest_one_sided} implies that \(\pi_\iota\) is a
\(\mathsf D\)-closest one-sided refinement on side \(\iota\). Lemma~\ref{lem:paired_to_one_sided_reduction},
part~(1), therefore yields that \((\pi_\iota,\pi_{\bar\iota})\) is a \(\mathsf D\)-closest refinement pair.
As this holds for every \(\mathsf D\in\mathfrak D_{\mathrm{DPI}}\), the pair is universally closest.
\end{proof}

\subsection{Structure of finite-valued strictly convex closest pairs}
\label{sec:strict_convex_closest_pairs_canonical}

We now prove the converse statement at the paired level.  The preceding theorem showed that a level-optimal maximin one-sided refinement, together with proportional response, produces a universally closest pair.  
The next result asks for the converse implication: if a pair attains a finite minimum of a single strictly convex paired objective, must it already have this structure? The answer is yes.
The one-sided component is identified through the paired-to-one-sided reduction, and the opposite-side component is then recovered from the uniqueness of minimizers for the adjoint strictly convex divergence.

The following fact is standard for \(\vartheta\)-divergences; see, for example,
Liese--Vajda~\cite[Section~II]{LieseVajda2006} and, for the classical probability-measure formulation,
Csisz\'ar~\cite{Csiszar1967}.

\begin{fact}[Adjoint integrand and divergence symmetry]
\label{fact:adjoint_integrand_divergence_symmetry}
Let \((\Omega,\Sigma)\) be a measurable space, and let
\(\vartheta:[0,\infty)\to\mathbb R\cup\{+\infty\}\) be proper, convex, and lower semicontinuous,
with
$
  \dom\vartheta\cap(0,\infty)\neq\emptyset.
$

Define its adjoint integrand \(\widehat\vartheta:[0,\infty)\to\mathbb R\cup\{+\infty\}\) by
\[
  \widehat\vartheta(t):=
  \begin{cases}
    t\,\vartheta(t^{-1}), & t>0,\\[0.3em]
    \vartheta'_\infty, & t=0,
  \end{cases}
\]
where \(\vartheta'_\infty:=\lim_{u\to\infty}\vartheta(u)/u\in\mathbb R\cup\{+\infty\}\).

Then \(\widehat\vartheta\) is proper, convex, and lower semicontinuous. Moreover, for every
\(P,Q\in\Mplus(\Omega)\),
\[
  \D_\vartheta(P\|Q)=\D_{\widehat\vartheta}(Q\|P).
\]

If, in addition, \(\vartheta\) is strictly convex on \(\dom\vartheta\), then
\(\widehat\vartheta\) is strictly convex on \(\dom\widehat\vartheta\).
\end{fact}

\begin{theorem}[Structure of finite-valued strictly convex closest pairs]
\label{thm:strict_convex_closest_pair_induces_canonical_pair}
Work under Assumption~\ref{assump:closed_polish_graph_nonempty_nbhd}.
Fix \(\iota\in\mathbb B\), write \(\bar\iota:=1-\iota\), and let
\(\vartheta:[0,\infty)\to\mathbb R\cup\{+\infty\}\) be proper, lower semicontinuous, and
strictly convex on \(\dom\vartheta\).

Let \((\pi_\iota,\pi_{\bar\iota})\in\Pairi\) satisfy
\[
  \D_\vartheta(\pi_\iota\|\pi_{\bar\iota})
  =
  \inf_{(\rho_\iota,\rho_{\bar\iota})\in\Pairi}
  \D_\vartheta(\rho_\iota\|\rho_{\bar\iota})
  <+\infty.
\]
Then:
\begin{enumerate}[(i)]
\item \(\pi_\iota\) is level-optimal maximin from side \(\iota\) to side \(\bar\iota\);
\item \(\pi_{\bar\iota}\) is a proportional response to \(\pi_\iota\), in the sense of
Definition~\ref{defn:proportional_response_plan}.
\end{enumerate}
\end{theorem}

%\noindent \textbf{Remark.}
%The finite-value hypothesis in
%Theorem~\ref{thm:strict_convex_closest_pair_induces_canonical_pair} ensures that
%\(\dom\vartheta\cap(0,\infty)\neq\varnothing\), so
%Fact~\ref{fact:adjoint_integrand_divergence_symmetry} applies. Moreover,
%\[
  %\D_{\widehat\vartheta}(\pi_{\bar\iota}\|\pi_\iota)
  %=
  %\D_\vartheta(\pi_\iota\|\pi_{\bar\iota})
  %<+\infty.
%\]
%Thus the theorem may be applied after replacing \(\vartheta\) by
%\(\widehat\vartheta\) and swapping the two sides. Consequently,
%\(\pi_{\bar\iota}\) is level-optimal maximin from side \(\bar\iota\) to side
%\(\iota\), and \(\pi_\iota\) is a proportional response to
%\(\pi_{\bar\iota}\).

\begin{proof}
The first assertion identifies the \(\iota\)-component. The second then reconstructs the
\(\bar\iota\)-component from the uniqueness of the \(\pi_\iota\)-absolutely continuous part.

Write the Lebesgue decomposition of \(\pi_\iota\) with respect to
\(\pi_{\bar\iota}\) as
\[
  \pi_\iota=r\,\pi_{\bar\iota}+\pi_\iota^\perp.
\]
If \(\pi_\iota^\perp\neq0\), the finiteness of
\(\D_\vartheta(\pi_\iota\|\pi_{\bar\iota})\) forces
\(\vartheta'_\infty<+\infty\), and therefore
\(\dom\vartheta\cap(0,\infty)\neq\varnothing\).

If \(\pi_\iota^\perp=0\), then
\[
  \int_{\Omega_0\times\Omega_1}r\,d\pi_{\bar\iota}
  =
  \pi_\iota(\Omega_0\times\Omega_1)
  =
  \nu_\iota(\Omega_\iota)
  >0.
\]
Thus \(r>0\) on a set of positive \(\pi_{\bar\iota}\)-measure.
The finiteness of the divergence implies that
\(\vartheta(r)<+\infty\) almost everywhere on this set, and again
\(\dom\vartheta\cap(0,\infty)\neq\varnothing\).

Hence Fact~\ref{fact:adjoint_integrand_divergence_symmetry} applies to
\(\vartheta\).

Because \(\D_\vartheta\) satisfies channel-DPI, Lemma~\ref{lem:paired_to_one_sided_reduction}
applies with \(\mathsf D=\D_\vartheta\). Since \((\pi_\iota,\pi_{\bar\iota})\in\Pairi\) attains the paired
\(\D_\vartheta\)-infimum, part~\emph{(2)} of that lemma shows that \(\pi_\iota\) is a
\(\D_\vartheta\)-closest one-sided refinement on side \(\iota\). The optimal value is finite by assumption,
so Theorem~\ref{thm:strict_convex_implies_level_opt_maximin} yields that \(\pi_\iota\) is
level-optimal maximin from side \(\iota\) to side \(\bar\iota\). This proves \emph{(i)}.

For \emph{(ii)}, choose a proportional response \(\eta_{\bar\iota}\) to \(\pi_\iota\), in the sense of
Definition~\ref{defn:proportional_response_plan}. By Lemma~\ref{lem:paired_to_one_sided_reduction},
part~\emph{(2)}, the paired optimality of \((\pi_\iota,\pi_{\bar\iota})\) implies that \(\pi_\iota\) is a
\(\D_\vartheta\)-closest one-sided refinement on side \(\iota\). Because \(\eta_{\bar\iota}\) is a
proportional response to \(\pi_\iota\), part~\emph{(1)} of the same lemma shows that
\((\pi_\iota,\eta_{\bar\iota})\in\Pairi\) is likewise a \(\D_\vartheta\)-closest refinement pair.

It follows that both \(\pi_{\bar\iota}\) and \(\eta_{\bar\iota}\) minimize the map
\[
  R\longmapsto \D_\vartheta(\pi_\iota\|R)
\]
over the convex class \(\Pi_{\bar\iota}^{\mathcal E}\). Let \(\widehat\vartheta\) be the adjoint
integrand associated with \(\vartheta\). By Fact~\ref{fact:adjoint_integrand_divergence_symmetry},
\(\widehat\vartheta\) is proper, lower semicontinuous, and strictly convex on
\(\dom\widehat\vartheta\), and
\[
  \D_\vartheta(\pi_\iota\|R)=\D_{\widehat\vartheta}(R\|\pi_\iota)
  \qquad\text{for every }R\in\Mplus(\Omega_0\times\Omega_1).
\]
Therefore \(\pi_{\bar\iota}\) and \(\eta_{\bar\iota}\) are both minimizers of
\(R\mapsto \D_{\widehat\vartheta}(R\|\pi_\iota)\) over \(\Pi_{\bar\iota}^{\mathcal E}\). Applying
Lemma~\ref{lem:strict_convex_unique_ac_part} with \(Q=\pi_\iota\), we conclude that the
\(\pi_\iota\)-absolutely continuous parts of \(\pi_{\bar\iota}\) and \(\eta_{\bar\iota}\) coincide.

We now identify this common absolutely continuous part. Set
\[
  P:=(\pr_{\bar\iota})_\#\pi_\iota=\nu^{(\pi_\iota)}_{\bar\iota}.
\]
Choose a disintegration
\[
  \pi_\iota=P\kprod \K
\]
as in Definition~\ref{defn:disintegration}, and write the Lebesgue decomposition of
\(\nu_{\bar\iota}\) with respect to \(P\) as
\[
  \nu_{\bar\iota}=(\nu_{\bar\iota})_{\mathrm{ac}}+(\nu_{\bar\iota})_{\perp}.
\]
Since \(\eta_{\bar\iota}\) is a proportional response to \(\pi_\iota\), there exists a measurable
probability kernel \(\K^{\mathrm{fb}}\), carried by \(N_\iota(y)\) for
\((\nu_{\bar\iota})_{\perp}\)-a.e.\ \(y\), such that
\[
  \eta_{\bar\iota}
  =
  (\nu_{\bar\iota})_{\mathrm{ac}}\kprod \K
  +
  (\nu_{\bar\iota})_{\perp}\kprod \K^{\mathrm{fb}}.
\]

Because \((\nu_{\bar\iota})_{\mathrm{ac}}\ll P\), one has
\[
  (\nu_{\bar\iota})_{\mathrm{ac}}\kprod \K \ll P\kprod \K=\pi_\iota.
\]
On the other hand, since \((\nu_{\bar\iota})_{\perp}\perp P\), choose \(Z\in\Sigma_{\bar\iota}\) such that
\(P(Z)=0\) and \((\nu_{\bar\iota})_{\perp}(Z^c)=0\). Then
\[
  \pi_\iota\bigl(\pr_{\bar\iota}^{-1}(Z)\bigr)=P(Z)=0,
  \qquad
  \bigl((\nu_{\bar\iota})_{\perp}\kprod \K^{\mathrm{fb}}\bigr)
  \bigl(\pr_{\bar\iota}^{-1}(Z^c)\bigr)
  =
  (\nu_{\bar\iota})_{\perp}(Z^c)=0.
\]
Hence \((\nu_{\bar\iota})_{\perp}\kprod \K^{\mathrm{fb}}\perp \pi_\iota\). Thus the
\(\pi_\iota\)-absolutely continuous part of \(\eta_{\bar\iota}\) is exactly
\[
  (\nu_{\bar\iota})_{\mathrm{ac}}\kprod \K.
\]
Therefore the \(\pi_\iota\)-absolutely continuous part of \(\pi_{\bar\iota}\) is also
\((\nu_{\bar\iota})_{\mathrm{ac}}\kprod \K\). We may therefore write
\[
  \pi_{\bar\iota}
  =
  (\nu_{\bar\iota})_{\mathrm{ac}}\kprod \K + S,
  \qquad
  S\perp \pi_\iota.
\]

Since \((\pr_{\bar\iota})_\#\pi_{\bar\iota}=\nu_{\bar\iota}\) and
\[
  (\pr_{\bar\iota})_\#\bigl((\nu_{\bar\iota})_{\mathrm{ac}}\kprod \K\bigr)
  =
  (\nu_{\bar\iota})_{\mathrm{ac}},
\]
it follows that
\[
  (\pr_{\bar\iota})_\#S=(\nu_{\bar\iota})_{\perp}.
\]

Next we check that \(S\) is concentrated on \(\mathcal E\). The measure \(\pi_{\bar\iota}\) is concentrated on
\(\mathcal E\) because \((\pi_\iota,\pi_{\bar\iota})\in\Pairi\). The kernel \(\K\) is carried by
\(N_\iota(y)\) for \(P\)-a.e.\ \(y\), and \((\nu_{\bar\iota})_{\mathrm{ac}}\ll P\), so
\((\nu_{\bar\iota})_{\mathrm{ac}}\kprod \K\) is also carried by \(\mathcal E\). Hence \(S(\mathcal E^c)=0\).

Finally, disintegrate \(S\) with respect to its \(\bar\iota\)-marginal \((\nu_{\bar\iota})_{\perp}\):
\[
  S=(\nu_{\bar\iota})_{\perp}\kprod \widetilde{\K}^{\mathrm{fb}}
\]
for some measurable probability kernel \(\widetilde{\K}^{\mathrm{fb}}\). Since \(S(\mathcal E^c)=0\),
this kernel is carried by \(N_\iota(y)\) for \((\nu_{\bar\iota})_{\perp}\)-a.e.\ \(y\). Consequently
\[
  \pi_{\bar\iota}
  =
  (\nu_{\bar\iota})_{\mathrm{ac}}\kprod \K
  +
  (\nu_{\bar\iota})_{\perp}\kprod \widetilde{\K}^{\mathrm{fb}},
\]
which is exactly the form required in Definition~\ref{defn:proportional_response_plan}. Thus
\(\pi_{\bar\iota}\) is a proportional response to \(\pi_\iota\).
\end{proof}

\begin{remark}[Symmetric conclusion]
\label{rem:strict_convex_closest_pair_symmetric_conclusion}
The conclusion of
Theorem~\ref{thm:strict_convex_closest_pair_induces_canonical_pair}
is symmetric. Indeed, the finite-value hypothesis implies
\(\dom\vartheta\cap(0,\infty)\neq\varnothing\), and
Fact~\ref{fact:adjoint_integrand_divergence_symmetry} gives
\[
  \D_{\widehat\vartheta}
  (\pi_{\bar\iota}\|\pi_\iota)
  =
  \D_\vartheta
  (\pi_\iota\|\pi_{\bar\iota})
  <+\infty.
\]
Applying the theorem to \(\widehat\vartheta\) after interchanging the
two sides shows that \(\pi_{\bar\iota}\) is level-optimal maximin from
side \(\bar\iota\) to side \(\iota\), and that \(\pi_\iota\) is a
proportional response to \(\pi_{\bar\iota}\).
\end{remark}

\begin{corollary}[Characterization of universal closest refinement pairs]
\label{cor:universal_closest_pair_characterization}
Work under
Assumption~\ref{assump:closed_polish_graph_nonempty_nbhd}.
Fix \(\iota\in\mathbb B\), and write \(\bar\iota:=1-\iota\).
For
\((\pi_\iota,\pi_{\bar\iota})\in\Pairi\),
the following are equivalent:
\begin{enumerate}[(i)]
\item
\((\pi_\iota,\pi_{\bar\iota})\) is a universal closest refinement
pair;

\item
\(\pi_\iota\) is level-optimal maximin from side \(\iota\) to side
\(\bar\iota\), and \(\pi_{\bar\iota}\) is a proportional response to
\(\pi_\iota\);

\item
\(\pi_{\bar\iota}\) is level-optimal maximin from side
\(\bar\iota\) to side \(\iota\), and \(\pi_\iota\) is a proportional
response to \(\pi_{\bar\iota}\).
\end{enumerate}

Consequently, every universal closest refinement pair is
level-optimal maximin on both sides and consists of mutual
proportional responses.
\end{corollary}

\begin{proof}
The implications \emph{(ii)}\(\Rightarrow\)\emph{(i)} and
\emph{(iii)}\(\Rightarrow\)\emph{(i)} follow from
Theorem~\ref{thm:universal_closest_refinement_pair}, applied with the
corresponding choice of side.

Assume \emph{(i)} and set
\[
  \vartheta(t):=e^{-t},
  \qquad t\ge0.
\]
This integrand is proper, convex, lower semicontinuous, bounded, and
strictly convex, with \(\vartheta'_\infty=0\). Hence
\(\D_\vartheta \in\mathfrak D_{\mathrm{DPI}}\) is finite on every pair of finite measures. 
Universal closestness therefore implies that
\((\pi_\iota,\pi_{\bar\iota})\) is a finite-valued
\(\D_\vartheta\)-closest pair. Theorem~
\ref{thm:strict_convex_closest_pair_induces_canonical_pair}
gives~\emph{(ii)}, while
Remark~\ref{rem:strict_convex_closest_pair_symmetric_conclusion}
gives~\emph{(iii)}.
\end{proof}

\begin{corollary}[Existence of universal closest refinement pairs]
\label{cor:universal_closest_pair_existence}
Under
Assumption~\ref{assump:closed_polish_graph_nonempty_nbhd},
there exists a universal closest refinement pair.
\end{corollary}

\begin{proof}
Fix \(\iota\in\mathbb B\) and set
\(\vartheta(t):=e^{-t}\). By
Proposition~\ref{prop:vartheta_one_sided_attainment}, the corresponding
one-sided refinement problem admits a minimizer
\(\pi_\iota\in\Pi_\iota^{\mathcal E}\). Since \(\vartheta\) is bounded,
the minimum is finite, and
Theorem~\ref{thm:strict_convex_implies_level_opt_maximin} shows that
\(\pi_\iota\) is level-optimal maximin.

Choose a proportional response
\(\pi_{\bar\iota}\in\Pi_{\bar\iota}^{\mathcal E}\) to \(\pi_\iota\).
Then
Theorem~\ref{thm:universal_closest_refinement_pair}
shows that
\((\pi_\iota,\pi_{\bar\iota})\) is universally closest.
\end{proof}

\subsection{Symmetric density decomposition and geometric structure}
\label{subsec:symmetric_density_decomposition_support_geometry}

We now leave the proportional-response setting. The input in this
subsection is only a pair
\((\pi_0,\pi_1)\) whose two components are level-optimal maximin;
neither universal closestness nor proportional response between the
two plans is assumed. The purpose is to isolate the symmetric density
decomposition and the geometric compatibility of its absolutely
continuous and singular layers.

This decomposition is the measure-theoretic counterpart of the
symmetric density decomposition in the finite theory
\cite{DBLP:conf/innovations/ChanX25}. In the finite setting, the
decomposition is obtained by repeatedly extracting densest subsets.
That recursive construction is unavailable in the present measurable
setting. Its structural role is instead played by the canonical
absolutely continuous payloads, the singular backward-transport
principle, reciprocal density layers, and the balanced transport
identity proved below.

\begin{definition}[Symmetric density decomposition]
\label{defn:sym_density_decomp}
Work under Assumption~\ref{assump:closed_polish_graph_nonempty_nbhd}. Let
\((\pi_0,\pi_1)\in \Pi^{\mathcal E}(\nu_0,\nu_1)\) be a refinement pair such that, for each
\(\iota\in\mathbb B\), the plan \(\pi_\iota\in\Pi_\iota^{\mathcal E}\) is level-optimal maximin in the
sense of Definition~\ref{defn:level_opt_maximin}.

For \(\iota\in\mathbb B\), define \(P_\iota:=(\pr_\iota)_\#\pi_{\bar\iota}\), the payload measure on side
\(\iota\) induced by the opposite-side refinement. Let
\(P_\iota=P_\iota^{\mathrm{ac}}+P_\iota^\perp\) be the Lebesgue decomposition of \(P_\iota\) with respect
to \(\nu_\iota\), and let
\(\nu_\iota=\nu_\iota^{\mathrm{ac}}+\nu_\iota^\perp\) be the Lebesgue decomposition of \(\nu_\iota\) with
respect to \(P_\iota^{\mathrm{ac}}\).

Finally, define
\[
  \rho_\iota^{\mathrm{ac}}
  :=
  \frac{dP_\iota^{\mathrm{ac}}}{d\nu_\iota}
  =
  \frac{dP_\iota^{\mathrm{ac}}}{d\nu_\iota^{\mathrm{ac}}}.
\]

The \emph{symmetric density decomposition} is the collection, for \(\iota\in\mathbb B\), of the measures
\(\nu_\iota^{\mathrm{ac}}, \nu_\iota^\perp\) and the density
\(\rho_\iota^{\mathrm{ac}}:\Omega_\iota\to[0,\infty)\).
\end{definition}

\begin{remark}[Canonicity]
\label{rem:sym_density_decomp_independent}
Although
Definition~\ref{defn:sym_density_decomp}
is formulated using a chosen pair, its absolutely continuous data
are canonical. Indeed,
Corollary~\ref{cor:unique_ac_opposite_marginal_level_opt_maximin}
shows that \(P_\iota^{\mathrm{ac}}\) depends only on the side
\(\iota\), and not on the particular level-optimal maximin
refinement. Consequently,
\(\nu_\iota^{\mathrm{ac}}\), \(\nu_\iota^\perp\), and
\(\rho_\iota^{\mathrm{ac}}\) are independent of the chosen
level-optimal maximin pair, up to the usual almost-everywhere choice
of density representative.
\end{remark}

The first structural fact is that transporting singular payload backward along an admissible reverse
kernel cannot create absolutely continuous mass on the opposite side.  This is the basic separation
principle for the singular component and will be used repeatedly below.

\begin{lemma}[Backward transport preserves singularity]
\label{lem:singular_payload_stays_singular_under_backward_kernel}
Work under Assumption~\ref{assump:closed_polish_graph_nonempty_nbhd}, and consider the symmetric density decomposition of Definition~\ref{defn:sym_density_decomp}. Fix \(\iota\in\mathbb B\), and write \(\bar\iota:=1-\iota\).

Let
\[
  K_\iota:\Omega_{\bar\iota}\times\Sigma_\iota\to[0,1]
\]
be a probability kernel such that
\[
  K_\iota\bigl(y,N_\iota(y)\bigr)=1
  \qquad\text{for }\nu_{\bar\iota}^\perp\text{-a.e.\ }y,
\]
where
\[
  N_\iota(y):=\{x\in\Omega_\iota:(x,y)\in\mathcal E\}.
\]
Then
\[
  (\pr_\iota)_\#\bigl(\nu_{\bar\iota}^\perp\kprod K_\iota\bigr)\perp \nu_\iota.
\]
\end{lemma}

\begin{proof}
Fix \(\iota\in\mathbb B\), write \(\bar\iota:=1-\iota\), and let \(\pi:=\pi_\iota\) be the
level-optimal maximin refinement from side \(\iota\) to side \(\bar\iota\) appearing in
Definition~\ref{defn:sym_density_decomp}. Set
\[
  P:=(\pr_{\bar\iota})_\#\pi.
\]
Write the Lebesgue decomposition of \(P\) with respect to \(\nu_{\bar\iota}\) as
\(P=r\,\nu_{\bar\iota}+P^\perp\).

Let \(K_\iota\) be as in the statement, and define
\[
  \tau:=\nu_{\bar\iota}^\perp\kprod K_\iota\in \mathcal M_+(\Omega_0\times\Omega_1).
\]
By the defining property of \(K_\iota\), the measure \(\tau\) is concentrated on \(\mathcal E\), and
\((\pr_{\bar\iota})_\#\tau=\nu_{\bar\iota}^\perp\). Set
\[
  \mu:=(\pr_\iota)_\#\tau\in\mathcal M_+(\Omega_\iota).
\]
We must show that \(\mu\perp \nu_\iota\).

Assume for contradiction that \(\mu\not\perp \nu_\iota\). Let
\(\mu=\mu^{\mathrm{ac}}+\mu^\perp\) be the Lebesgue decomposition with respect to \(\nu_\iota\), and
write \(f:=d\mu^{\mathrm{ac}}/d\nu_\iota\). Since \(\mu^{\mathrm{ac}}\neq 0\), the sets
\(A_n:=\{0<f\le n\}\) satisfy \(\mu^{\mathrm{ac}}(A_n)>0\) for some \(n\). Fix such an \(n\), and set
\[
  A:=A_n=\{0<f\le n\}.
\]

Let \(m_{\bar\iota}:=\nu_{\bar\iota}(\Omega_{\bar\iota})\), and choose
\begin{equation}\label{eq:t_choice_singular_backward_clean}
  t:=\frac{\mu^{\mathrm{ac}}(A)}{4n\,m_{\bar\iota}}>0.
\end{equation}

Since
\[
  \mu^{\mathrm{ac}}(A)
  \le
  \mu(\Omega_\iota)
  =
  \tau(\Omega_0\times\Omega_1)
  =
  \nu_{\bar\iota}^\perp(\Omega_{\bar\iota})
  \le
  m_{\bar\iota},
\]
one has
\begin{equation}\label{eq:t_choice_bounds_singular_backward_clean}
  tn=\frac{\mu^{\mathrm{ac}}(A)}{4m_{\bar\iota}}\le \frac14,
  \qquad
  2nt\,m_{\bar\iota}=\frac{\mu^{\mathrm{ac}}(A)}{2}.
\end{equation}

\paragraph{A \(P\)-null set carrying \(\nu_{\bar\iota}^\perp\).}
Since
\(\nu_{\bar\iota}^\perp\perp P^{\mathrm{ac}}\) and
\(\nu_{\bar\iota}^\perp\le\nu_{\bar\iota}\perp P^\perp\),
one has \(\nu_{\bar\iota}^\perp\perp P\).
Hence there exists \(S\in\Sigma_{\bar\iota}\) such that
\[
  P(S)=0,
  \qquad
  \nu_{\bar\iota}^\perp(S^\c)=0.
\]

Because
\(P(S)=\int_S r\,d\nu_{\bar\iota}+P^\perp(S)\), nonnegativity gives
\begin{equation}\label{eq:r_zero_on_support_of_nu_perp_clean}
  r=0 \ \ \nu_{\bar\iota}\text{-a.e.\ on }S,
  \qquad
  P^\perp(S)=0.
\end{equation}

\paragraph{The under-level part.}
Apply Lemma~\ref{lem:overflow_decomposition} to \(\pi\) at level \(t\). Thus
\[
  \pi=\sigma_t+\pi_t^{\mathrm{ov}},
  \qquad
  \sigma_t\in\Feas_{\bar\iota}(t),
  \qquad
  (\pr_{\bar\iota})_\#\sigma_t=(r\wedge t)\,\nu_{\bar\iota}.
\]
Let \(\mu_t:=(\pr_\iota)_\#\sigma_t\). Since \(\mu_t\le \nu_\iota\), we may write
\(\mu_t=g\,\nu_\iota\) with \(0\le g\le 1\) \(\nu_\iota\)-a.e.

Set
\[
  B:=\{x\in A:g(x)\le 1/2\}.
\]
On \(B\),
\begin{equation}\label{eq:slack_on_B_singular_backward_clean}
  (\nu_\iota-\mu_t)\!\restriction B\ge \frac12\,\nu_\iota\!\restriction B.
\end{equation}
Moreover,
\[
  \nu_\iota(A\setminus B)
  =\nu_\iota(\{g>1/2\}\cap A)
  \le 2\int_A g\,d\nu_\iota
  =2\mu_t(A)
  \le 2\mu_t(\Omega_\iota).
\]
Since \(\mu_t(\Omega_\iota)=\sigma_t(\Omega_0\times\Omega_1)\le t\,m_{\bar\iota}\), we obtain
\[
  \nu_\iota(A\setminus B)\le 2t\,m_{\bar\iota}.
\]
Because \(f\le n\) on \(A\),
\[
  \mu^{\mathrm{ac}}(A\setminus B)
  =\int_{A\setminus B} f\,d\nu_\iota
  \le n\,\nu_\iota(A\setminus B)
  \le 2nt\,m_{\bar\iota}
  =\frac{\mu^{\mathrm{ac}}(A)}{2}
\]
by \eqref{eq:t_choice_bounds_singular_backward_clean}. Hence
\begin{equation}\label{eq:muac_B_positive_singular_backward_clean}
  \mu^{\mathrm{ac}}(B)\ge \frac{\mu^{\mathrm{ac}}(A)}{2}>0.
\end{equation}

\paragraph{A positive perturbation carried by \(S\).}
Since \(\mu^{\mathrm{ac}}\ll \mu\), let \(w:=d\mu^{\mathrm{ac}}/d\mu\), so \(0\le w\le 1\)
\(\mu\)-a.e.
Define the submeasure \(\tau^{\mathrm{ac}}\le \tau\) by
\[
  \tau^{\mathrm{ac}}(E):=\int_E w(x)\,\tau(d(x,y)),
  \qquad
  E\in \Sigma_\iota\otimes\Sigma_{\bar\iota}.
\]

Then
\[
  (\pr_\iota)_\#\tau^{\mathrm{ac}}=\mu^{\mathrm{ac}},
  \qquad
  (\pr_{\bar\iota})_\#\tau^{\mathrm{ac}}\le (\pr_{\bar\iota})_\#\tau=\nu_{\bar\iota}^\perp.
\]
In particular, \((\pr_{\bar\iota})_\#\tau^{\mathrm{ac}}\) is carried by \(S\).

Now set
\[
  \eta:=t\,\tau^{\mathrm{ac}}\!\restriction \pr_\iota^{-1}(B).
\]
Then \(\eta(\mathcal E^\c)=0\) and
\[
  (\pr_\iota)_\#\eta=t\,\mu^{\mathrm{ac}}\!\restriction B.
\]
Since \(\mu^{\mathrm{ac}}\!\restriction B\le n\,\nu_\iota\!\restriction B\) and \(tn\le 1/4\),
\[
  (\pr_\iota)_\#\eta
  \le tn\,\nu_\iota\!\restriction B
  \le \frac14\,\nu_\iota\!\restriction B
  \le \nu_\iota\!\restriction B-\mu_t\!\restriction B
\]
by \eqref{eq:slack_on_B_singular_backward_clean}. Therefore
\begin{equation}\label{eq:source_feasibility_sigma_plus_eta_clean}
  (\pr_\iota)_\#(\sigma_t+\eta)\le \nu_\iota.
\end{equation}

For the opposite marginal, we have
\[
  (\pr_{\bar\iota})_\#\eta\le t\,(\pr_{\bar\iota})_\#\tau^{\mathrm{ac}}\le t\,\nu_{\bar\iota}^\perp,
\]
so \((\pr_{\bar\iota})_\#\eta\) is carried by \(S\). On the other hand,
\((\pr_{\bar\iota})_\#\sigma_t=(r\wedge t)\nu_{\bar\iota}\), and
\eqref{eq:r_zero_on_support_of_nu_perp_clean} gives \((r\wedge t)=0\) \(\nu_{\bar\iota}\)-a.e.\ on \(S\).
Hence
\[
  (\pr_{\bar\iota})_\#\sigma_t\!\restriction S=0,
  \qquad
  (\pr_{\bar\iota})_\#\eta\le t\,\nu_{\bar\iota}\!\restriction S.
\]
Outside \(S\), \((\pr_{\bar\iota})_\#\eta=0\), while
\((\pr_{\bar\iota})_\#\sigma_t\le t\,\nu_{\bar\iota}\). Therefore
\begin{equation}\label{eq:target_feasibility_sigma_plus_eta_clean}
  (\pr_{\bar\iota})_\#(\sigma_t+\eta)\le t\,\nu_{\bar\iota}.
\end{equation}
Combining \eqref{eq:source_feasibility_sigma_plus_eta_clean} and
\eqref{eq:target_feasibility_sigma_plus_eta_clean}, we obtain
\[
  \sigma_t+\eta\in \Feas_{\bar\iota}(t).
\]

Finally,
\[
  \eta(\Omega_0\times\Omega_1)
  =(\pr_\iota)_\#\eta(\Omega_\iota)
  =t\,\mu^{\mathrm{ac}}(B)
  >0
\]
by \eqref{eq:muac_B_positive_singular_backward_clean}. Hence
\[
  \Fit_{\bar\iota}(t)\ge (\sigma_t+\eta)(\Omega_0\times\Omega_1)>\sigma_t(\Omega_0\times\Omega_1).
\]
But
\[
  \sigma_t(\Omega_0\times\Omega_1)
  =(\pr_{\bar\iota})_\#\sigma_t(\Omega_{\bar\iota})
  =\int_{\Omega_{\bar\iota}}(r\wedge t)\,d\nu_{\bar\iota}
  =P_t^{(\pi)}(\Omega_{\bar\iota}),
\]
so this contradicts the level-optimal maximin identity
\(\Fit_{\bar\iota}(t)=P_t^{(\pi)}(\Omega_{\bar\iota})\).

Therefore \(\mu\perp \nu_\iota\), that is,
\[
  (\pr_\iota)_\#(\nu_{\bar\iota}^\perp\kprod K_\iota)\perp \nu_\iota.
\]
\end{proof}

The next lemma records the reciprocal compatibility of the absolutely continuous density
layers.  It shows that no positive subflow can travel from a region where
\(\rho_\iota^{\mathrm{ac}}\le t\) into a region where \(\rho_{\bar\iota}^{\mathrm{ac}}<1/t\).  This
is the key geometric obstruction behind the null-neighborhood statements proved afterward.

\begin{lemma}[No subflow across incompatible layers]
\label{lem:reciprocal_densities_layer_compatibility}
Work under Assumption~\ref{assump:closed_polish_graph_nonempty_nbhd}, and consider the symmetric density decomposition of Definition~\ref{defn:sym_density_decomp}. Fix \(\iota\in\mathbb B\), and write \(\bar\iota:=1-\iota\).

For \(t>0\), define
\[
  A_t:=\{x\in\Omega_\iota:\rho_\iota^{\mathrm{ac}}(x)\le t\},
  \qquad
  B_t:=\Bigl\{y\in\Omega_{\bar\iota}:\rho_{\bar\iota}^{\mathrm{ac}}(y)\ge \frac1t\Bigr\}.
\]

If \(\sigma\in\mathcal M_+(\Omega_\iota\times\Omega_{\bar\iota})\) satisfies
\[
  \sigma(\mathcal E_\iota^\c)=0,
  \qquad
  (\pr_\iota)_\#\sigma\le \nu_\iota\!\restriction A_t,
  \qquad
  (\pr_{\bar\iota})_\#\sigma\le \nu_{\bar\iota}^{\mathrm{ac}}\!\restriction B_t^\c,
\]
then \(\sigma=0\).
\end{lemma}

\begin{proof}
Fix \(\iota\in\mathbb B\), and write \(\bar\iota:=1-\iota\). Let
\(\pi_\iota\in\Pi_\iota^{\mathcal E}\) be a level-optimal maximin refinement from side \(\iota\) to
side \(\bar\iota\). Set
\[
  P_{\bar\iota}:=(\pr_{\bar\iota})_\#\pi_\iota,
  \qquad
  P_{\bar\iota}=P_{\bar\iota}^{\mathrm{ac}}+P_{\bar\iota}^\perp
  \quad\text{with respect to }\nu_{\bar\iota},
\]
so that \(P_{\bar\iota}^{\mathrm{ac}}=\rho_{\bar\iota}^{\mathrm{ac}}\,\nu_{\bar\iota}\).

Choose a reverse disintegration \(\pi_\iota=P_{\bar\iota}\kprod K\), and let
\(\nu_{\bar\iota}=\nu_{\bar\iota}^{\mathrm{ac}}+\nu_{\bar\iota}^{\perp}\) be the Lebesgue decomposition
with respect to \(P_{\bar\iota}^{\mathrm{ac}}\). Choose a proportional response \(\pi_{\bar\iota}\) to
\(\pi_\iota\) built from this same kernel \(K\), namely
\[
  \pi_{\bar\iota}
  =
  \nu_{\bar\iota}^{\mathrm{ac}}\kprod K
  +
  \nu_{\bar\iota}^{\perp}\kprod K^{\mathrm{fb}}
\]
for some admissible fallback kernel \(K^{\mathrm{fb}}\). By
Theorem~\ref{thm:universal_closest_refinement_pair}, the pair
\((\pi_\iota,\pi_{\bar\iota})\) is a universal closest refinement pair. In particular, it is
\(\D_\vartheta\)-closest for the bounded strictly convex integrand \(\vartheta(u)=e^{-u}\). Applying
Theorem~\ref{thm:strict_convex_closest_pair_induces_canonical_pair}, together with the symmetry
remark following that theorem, we may realize the symmetric density decomposition by the pair
\((\pi_\iota,\pi_{\bar\iota})\).

Define \(P_\iota:=(\pr_\iota)_\#\pi_{\bar\iota}\), let
\(P_\iota=P_\iota^{\mathrm{ac}}+P_\iota^\perp\) be the Lebesgue decomposition with respect to
\(\nu_\iota\), and write \(P_\iota^{\mathrm{ac}}=\rho_\iota^{\mathrm{ac}}\,\nu_\iota\). The measure
\(\nu_{\bar\iota}^{\mathrm{ac}}\kprod K\) is absolutely continuous with respect to
\(P_{\bar\iota}\kprod K=\pi_\iota\), hence its \(\iota\)-marginal is absolutely continuous with
respect to \((\pr_\iota)_\#\pi_\iota=\nu_\iota\). On the other hand,
Lemma~\ref{lem:singular_payload_stays_singular_under_backward_kernel} shows that the
\(\iota\)-marginal of \(\nu_{\bar\iota}^{\perp}\kprod K^{\mathrm{fb}}\) is singular with respect to
\(\nu_\iota\). Therefore
\begin{equation}
\label{eq:Piac_from_proportional_response}
  P_\iota^{\mathrm{ac}}
  =
  (\pr_\iota)_\#(\nu_{\bar\iota}^{\mathrm{ac}}\kprod K).
\end{equation}

Now let \(\sigma\in\mathcal M_+(\Omega_\iota\times\Omega_{\bar\iota})\) satisfy
\(\sigma(\mathcal E_\iota^\c)=0\),
\((\pr_\iota)_\#\sigma\le \nu_\iota\!\restriction A_t\), and
\((\pr_{\bar\iota})_\#\sigma\le \nu_{\bar\iota}^{\mathrm{ac}}\!\restriction B_t^\c\).
Assume for contradiction that \(\sigma\neq 0\).

\paragraph{Localization to a bounded incompatible layer.}
Write \(\alpha:=(\pr_\iota)_\#\sigma\) and \(\beta:=(\pr_{\bar\iota})_\#\sigma\). Since
\(\alpha\ll \nu_\iota\!\restriction A_t\) and
\(\beta\ll \nu_{\bar\iota}^{\mathrm{ac}}\!\restriction B_t^\c\), let
\(h:=d\alpha/d\nu_\iota\) on \(A_t\) and \(k:=d\beta/d\nu_{\bar\iota}^{\mathrm{ac}}\) on \(B_t^\c\).
For \(n,m,\ell\in\mathbb N\), set
\[
  X_n:=\{x\in A_t:h(x)\le n\},
  \qquad
  Y_m:=\{y\in B_t^\c:k(y)\le m\},
  \qquad
  L_\ell:=\Bigl\{y\in B_t^\c:\rho_{\bar\iota}^{\mathrm{ac}}(y)\le \frac1t-\frac1\ell\Bigr\}.
\]
Since \(A_t=\bigcup_n X_n\) \(\nu_\iota\)-a.e., \(B_t^\c=\bigcup_m Y_m\)
\(\nu_{\bar\iota}^{\mathrm{ac}}\)-a.e., and \(B_t^\c=\bigcup_\ell L_\ell\), there exist
\(n,m,\ell\) such that
\[
  \widetilde\sigma:=\sigma\!\restriction\bigl(X_n\times(Y_m\cap L_\ell)\bigr)
\]
is nonzero. Set \(L:=Y_m\cap L_\ell\), and write
\(\widetilde\alpha:=(\pr_\iota)_\#\widetilde\sigma\) and
\(\widetilde\beta:=(\pr_{\bar\iota})_\#\widetilde\sigma\). Then
\begin{equation}
\label{eq:tilde_bounds_layer_compatibility_revised}
  \widetilde\alpha\le n\,\nu_\iota\!\restriction A_t,
  \qquad
  \widetilde\beta\le m\,\nu_{\bar\iota}^{\mathrm{ac}}\!\restriction L,
\end{equation}
and
\begin{equation}
\label{eq:rho_bar_strict_slack_layer_compatibility_revised}
  \rho_{\bar\iota}^{\mathrm{ac}}(y)\le \frac1t-\frac1\ell
  \qquad\text{for all }y\in L.
\end{equation}

Choose
\[
  t':=\frac1t-\frac{1}{2\ell},
  \qquad
  \varepsilon:=\frac{1}{2\ell},
  \qquad
  \kappa:=1-tt'>0.
\]
Then \(t'<1/t\), and \eqref{eq:rho_bar_strict_slack_layer_compatibility_revised} yields
\begin{equation}
\label{eq:rho_bar_slack_layer_compatibility_revised}
  \rho_{\bar\iota}^{\mathrm{ac}}(y)\le t'-\varepsilon
  \qquad\text{for all }y\in L.
\end{equation}

\paragraph{The competitor at level \(t'\).}
Let \(g:=d\nu_{\bar\iota}^{\mathrm{ac}}/dP_{\bar\iota}\), define
\(\theta(y):=\min\{1,t'g(y)\}\), and set
\[
  \sigma_0(E):=\int_E \theta(y)\,\pi_\iota(d(x,y)),
  \qquad
  E\in\Sigma_\iota\otimes\Sigma_{\bar\iota}.
\]
Then \(\sigma_0(\mathcal E_\iota^\c)=0\). Writing
\(\xi_0:=(\pr_\iota)_\#\sigma_0\) and \(\psi_0:=(\pr_{\bar\iota})_\#\sigma_0\), one has
\[
  \psi_0
  =
  \theta\,P_{\bar\iota}
  =
  (\rho_{\bar\iota}^{\mathrm{ac}}\wedge t')\,\nu_{\bar\iota},
\]
so \(\psi_0\le t'\nu_{\bar\iota}\). Since \(0\le \theta\le 1\), also \(\sigma_0\le \pi_\iota\), hence
\((\pr_\iota)_\#\sigma_0\le (\pr_\iota)_\#\pi_\iota=\nu_\iota\). Thus \(\sigma_0\in\Feas_{\bar\iota}(t')\).
By the level-optimal maximin property of \(\pi_\iota\),
\[
  \sigma_0(\Omega_\iota\times\Omega_{\bar\iota})
  =
  \int_{\Omega_{\bar\iota}}(\rho_{\bar\iota}^{\mathrm{ac}}\wedge t')\,d\nu_{\bar\iota}
  =
  \Fit_{\bar\iota}(t').
\]

We next estimate the source slack of \(\sigma_0\) on \(A_t\). For measurable \(A\subseteq A_t\),
\[
  \xi_0(A)
  =
  \int_{A\times\Omega_{\bar\iota}} \theta(y)\,\pi_\iota(d(x,y))
  \le
  t'\int_{A\times\Omega_{\bar\iota}} g(y)\,\pi_\iota(d(x,y)).
\]
By \eqref{eq:Piac_from_proportional_response}, the right-hand side is
\(t'P_\iota^{\mathrm{ac}}(A)\). Hence
\[
  \xi_0(A)\le t'\int_A \rho_\iota^{\mathrm{ac}}\,d\nu_\iota\le tt'\,\nu_\iota(A),
\]
because \(A\subseteq A_t=\{\rho_\iota^{\mathrm{ac}}\le t\}\). Therefore
\begin{equation}
\label{eq:source_slack_layer_compatibility_revised}
  \nu_\iota-\xi_0\ge \kappa\,\nu_\iota\!\restriction A_t.
\end{equation}

\paragraph{Adding a small incompatible subflow.}
Choose
\[
  \lambda:=\frac12\min\Bigl\{\frac{\kappa}{n},\,\frac{\varepsilon}{m}\Bigr\}>0,
  \qquad
  \eta:=\lambda\,\widetilde\sigma.
\]
Then \(\eta\neq 0\), \(\eta(\mathcal E_\iota^\c)=0\), and
\[
  (\pr_\iota)_\#\eta=\lambda\widetilde\alpha
  \le \lambda n\,\nu_\iota\!\restriction A_t
  \le \kappa\,\nu_\iota\!\restriction A_t
  \le \nu_\iota-\xi_0
\]
by \eqref{eq:tilde_bounds_layer_compatibility_revised} and
\eqref{eq:source_slack_layer_compatibility_revised}. Hence
\[
  (\pr_\iota)_\#(\sigma_0+\eta)\le \nu_\iota.
\]

For the opposite marginal,
\[
  (\pr_{\bar\iota})_\#\eta
  =\lambda\widetilde\beta
  \le \lambda m\,\nu_{\bar\iota}^{\mathrm{ac}}\!\restriction L
  \le \varepsilon\,\nu_{\bar\iota}\!\restriction L
\]
by \eqref{eq:tilde_bounds_layer_compatibility_revised}, while on \(L\),
\[
  \psi_0
  =(\rho_{\bar\iota}^{\mathrm{ac}}\wedge t')\,\nu_{\bar\iota}
  =\rho_{\bar\iota}^{\mathrm{ac}}\,\nu_{\bar\iota}
  \le (t'-\varepsilon)\,\nu_{\bar\iota}
\]
by \eqref{eq:rho_bar_slack_layer_compatibility_revised}. Hence
\[
  (\pr_{\bar\iota})_\#(\sigma_0+\eta)\le t'\nu_{\bar\iota}.
\]
Thus \(\sigma_0+\eta\in\Feas_{\bar\iota}(t')\).

Since \(\eta\neq 0\),
\[
  (\sigma_0+\eta)(\Omega_\iota\times\Omega_{\bar\iota})
  >
  \sigma_0(\Omega_\iota\times\Omega_{\bar\iota})
  =
  \Fit_{\bar\iota}(t'),
\]
contradicting the definition of \(\Fit_{\bar\iota}(t')\). Therefore \(\sigma=0\).
\end{proof}

The final structural lemma packages the preceding compatibility statements into full-measure
sets for the absolutely continuous and singular parts.  These are the form needed later
for the price construction in the continuum-economy section.

\begin{lemma}[Disjoint good representatives and null bad neighborhoods]
\label{lem:disjoint_good_supports_null_bad_neighborhoods}
Consider the symmetric density decomposition of
Definition~\ref{defn:sym_density_decomp}.

For each \(\iota\in\mathbb B\), there exist disjoint measurable sets
\[
  S_\iota^{\mathrm{ac}},\,S_\iota^\perp \in \Sigma_\iota
\]
such that
\(\nu_\iota^{\mathrm{ac}}(S_\iota^{\mathrm{ac}})
= \nu_\iota^{\mathrm{ac}}(\Omega_\iota)\) and
\(\nu_\iota^\perp(S_\iota^\perp)
= \nu_\iota^\perp(\Omega_\iota)\).

Moreover, each of the following subsets of \(\Omega_{\bar\iota}\) is contained in a measurable
\(\nu_{\bar\iota}\)-null set; equivalently, each is \(\nu_{\bar\iota}\)-null after passage to the
\(\nu_{\bar\iota}\)-completion:
\begin{enumerate}[(a)]
\item
\[
  \bigcup_{x\in S_\iota^{\mathrm{ac}}}
  \Bigl\{y\in N_{\bar\iota}(x):
  \rho_\iota^{\mathrm{ac}}(x)\rho_{\bar\iota}^{\mathrm{ac}}(y)<1\Bigr\},
\]
\item
\[
  N_{\bar\iota}(S_\iota^\perp).
\]
\end{enumerate}
\end{lemma}

\begin{proof}
Fix \(\iota\in\mathbb B\), and write \(\bar\iota:=1-\iota\). Let
\[
  \mathcal E_\iota
  :=
  \{(x,y)\in\Omega_\iota\times\Omega_{\bar\iota}:
    y\in N_{\bar\iota}(x)\}.
\]
Thus \(\mathcal E_0=\mathcal E\), whereas \(\mathcal E_1\) is the
transposed relation. For finite measures \(a\) on \(\Omega_\iota\) and
\(b\) on \(\Omega_{\bar\iota}\), let \(\mathcal F_\iota(a,b)\) denote
the flow value from
Proposition~\ref{prop:measurable_maxflow_mincut_bipartite} associated
with \(\mathcal E_\iota\).

We construct full-measure sets
\(S_\iota^\perp,S_\iota^{\mathrm{ac}}\subseteq\Omega_\iota\) with the
required properties. The singular part is handled by a zero-flow
argument and an attained minimum cut. The absolutely continuous part
is treated by rational density layers and the same minimum-cut
attainment result.

\noindent\textbf{Zero flow against singular payloads.}
We first show that
\[
  \mathcal F_\iota(\nu_\iota,\nu_{\bar\iota}^\perp)=0.
\]
Suppose otherwise, and let
\(\sigma\in\mathcal M_+(\Omega_\iota\times\Omega_{\bar\iota})\) be a
nonzero feasible flow. Write
\[
  \alpha:=(\pr_\iota)_\#\sigma,
  \qquad
  \beta:=(\pr_{\bar\iota})_\#\sigma.
\]
Disintegrating with respect to \(\beta\), write
\(\sigma=\beta\kprod\widehat K_\iota\), where
\(\widehat K_\iota(y,N_\iota(y))=1\) for \(\beta\)-a.e.\ \(y\).

By measurable selection, there is a probability kernel
\(H_\iota:\Omega_{\bar\iota}\times\Sigma_\iota\to[0,1]\) such that
\(H_\iota(y,N_\iota(y))=1\) for
\(\nu_{\bar\iota}^\perp\)-a.e.\ \(y\). Since
\(\mathcal E_\iota\) is measurable, the map
\[
  y\longmapsto
  \int_{\Omega_\iota}
  \mathbf 1_{\mathcal E_\iota}(x,y)\,
  \widehat K_\iota(y,dx)
\]
is measurable. 
Hence the set
\[
  Y_0
  :=
  \left\{
    y\in\Omega_{\bar\iota}:
    \widehat K_\iota(y,\Omega_\iota)=1
    \ \text{and}\
    \int_{\Omega_\iota}
    \mathbf 1_{\mathcal E_\iota}(x,y)\,
    \widehat K_\iota(y,dx)=1
  \right\}
\]
is measurable and \(\beta(Y_0^\c)=0\).
Define
\[
  K_\iota(y,\cdot)
  :=
  \begin{cases}
    \widehat K_\iota(y,\cdot), & y\in Y_0,\\
    H_\iota(y,\cdot), & y\in Y_0^\c.
  \end{cases}
\]
Then \(K_\iota\) is a probability kernel,
\(\sigma=\beta\kprod K_\iota\), and
\[
  K_\iota(y,N_\iota(y))=1
  \qquad
  \text{for }\nu_{\bar\iota}^\perp\text{-a.e.\ }y.
\]

Lemma~\ref{lem:singular_payload_stays_singular_under_backward_kernel}
therefore gives
\[
  (\pr_\iota)_\#
  \bigl(\nu_{\bar\iota}^\perp\kprod K_\iota\bigr)
  \perp\nu_\iota.
\]
Since \(\beta\le\nu_{\bar\iota}^\perp\), the measure on the left
dominates
\[
  (\pr_\iota)_\#(\beta\kprod K_\iota)
  =
  \alpha.
\]
Hence \(\alpha\perp\nu_\iota\). But feasibility also gives
\(\alpha\le\nu_\iota\), so \(\alpha=0\), contradicting
\(\sigma\neq0\). Thus
\(\mathcal F_\iota(\nu_\iota,\nu_{\bar\iota}^\perp)=0\).

\noindent\textbf{The singular carrier.}
Applying the preceding conclusion with \(\iota\) replaced by
\(\bar\iota\), we obtain
\[
  \mathcal F_{\bar\iota}
  (\nu_{\bar\iota},\nu_\iota^\perp)=0.
\]
Apply Lemma~\ref{lem:measurable_min_cut_attainment} to the relation
\(\mathcal E_{\bar\iota}\) and the measures
\((\nu_{\bar\iota},\nu_\iota^\perp)\). There exist
\(C_{\bar\iota}\in\Sigma_{\bar\iota}\) and
\(G_\iota\in\Sigma_\iota\) such that
\[
  N_\iota(C_{\bar\iota}^\c)\subseteq G_\iota,
  \qquad
  0
  =
  \nu_{\bar\iota}(C_{\bar\iota})
  +
  \nu_\iota^\perp(G_\iota).
\]
Both terms are nonnegative; hence
\[
  \nu_{\bar\iota}(C_{\bar\iota})=0,
  \qquad
  \nu_\iota^\perp(G_\iota)=0.
\]
Set \(\widetilde S_\iota^\perp:=G_\iota^\c\). Then
\(\widetilde S_\iota^\perp\in\Sigma_\iota\) and
\[
  \nu_\iota^\perp(\widetilde S_\iota^\perp)
  =
  \nu_\iota^\perp(\Omega_\iota).
\]
Moreover,
\[
  N_{\bar\iota}(\widetilde S_\iota^\perp)
  \subseteq C_{\bar\iota}.
\]
Indeed, if \(x\in\widetilde S_\iota^\perp\) were adjacent to some
\(y\in C_{\bar\iota}^\c\), then
\(x\in N_\iota(C_{\bar\iota}^\c)\subseteq G_\iota\), a contradiction.
Thus \(N_{\bar\iota}(\widetilde S_\iota^\perp)\) is contained in the
measurable \(\nu_{\bar\iota}\)-null set \(C_{\bar\iota}\). This proves
part~\emph{(b)} for \(\widetilde S_\iota^\perp\), and hence also for
the set \(S_\iota^\perp\) obtained by shrinking it below.

\noindent\textbf{Density layers.}
Let \(\pi_\iota\) be the level-optimal maximin refinement from side
\(\iota\) to side \(\bar\iota\) appearing in
Definition~\ref{defn:sym_density_decomp}, and set
\[
  P_{\bar\iota}:=(\pr_{\bar\iota})_\#\pi_\iota.
\]
Write
\(P_{\bar\iota}
 =P_{\bar\iota}^{\mathrm{ac}}+P_{\bar\iota}^\perp\)
with respect to \(\nu_{\bar\iota}\). By
Definition~\ref{defn:sym_density_decomp},
\(\nu_{\bar\iota}
 =\nu_{\bar\iota}^{\mathrm{ac}}+\nu_{\bar\iota}^\perp\)
is the Lebesgue decomposition with respect to
\(P_{\bar\iota}^{\mathrm{ac}}\). Since
\(\nu_{\bar\iota}^\perp\perp P_{\bar\iota}^{\mathrm{ac}}\), while
\(P_{\bar\iota}^\perp\perp\nu_{\bar\iota}\) and
\(\nu_{\bar\iota}^\perp\le\nu_{\bar\iota}\), it follows that
\[
  \nu_{\bar\iota}^\perp\perp P_{\bar\iota}.
\]

Define
\[
  D_\iota
  :=
  \{x\in\Omega_\iota:
    0<\rho_\iota^{\mathrm{ac}}(x)<\infty\}.
\]
Then
\(\nu_\iota^{\mathrm{ac}}(D_\iota)
 =\nu_\iota^{\mathrm{ac}}(\Omega_\iota)\).
Indeed, \(\rho_\iota^{\mathrm{ac}}<\infty\) \(\nu_\iota\)-a.e., while
\(\nu_\iota^{\mathrm{ac}}\ll
 P_\iota^{\mathrm{ac}}
 =\rho_\iota^{\mathrm{ac}}\nu_\iota\), so
\(\nu_\iota^{\mathrm{ac}}
 (\{\rho_\iota^{\mathrm{ac}}=0\})=0\).

For \(q\in\mathbb Q_{>0}\), set
\[
  A_q
  :=
  \{x\in\Omega_\iota:
    \rho_\iota^{\mathrm{ac}}(x)\le q\},
  \qquad
  B_q
  :=
  \left\{y\in\Omega_{\bar\iota}:
    \rho_{\bar\iota}^{\mathrm{ac}}(y)\ge\frac1q\right\},
\]
and define
\[
  a_q:=\nu_\iota^{\mathrm{ac}}\!\restriction A_q,
  \qquad
  b_q:=\nu_{\bar\iota}\!\restriction B_q^\c.
\]

We claim that \(\mathcal F_\iota(a_q,b_q)=0\) for every
\(q\in\mathbb Q_{>0}\). Fix \(q\), and let
\(\sigma\in\mathcal M_+(\Omega_\iota\times\Omega_{\bar\iota})\) satisfy
\[
  \sigma(\mathcal E_\iota^\c)=0,
  \qquad
  (\pr_\iota)_\#\sigma\le a_q,
  \qquad
  (\pr_{\bar\iota})_\#\sigma\le b_q.
\]
Choose \(Z\in\Sigma_{\bar\iota}\) such that
\[
  \nu_{\bar\iota}^\perp(Z^\c)=0,
  \qquad
  P_{\bar\iota}(Z)=0.
\]
Since
\(\nu_{\bar\iota}^{\mathrm{ac}}
 \ll P_{\bar\iota}^{\mathrm{ac}}
 \le P_{\bar\iota}\),
one also has \(\nu_{\bar\iota}^{\mathrm{ac}}(Z)=0\).

The second marginal of \(\sigma\) is carried by \(B_q^\c\), so
\[
  \sigma
  =
  \sigma\!\restriction
  \bigl(\Omega_\iota\times(Z^\c\cap B_q^\c)\bigr)
  +
  \sigma\!\restriction
  \bigl(\Omega_\iota\times(Z\cap B_q^\c)\bigr)
  =:
  \sigma^{\mathrm{ac}}+\sigma^\perp.
\]

The first marginal of \(\sigma^{\mathrm{ac}}\) is dominated by
\(a_q\le\nu_\iota\!\restriction A_q\). Because
\(\nu_{\bar\iota}^\perp(Z^\c)=0\), the restriction of
\(\nu_{\bar\iota}\) to \(Z^\c\) equals that of
\(\nu_{\bar\iota}^{\mathrm{ac}}\); hence the second marginal of
\(\sigma^{\mathrm{ac}}\) is dominated by
\(\nu_{\bar\iota}^{\mathrm{ac}}\!\restriction B_q^\c\).
Lemma~\ref{lem:reciprocal_densities_layer_compatibility} therefore
gives \(\sigma^{\mathrm{ac}}=0\).

The first marginal of \(\sigma^\perp\) is dominated by
\(a_q\le\nu_\iota\). Since
\(\nu_{\bar\iota}^{\mathrm{ac}}(Z)=0\), the restriction of
\(\nu_{\bar\iota}\) to \(Z\) equals that of
\(\nu_{\bar\iota}^\perp\); hence the second marginal of
\(\sigma^\perp\) is dominated by \(\nu_{\bar\iota}^\perp\).
The zero-flow principle proved above therefore gives
\(\sigma^\perp=0\). Thus \(\sigma=0\), proving
\[
  \mathcal F_\iota(a_q,b_q)=0.
\]

\noindent\textbf{Rational minimum cuts.}
For each \(q\in\mathbb Q_{>0}\), apply
Lemma~\ref{lem:measurable_min_cut_attainment} to the relation
\(\mathcal E_\iota\) and the measures \((a_q,b_q)\). Since
\(\mathcal F_\iota(a_q,b_q)=0\), there exist
\(C_q\in\Sigma_\iota\) and \(G_q\in\Sigma_{\bar\iota}\) such that
\[
  N_{\bar\iota}(C_q^\c)\subseteq G_q,
  \qquad
  0=a_q(C_q)+b_q(G_q).
\]
Consequently,
\[
  a_q(C_q)=0,
  \qquad
  b_q(G_q)=0.
\]
Set
\[
  S_q:=A_q\cap C_q^\c,
  \qquad
  M_q:=G_q\cap B_q^\c.
\]
Then
\[
  \nu_\iota^{\mathrm{ac}}(A_q\setminus S_q)
  =
  a_q(A_q\cap C_q)
  \le
  a_q(C_q)
  =
  0,
\]
and, since
\(b_q=\nu_{\bar\iota}\!\restriction B_q^\c\),
\[
  \nu_{\bar\iota}(M_q)=b_q(G_q)=0.
\]
Moreover,
\[
\begin{aligned}
  \bigcup_{x\in S_q}
  \bigl(N_{\bar\iota}(x)\cap B_q^\c\bigr)
  &\subseteq
  N_{\bar\iota}(C_q^\c)\cap B_q^\c \\
  &\subseteq
  G_q\cap B_q^\c
  =
  M_q.
\end{aligned}
\]

\noindent\textbf{The global absolutely continuous carrier.}
Define
\[
  \widetilde S_\iota^{\mathrm{ac}}
  :=
  D_\iota
  \cap
  \bigcap_{q\in\mathbb Q_{>0}}(A_q^\c\cup S_q).
\]
The complement of \(D_\iota\) is
\(\nu_\iota^{\mathrm{ac}}\)-null, and
\(A_q\setminus S_q\) is
\(\nu_\iota^{\mathrm{ac}}\)-null for every \(q\). Since
\(\mathbb Q_{>0}\) is countable,
\[
  \nu_\iota^{\mathrm{ac}}
  (\widetilde S_\iota^{\mathrm{ac}})
  =
  \nu_\iota^{\mathrm{ac}}(\Omega_\iota).
\]

Let
\[
  Y_{\mathrm{bad}}
  :=
  \bigcup_{x\in\widetilde S_\iota^{\mathrm{ac}}}
  \left\{
    y\in N_{\bar\iota}(x):
    \rho_\iota^{\mathrm{ac}}(x)
    \rho_{\bar\iota}^{\mathrm{ac}}(y)<1
  \right\}.
\]
We claim that
\[
  Y_{\mathrm{bad}}
  \subseteq
  \bigcup_{q\in\mathbb Q_{>0}}M_q.
\]
The set on the right is measurable and
\(\nu_{\bar\iota}\)-null, so this proves part~\emph{(a)}.

To verify the claim, let \(y\in Y_{\mathrm{bad}}\). Then there exists
\(x\in\widetilde S_\iota^{\mathrm{ac}}\) such that
\(y\in N_{\bar\iota}(x)\) and
\[
  \rho_\iota^{\mathrm{ac}}(x)
  \rho_{\bar\iota}^{\mathrm{ac}}(y)<1.
\]

Since \(x\in D_\iota\), one has
\(0<\rho_\iota^{\mathrm{ac}}(x)<\infty\). The strict inequality above
also implies
\(\rho_{\bar\iota}^{\mathrm{ac}}(y)<\infty\). If
\(0<\rho_{\bar\iota}^{\mathrm{ac}}(y)<\infty\), choose
\(q\in\mathbb Q_{>0}\) such that
\[
  \rho_\iota^{\mathrm{ac}}(x)
  \le q
  <
  \frac{1}{\rho_{\bar\iota}^{\mathrm{ac}}(y)}.
\]
If \(\rho_{\bar\iota}^{\mathrm{ac}}(y)=0\), choose any
\(q\in\mathbb Q_{>0}\) with
\(q\ge\rho_\iota^{\mathrm{ac}}(x)\).

In either case,
\(x\in A_q\) and \(y\in B_q^\c\). Since
\[
  x\in
  \widetilde S_\iota^{\mathrm{ac}}
  \subseteq
  A_q^\c\cup S_q,
\]
it follows that \(x\in S_q\). Therefore
\[
  y
  \in
  \bigcup_{x'\in S_q}
  \bigl(N_{\bar\iota}(x')\cap B_q^\c\bigr)
  \subseteq
  M_q,
\]
which proves the claim.

\noindent\textbf{Disjoint representatives.}
Since
\(\nu_\iota^{\mathrm{ac}}\perp\nu_\iota^\perp\), choose disjoint
measurable sets
\(T_\iota^{\mathrm{ac}},T_\iota^\perp\in\Sigma_\iota\) such that
\[
  \nu_\iota^{\mathrm{ac}}(T_\iota^{\mathrm{ac}})
  =
  \nu_\iota^{\mathrm{ac}}(\Omega_\iota),
  \qquad
  \nu_\iota^\perp(T_\iota^\perp)
  =
  \nu_\iota^\perp(\Omega_\iota).
\]
Set
\[
  S_\iota^{\mathrm{ac}}
  :=
  \widetilde S_\iota^{\mathrm{ac}}
  \cap T_\iota^{\mathrm{ac}},
  \qquad
  S_\iota^\perp
  :=
  \widetilde S_\iota^\perp
  \cap T_\iota^\perp.
\]
Then \(S_\iota^{\mathrm{ac}}\) and \(S_\iota^\perp\) are disjoint and
have full \(\nu_\iota^{\mathrm{ac}}\)- and
\(\nu_\iota^\perp\)-measure, respectively. The two null-neighborhood
properties persist because both sets were only shrunk.

Since \(\iota\in\mathbb B\) was arbitrary, the proof is complete.
\end{proof}

The preceding lemmas isolate the singular and absolutely continuous
layers geometrically. For the equilibrium construction, one needs a
stronger compatibility statement: the source-side Lebesgue
decomposition of a level-optimal maximin plan must be transported
exactly to the corresponding target-side decomposition. The next
lemma records this balance and strengthens the null-neighborhood
condition of
Lemma~\ref{lem:disjoint_good_supports_null_bad_neighborhoods}
to reciprocal equality on the absolutely continuous component.

\begin{lemma}[Balanced absolutely continuous and singular transport]
\label{lem:balanced_active_singular_transport}
Work under Assumption~\ref{assump:closed_polish_graph_nonempty_nbhd},
and consider the symmetric density decomposition of
Definition~\ref{defn:sym_density_decomp}. Fix
\(\iota\in\mathbb B\), and write \(\bar\iota:=1-\iota\).

Choose a forward disintegration
\[
  \pi_\iota=\nu_\iota\kprod G_\iota,
\]
and define
\[
  \pi_\iota^{\mathrm{ac}}
  :=
  \nu_\iota^{\mathrm{ac}}\kprod G_\iota,
  \qquad
  \pi_\iota^\perp
  :=
  \nu_\iota^\perp\kprod G_\iota.
\]
Then
\[
  (\pr_{\bar\iota})_\#\pi_\iota^{\mathrm{ac}}
  =
  P_{\bar\iota}^{\mathrm{ac}},
  \qquad
  (\pr_{\bar\iota})_\#\pi_\iota^\perp
  =
  P_{\bar\iota}^{\perp}.
\]
Moreover,
\[
  \rho_\iota^{\mathrm{ac}}(x)
  \rho_{\bar\iota}^{\mathrm{ac}}(y)
  =
  1
\]
for \(\pi_\iota^{\mathrm{ac}}\)-a.e.\
\((x,y)\in\Omega_\iota\times\Omega_{\bar\iota}\).
\end{lemma}

\begin{proof}
Set
\[
  Q_{\bar\iota}^{\mathrm{ac}}
  :=
  (\pr_{\bar\iota})_\#\pi_\iota^{\mathrm{ac}},
  \qquad
  Q_{\bar\iota}^{\perp}
  :=
  (\pr_{\bar\iota})_\#\pi_\iota^\perp.
\]
Applying
Lemma~\ref{lem:singular_payload_stays_singular_under_backward_kernel}
with the indices exchanged gives
\[
  Q_{\bar\iota}^{\perp}\perp\nu_{\bar\iota}.
\]
Since
\[
  P_{\bar\iota}
  =
  Q_{\bar\iota}^{\mathrm{ac}}
  +
  Q_{\bar\iota}^{\perp},
\]
the absolutely continuous part of
\(Q_{\bar\iota}^{\mathrm{ac}}\) with respect to
\(\nu_{\bar\iota}\) is \(P_{\bar\iota}^{\mathrm{ac}}\).

We next compare total masses. Let
\(\widetilde\pi_\iota\) be a proportional response to
\(\pi_{\bar\iota}\). Since \(\pi_{\bar\iota}\) is level-optimal
maximin,
Theorem~\ref{thm:universal_closest_refinement_pair} shows that
\((\pi_{\bar\iota},\widetilde\pi_\iota)\) is universally closest.
Applying
Theorem~\ref{thm:strict_convex_closest_pair_induces_canonical_pair}
with the bounded strictly convex integrand
\(\vartheta(u)=e^{-u}\), together with the symmetry remark following
that theorem, shows that \(\widetilde\pi_\iota\) is level-optimal
maximin from side \(\iota\) to side \(\bar\iota\).

Write the proportional response as
\[
  \widetilde\pi_\iota
  =
  \nu_\iota^{\mathrm{ac}}\kprod K
  +
  \nu_\iota^\perp\kprod K^{\mathrm{fb}},
\]
where \(K\) is obtained from a reverse disintegration of
\(\pi_{\bar\iota}\). Since
\(\nu_\iota^{\mathrm{ac}}\ll P_\iota\), the first summand is
absolutely continuous with respect to \(\pi_{\bar\iota}\), and its
\(\bar\iota\)-marginal is therefore absolutely continuous with
respect to \(\nu_{\bar\iota}\). By
Lemma~\ref{lem:singular_payload_stays_singular_under_backward_kernel},
with the indices exchanged, the \(\bar\iota\)-marginal of the fallback
summand is singular with respect to \(\nu_{\bar\iota}\). Thus the
absolutely continuous part of the payload of
\(\widetilde\pi_\iota\) has total mass
\[
  \nu_\iota^{\mathrm{ac}}(\Omega_\iota).
\]

Both \(\widetilde\pi_\iota\) and \(\pi_\iota\) are level-optimal
maximin from side \(\iota\) to side \(\bar\iota\). Hence
Corollary~\ref{cor:unique_ac_opposite_marginal_level_opt_maximin}
identifies their absolutely continuous payloads. Consequently,
\[
  P_{\bar\iota}^{\mathrm{ac}}(\Omega_{\bar\iota})
  =
  \nu_\iota^{\mathrm{ac}}(\Omega_\iota).
\]
On the other hand,
\[
  Q_{\bar\iota}^{\mathrm{ac}}(\Omega_{\bar\iota})
  =
  \pi_\iota^{\mathrm{ac}}
  (\Omega_\iota\times\Omega_{\bar\iota})
  =
  \nu_\iota^{\mathrm{ac}}(\Omega_\iota).
\]
Since the absolutely continuous part of
\(Q_{\bar\iota}^{\mathrm{ac}}\) is
\(P_{\bar\iota}^{\mathrm{ac}}\), equality of their total masses
forces
\[
  Q_{\bar\iota}^{\mathrm{ac}}
  =
  P_{\bar\iota}^{\mathrm{ac}}.
\]
Subtracting from \(P_{\bar\iota}\) gives
\[
  Q_{\bar\iota}^{\perp}
  =
  P_{\bar\iota}^{\perp}.
\]

It remains to prove reciprocity. Let
\(S_\iota^{\mathrm{ac}}\) be the full-measure set supplied by
Lemma~\ref{lem:disjoint_good_supports_null_bad_neighborhoods}.
The source marginal of \(\pi_\iota^{\mathrm{ac}}\) is
\(\nu_\iota^{\mathrm{ac}}\), while its opposite marginal is
\(P_{\bar\iota}^{\mathrm{ac}}\ll\nu_{\bar\iota}\). That lemma
therefore gives
\[
  \rho_\iota^{\mathrm{ac}}(x)
  \rho_{\bar\iota}^{\mathrm{ac}}(y)
  \ge 1
\]
for \(\pi_\iota^{\mathrm{ac}}\)-a.e.\ \((x,y)\). In particular,
\[
  \rho_\iota^{\mathrm{ac}}(x)
  \ge
  \frac{1}{\rho_{\bar\iota}^{\mathrm{ac}}(y)}
\]
almost everywhere with respect to \(\pi_\iota^{\mathrm{ac}}\).

Using the two marginals of \(\pi_\iota^{\mathrm{ac}}\), one has
\[
  \int
  \rho_\iota^{\mathrm{ac}}(x)\,
  d\pi_\iota^{\mathrm{ac}}(x,y)
  =
  \int_{\Omega_\iota}
  \rho_\iota^{\mathrm{ac}}\,
  d\nu_\iota^{\mathrm{ac}}
  =
  P_\iota^{\mathrm{ac}}(\Omega_\iota),
\]
whereas
\[
  \int
  \frac{1}{\rho_{\bar\iota}^{\mathrm{ac}}(y)}\,
  d\pi_\iota^{\mathrm{ac}}(x,y)
  =
  \int_{\Omega_{\bar\iota}}
  \frac{1}{\rho_{\bar\iota}^{\mathrm{ac}}}\,
  dP_{\bar\iota}^{\mathrm{ac}}
  =
  \nu_{\bar\iota}^{\mathrm{ac}}
  (\Omega_{\bar\iota}).
\]
Applying the preceding mass identity with \(\iota\) replaced by
\(\bar\iota\) yields
\[
  P_\iota^{\mathrm{ac}}(\Omega_\iota)
  =
  \nu_{\bar\iota}^{\mathrm{ac}}
  (\Omega_{\bar\iota}).
\]
The two displayed integrals are therefore equal. Since their
integrands satisfy the stated pointwise inequality, equality holds
almost everywhere:
\[
  \rho_\iota^{\mathrm{ac}}(x)
  =
  \frac{1}{\rho_{\bar\iota}^{\mathrm{ac}}(y)}
  \qquad
  \pi_\iota^{\mathrm{ac}}\text{-a.e.}
\]
This is the required reciprocal identity.
\end{proof}

\section{Equilibrium-Theoretic Characterization via a Continuum Economy}
\label{sec:continuum_agents_commodities}

The preceding sections identified level-optimal maximin refinements as the central one-sided objects and showed that, together with proportional response, they yield universally closest refinement pairs. We now characterize pairs that are level-optimal maximin on both sides through a naturally associated continuum economy. Refinement plans become allocations of measure-valued commodities, while the bipartite relation enters agents' preferences through utility. We prove both directions: level-optimal maximin pairs yield Walrasian equilibria, and every Walrasian equilibrium yields a level-optimal maximin pair. Proportional response is not required.

\subsection{Continuum-Economy Framework}
\label{subsec:continuum_agents_commodities}

We record only the equilibrium terminology needed below. The continuum-agent framework is classical; see
\cite{Aumann1964,Aumann1966,Hildenbrand1974}. For infinite-dimensional commodity spaces, see
\cite{GretskyOstroyZame1992,GretskyOstroyZame1999}, and for broader context on infinite-dimensional and nonconvex models, see
\cite{Podczeck1997,TOURKY2001189}. We use no general equilibrium existence theorem; all equilibrium assertions are verified directly from the refinement structure.

\begin{definition}[Continuum Economy (Dual-Pair Form)]
\label{defn:continuum_economy_dual_pair}
A \emph{continuum economy} is a tuple
\(\mathfrak{E}=\bigl[(T,\Sigma,\mu),\ (L,\tau,L_+,E),\ X,\ \{\succeq_t\}_{t\in T},\ e\bigr]\),
where:
\begin{enumerate}\setlength{\itemsep}{0.2em}
\item \textbf{Agents.} \((T,\Sigma,\mu)\) is a finite measure space.

\item \textbf{Commodities and prices.} \((L,\tau)\) is a Hausdorff locally convex space, and
\(E\subseteq L^*\) is a separating vector subspace of continuous linear functionals, with pairing
\(\langle p,z\rangle:=p(z)\) for \(p\in E\), \(z\in L\).
The convex order cone \(L_+\subseteq L\) induces
\(z\le z'\Longleftrightarrow z'-z\in L_+\), and
\(E_+:=\{p\in E:\langle p,z\rangle\ge0\ \forall z\in L_+\}\).

\item \textbf{Consumption sets.} \(X:T\rightrightarrows L_+\) satisfies \(X(t)\neq\emptyset\) for all \(t\). The required measurability and closedness hypotheses will be imposed directly on the specialized model.

\item \textbf{Preferences.} For each \(t\in T\), \(\succeq_t\) is a preference relation on \(X(t)\), not assumed convex; \(z\succ_t z'\) denotes strict preference.

\item \textbf{Endowments.} \(e:T\to L_+\) satisfies \(e(t)\in X(t)\) for all \(t\), and is \(E\)-scalarly measurable and \(E\)-integrable: for every \(p\in E\), the map \(t\mapsto\langle p,e(t)\rangle\) is \(\mu\)-measurable and integrable, and there exists \(\int_Te\,d\mu\in L\) such that
\(\bigl\langle p,\int_Te\,d\mu\bigr\rangle=\int_T\langle p,e(t)\rangle\,\mu(dt)\) for every \(p\in E\).
\end{enumerate}
\end{definition}

\begin{definition}[Allocation and Feasibility]
\label{defn:continuum_allocation_feasible}
In the economy of Definition~\ref{defn:continuum_economy_dual_pair}, an \emph{allocation} is an \(E\)-scalarly measurable map \(a:T\to L_+\) such that \(a(t)\in X(t)\) for \(\mu\)-a.e.\ \(t\), and \(a\) is \(E\)-integrable in the same sense as \(e\). We write \(a_t:=a(t)\). It is \emph{feasible} if
\(\int_Ta\,d\mu=\int_Te\,d\mu\) in \(L\).
\end{definition}

\begin{definition}[Budget Set and Walrasian Equilibrium]
\label{defn:walras_equilibrium_dual_pair}
For \(p\in E_+\setminus\{0\}\), the \emph{budget set} of agent \(t\) is
\(\mathcal B(t,p):=\{z\in X(t):\langle p,z\rangle\le\langle p,e(t)\rangle\}\).
A pair \((a,p)\) is a \emph{Walrasian equilibrium} if:
\begin{enumerate}\setlength{\itemsep}{0.2em}
\item \textbf{Feasibility:} \(a\) is feasible in the sense of Definition~\ref{defn:continuum_allocation_feasible};
\item \textbf{Individual optimality:} for \(\mu\)-a.e.\ \(t\), \(a(t)\in\mathcal B(t,p)\), and no \(z\in\mathcal B(t,p)\) satisfies \(z\succ_t a(t)\).
\end{enumerate}
\end{definition}

\begin{remark}[Measure-Valued Commodities]
\label{rem:measure_valued_specialization}
In Subsection~\ref{subsec:graph_to_continuum_economy} we take
\(L:=\mathcal M(T)\), \(L_+:=\mathcal M_+(T)\), where \(\mathcal M(T)\) is the space of finite signed measures on \((T,\Sigma)\), equip \(L\) with
\(\tau:=\sigma\bigl(\mathcal M(T),B_b(T)\bigr)\), and take \(E:=B_b(T)\) with
\(\langle p,\eta\rangle=\int_Tp\,d\eta\).
Thus a consumption bundle is a finite measure of holdings of commodities indexed by \(T\); the bipartite relation determines which opposite-side commodities contribute to utility. Dirac endowments represent ownership of one unit of oneself even when \((T,\Sigma,\mu)\) is atomless.
\end{remark}

\subsection{Embedding the Measurable Bipartite Relation Model}
\label{subsec:graph_to_continuum_economy}

We specialize the preceding framework to the measurable bipartite relation model of Section~\ref{sec:prelim}. The relation enters utility rather than feasibility: every nonnegative measure-valued bundle is available, but an agent values only mass assigned to adjacent agents on the opposite side. Throughout, \(\iota\in\mathbb B:=\{0,1\}\) and \(\bar\iota:=1-\iota\).

\paragraph{Bipartite relation data.}
The datum consists of measurable spaces \((\Omega_\iota,\Sigma_\iota)\), finite measures
\(\nu_\iota\in\mathcal M_+(\Omega_\iota)\) with \(\nu_\iota(\Omega_\iota)>0\), and a measurable edge relation
\(\mathcal E\subseteq\Omega_0\times\Omega_1\), with
\(N_1(x):=\{y\in\Omega_1:(x,y)\in\mathcal E\}\) and
\(N_0(y):=\{x\in\Omega_0:(x,y)\in\mathcal E\}\).
We continue to work under Assumption~\ref{assump:closed_polish_graph_nonempty_nbhd}.

\begin{remark}[On the nonempty-neighborhood assumption]
The nonempty-neighborhood part of Assumption~\ref{assump:closed_polish_graph_nonempty_nbhd} is unnecessary for defining the economy: an agent with empty opposite-side neighborhood simply has zero utility, while its endowment remains feasible. It is used, however, in Construction~\ref{const:extract_refinement_from_allocation} to assign residual mass through graph-supported fallback kernels, so we retain it throughout.
\end{remark}

\paragraph{Agents, commodities, and prices.}
Set \(T:=\Omega_0\sqcup\Omega_1\), \(\Sigma:=\Sigma_0\oplus\Sigma_1:=\{A_0\sqcup A_1:A_0\in\Sigma_0,\ A_1\in\Sigma_1\}\), and
\(\mu:=\nu_0\oplus\nu_1\), so that
\(\mu(A_0\sqcup A_1):=\nu_0(A_0)+\nu_1(A_1)\).
Under Assumption~\ref{assump:closed_polish_graph_nonempty_nbhd}, \((T,\Sigma)\) is standard Borel.

Take \(L:=\mathcal M(T)\), \(L_+:=\mathcal M_+(T)\),
\(\tau:=\sigma\bigl(\mathcal M(T),B_b(T)\bigr)\), and \(E:=B_b(T)\), with
\(\langle p,\eta\rangle:=\int_Tp\,d\eta\).
Every \(\eta\in L\) decomposes uniquely as \(\eta=\eta_0\oplus\eta_1\), with
\(\eta_\iota\in\mathcal M(\Omega_\iota)\), and
\(E_+=\{p\in B_b(T):p\ge0\}\).

\paragraph{Consumption sets and endowments.}
For every \(t\in T\), set \(X(t):=L_+=\mathcal M_+(T)\). Thus every agent may consume any nonnegative measure-valued bundle; adjacency affects only utility.

Define \(e(t):=\delta_t\in L_+\). Then \(e(t)\in X(t)\) and
\(\langle p,\delta_t\rangle=p(t)\) for \(p\in E\), so \(e\) is \(E\)-scalarly measurable. Moreover,
\(\bigl\langle p,\int_Te\,d\mu\bigr\rangle
=\int_T\langle p,\delta_t\rangle\,\mu(dt)
=\int_Tp(t)\,\mu(dt)
=\langle p,\mu\rangle\) for every \(p\in E\); hence
\(\int_Te\,d\mu=\mu\) in \(L\).

Accordingly, an allocation is an \(E\)-scalarly measurable and \(E\)-integrable map
\(a:T\to L_+\), \(t\mapsto a_t:=a(t)\), and feasibility is
\(\int_Ta\,d\mu=\mu\) in \(L\), equivalently
\(\int_Ta_t(B)\,\mu(dt)=\mu(B)\) for every \(B\in\Sigma\), since
\(\mathbf1_B\in B_b(T)\).

\paragraph{Utilities and preferences.}
For \(x\in\Omega_0\) and \(z=z_0\oplus z_1\in X(x)\), define
\begin{equation}
\label{eq:utility_side0}
u_x(z)
:=
\int_{\Omega_1}\mathbf 1_{\mathcal E}(x,y)\,z_1(dy)
=
z_1\bigl(N_1(x)\bigr).
\end{equation}
For \(y\in\Omega_1\) and \(z=z_0\oplus z_1\in X(y)\), define
\begin{equation}
\label{eq:utility_side1}
u_y(z)
:=
\int_{\Omega_0}\mathbf 1_{\mathcal E}(x,y)\,z_0(dx)
=
z_0\bigl(N_0(y)\bigr).
\end{equation}
These are well defined because \(\mathcal E\) and its sections are measurable and the measures are finite. Let \(z\succeq_tz'\Longleftrightarrow u_t(z)\ge u_t(z')\). Thus same-side mass and nonadjacent opposite-side mass are utility-null.

\paragraph{The induced continuum economy.}
The instance
\(\bigl[(\Omega_0,\Sigma_0,\nu_0),(\Omega_1,\Sigma_1,\nu_1),\mathcal E\bigr]\)
induces
\(\mathfrak E_{\mathcal E}
:=\bigl[(T,\Sigma,\mu),(L,\tau,L_+,E),X,\{\succeq_t\}_{t\in T},e\bigr]\)
in the sense of Definition~\ref{defn:continuum_economy_dual_pair}.

\subsection{From Level-Optimal Maximin Pairs to Walrasian Equilibrium}
\label{subsec:symmetric_level_optimal_pair_to_equilibrium}

Subsection~\ref{subsec:symmetric_density_decomposition_support_geometry} supplies the structural input. The symmetric density decomposition separates bilateral trade from singular mass; Lemma~\ref{lem:singular_payload_stays_singular_under_backward_kernel} shows that singular mass cannot be transported backward into an absolutely continuous payload, so it is retained rather than traded. Lemma~\ref{lem:reciprocal_densities_layer_compatibility} gives the reciprocal density constraint underlying the price transform \(r\mapsto r/(1+r)\), while Lemma~\ref{lem:disjoint_good_supports_null_bad_neighborhoods} supplies full-measure representatives and measurable null exceptional sets on which the price may be raised.

Using these ingredients, we construct an allocation--price pair from any refinement pair whose two components are level-optimal maximin. No proportional-response relation between the two plans is assumed.

\begin{construction}[Allocation--price pair induced by a level-optimal maximin pair]
\label{const:allocation_price_from_refinement_pair}
Work under Assumption~\ref{assump:closed_polish_graph_nonempty_nbhd} and in the setting of
Subsection~\ref{subsec:graph_to_continuum_economy}.

Let \((\pi_0,\pi_1)\in\Pi_0^{\mathcal E}\times\Pi_1^{\mathcal E}\) be a refinement pair such that both components are level-optimal maximin in the sense of Definition~\ref{defn:level_opt_maximin}. Let
\(\nu_\iota=\nu_\iota^{\mathrm{ac}}+\nu_\iota^\perp\), \(\iota\in\mathbb B\), be the symmetric density decomposition of Definition~\ref{defn:sym_density_decomp}, with densities
\(\rho_\iota^{\mathrm{ac}}:\Omega_\iota\to[0,\infty)\). Choose disjoint measurable
\(S_\iota^{\mathrm{ac}},S_\iota^\perp\in\Sigma_\iota\) as in Lemma~\ref{lem:disjoint_good_supports_null_bad_neighborhoods}.

\paragraph{Absolutely continuous transport kernels.}
Fix forward disintegrations
\(\pi_0=\nu_0\kprod G_0\) and \(\pi_1=\nu_1\kprod G_1\),
where \(G_0(x,\cdot)\) is carried by \(N_1(x)\) for \(\nu_0\)-a.e.\ \(x\), and
\(G_1(y,\cdot)\) by \(N_0(y)\) for \(\nu_1\)-a.e.\ \(y\). Set
\(\pi_0^{\mathrm{ac}}:=\nu_0^{\mathrm{ac}}\kprod G_0\) and
\(\pi_1^{\mathrm{ac}}:=\nu_1^{\mathrm{ac}}\kprod G_1\).
By Lemma~\ref{lem:balanced_active_singular_transport},
\((\pr_1)_\#\pi_0^{\mathrm{ac}}=P_1^{\mathrm{ac}}=\rho_1^{\mathrm{ac}}\nu_1^{\mathrm{ac}}\) and
\((\pr_0)_\#\pi_1^{\mathrm{ac}}=P_0^{\mathrm{ac}}=\rho_0^{\mathrm{ac}}\nu_0^{\mathrm{ac}}\).
Hence, by disintegration, there are finite kernels
\(K_0:\Omega_1\times\Sigma_0\to[0,\infty)\) and
\(K_1:\Omega_0\times\Sigma_1\to[0,\infty)\) such that
\(\pi_0^{\mathrm{ac}}=\nu_1^{\mathrm{ac}}\kprod K_0\) and
\(\pi_1^{\mathrm{ac}}=\nu_0^{\mathrm{ac}}\kprod K_1\).
They are carried by \(N_0(y)\) and \(N_1(x)\), respectively, for the corresponding absolutely continuous source measures.

\paragraph{Allocation.}
Assign the absolutely continuous part through these reverse kernels and retain the singular part:
\[
a(x):=
\begin{cases}
0\oplus K_1(x,\cdot), & x\in S_0^{\mathrm{ac}},\\
\delta_x\oplus0, & x\in S_0^\perp,
\end{cases}
\qquad
a(y):=
\begin{cases}
K_0(y,\cdot)\oplus0, & y\in S_1^{\mathrm{ac}},\\
0\oplus\delta_y, & y\in S_1^\perp.
\end{cases}
\]
The remaining points form a \(\mu\)-null set; there set \(a(t):=\delta_t\).

\paragraph{Preliminary price.}
For \(\iota\in\mathbb B\), set
\(R_\iota:=\Omega_\iota\setminus(S_\iota^{\mathrm{ac}}\cup S_\iota^\perp)\), so \(\nu_\iota(R_\iota)=0\). Define
\[
\widetilde p(x):=
\begin{cases}
\dfrac{\rho_0^{\mathrm{ac}}(x)}{1+\rho_0^{\mathrm{ac}}(x)}, & x\in S_0^{\mathrm{ac}},\\[0.6em]
0, & x\in S_0^\perp\cup R_0,
\end{cases}
\qquad
\widetilde p(y):=
\begin{cases}
\dfrac{\rho_1^{\mathrm{ac}}(y)}{1+\rho_1^{\mathrm{ac}}(y)}, & y\in S_1^{\mathrm{ac}},\\[0.6em]
0, & y\in S_1^\perp\cup R_1.
\end{cases}
\]

\paragraph{Null bad neighborhoods and final price.}
For each \(\iota\in\mathbb B\), Lemma~\ref{lem:disjoint_good_supports_null_bad_neighborhoods} gives a measurable \(\nu_{\bar\iota}\)-null set \(M_{\bar\iota}\in\Sigma_{\bar\iota}\) such that
\[
\bigcup_{x\in S_\iota^{\mathrm{ac}}}
\Bigl\{
y\in N_{\bar\iota}(x):
\rho_\iota^{\mathrm{ac}}(x)\rho_{\bar\iota}^{\mathrm{ac}}(y)<1
\Bigr\}
\subseteq M_{\bar\iota},
\qquad
N_{\bar\iota}(S_\iota^\perp)\subseteq M_{\bar\iota}.
\]
Define
\[
p(t):=
\begin{cases}
2, & t\in M_0\cup M_1\cup R_0\cup R_1,\\
\widetilde p(t), & \text{otherwise}.
\end{cases}
\]

\paragraph{Output.}
The construction returns \((a,p)\).
\end{construction}

\begin{theorem}[Level-optimal maximin construction yields a Walrasian equilibrium]
\label{thm:refinement_pair_construction_walras}
Work under Assumption~\ref{assump:closed_polish_graph_nonempty_nbhd} and in the setting of
Subsection~\ref{subsec:graph_to_continuum_economy}.

Let \((\pi_0,\pi_1)\in\Pi^{\mathcal E}(\nu_0,\nu_1)\) be a refinement pair whose two components are level-optimal maximin in the sense of Definition~\ref{defn:level_opt_maximin}, and let \((a,p)\) be produced by Construction~\ref{const:allocation_price_from_refinement_pair}. Then \((a,p)\) is a Walrasian equilibrium of \(\mathfrak E_{\mathcal E}\).
\end{theorem}

\begin{proof}
Write \(\rho_\iota:=\rho_\iota^{\mathrm{ac}}\), and let \(M_0,M_1\) be the null sets from Construction~\ref{const:allocation_price_from_refinement_pair}. Set
\(R_\iota:=\Omega_\iota\setminus(S_\iota^{\mathrm{ac}}\cup S_\iota^\perp)\).
Since \(\nu_\iota(R_\iota)=0\) and \(p=2\) on \(R_0\cup R_1\), enlarge \(M_\iota\) by \(R_\iota\) and assume \(R_\iota\subseteq M_\iota\).

Choose measurable full-measure sets
\(\widehat S_\iota^{\mathrm{ac}}\subseteq S_\iota^{\mathrm{ac}}\setminus M_\iota\) and
\(\widehat S_\iota^\perp\subseteq S_\iota^\perp\setminus M_\iota\) such that
\(\nu_\iota(S_\iota^{\mathrm{ac}}\setminus\widehat S_\iota^{\mathrm{ac}})
=\nu_\iota(S_\iota^\perp\setminus\widehat S_\iota^\perp)=0\),
and, for every \(x\in\widehat S_0^{\mathrm{ac}}\),
\[
K_1(x,\Omega_1\setminus N_1(x))
=
K_1(x,M_1)
=
K_1(x,\Omega_1\setminus S_1^{\mathrm{ac}})
=
0,
\]
\[
K_1(x,\Omega_1)=\rho_0(x),
\qquad
K_1\bigl(x,\{y\in S_1^{\mathrm{ac}}:\rho_0(x)\rho_1(y)\neq1\}\bigr)=0,
\]
and symmetrically, for every \(y\in\widehat S_1^{\mathrm{ac}}\),
\[
K_0(y,\Omega_0\setminus N_0(y))
=
K_0(y,M_0)
=
K_0(y,\Omega_0\setminus S_0^{\mathrm{ac}})
=
0,
\]
\[
K_0(y,\Omega_0)=\rho_1(y),
\qquad
K_0\bigl(y,\{x\in S_0^{\mathrm{ac}}:\rho_0(x)\rho_1(y)\neq1\}\bigr)=0.
\]
Indeed, the carrier properties follow from
\(\pi_0^{\mathrm{ac}},\pi_1^{\mathrm{ac}}\le\pi_0,\pi_1\) and
\(\pi_0(\mathcal E^\c)=\pi_1(\mathcal E^\c)=0\).
Since \(S_\iota^{\mathrm{ac}}\) carries \(\nu_\iota^{\mathrm{ac}}\) and \(M_\iota\) is \(\nu_\iota\)-null, disintegration yields the corresponding null-kernel assertions. The identities
\(K_1(x,\Omega_1)=\rho_0(x)\) and \(K_0(y,\Omega_0)=\rho_1(y)\) follow from
\((\pr_0)_\#\pi_1^{\mathrm{ac}}=\rho_0\nu_0^{\mathrm{ac}}\) and
\((\pr_1)_\#\pi_0^{\mathrm{ac}}=\rho_1\nu_1^{\mathrm{ac}}\), while Lemma~\ref{lem:balanced_active_singular_transport} gives
\(\rho_0(x)\rho_1(y)=1\) for \(\pi_0^{\mathrm{ac}}\)- and \(\pi_1^{\mathrm{ac}}\)-a.e.\ \((x,y)\).

\paragraph{Feasibility and positivity of the price.}
By construction \(a(t)\in X(t)=L_+\). Measurability of \(K_0,K_1\) and the measurable piecewise definition of \(a\) imply \(E\)-scalar measurability. Likewise \(p\) is measurable, nonnegative, and bounded by \(2\), hence \(p\in E_+\).

For \(A_1\in\Sigma_1\),
\[
\int_Ta_t(A_1)\,\mu(dt)
=
\int_{S_0^{\mathrm{ac}}}K_1(x,A_1)\,\nu_0^{\mathrm{ac}}(dx)
+\nu_1^\perp(A_1)
=
\pi_1^{\mathrm{ac}}(\Omega_0\times A_1)+\nu_1^\perp(A_1)
=
\nu_1(A_1),
\]
and symmetrically, for \(A_0\in\Sigma_0\),
\[
\int_Ta_t(A_0)\,\mu(dt)
=
\int_{S_1^{\mathrm{ac}}}K_0(y,A_0)\,\nu_1^{\mathrm{ac}}(dy)
+\nu_0^\perp(A_0)
=
\pi_0^{\mathrm{ac}}(A_0\times\Omega_1)+\nu_0^\perp(A_0)
=
\nu_0(A_0).
\]
Thus \(\int_Ta_t(B)\,\mu(dt)=\mu(B)\) for every \(B\in\Sigma\). In particular,
\(\int_Ta_t(T)\,\mu(dt)=\mu(T)<\infty\), and for \(q\in B_b(T)\),
\(\int_T|\langle q,a(t)\rangle|\,\mu(dt)\le\lVert q\rVert_\infty\mu(T)\).
Moreover,
\(\int_T\langle q,a(t)\rangle\,\mu(dt)=\int_Tq\,d\mu=\langle q,\mu\rangle\).
Hence \(a\) is \(E\)-integrable and
\(\int_Ta\,d\mu=\mu=\int_Te\,d\mu\).

To see \(p\not\equiv0\), if \(\nu_0^{\mathrm{ac}}(\Omega_0)>0\), then
\((\pr_1)_\#\pi_0^{\mathrm{ac}}=\rho_1\nu_1^{\mathrm{ac}}\) has positive mass, so
\(\rho_1>0\) on a positive-\(\nu_1^{\mathrm{ac}}\)-measure subset of
\(S_1^{\mathrm{ac}}\setminus M_1\), where \(p>0\). The case
\(\nu_1^{\mathrm{ac}}(\Omega_1)>0\) is symmetric. If both absolutely continuous masses vanish, then
\(\nu_0^\perp(\Omega_0)>0\), so choose \(x\in\widehat S_0^\perp\). By
Assumption~\ref{assump:closed_polish_graph_nonempty_nbhd},
\(N_1(x)\neq\emptyset\), while
\(N_1(x)\subseteq N_1(S_0^\perp)\subseteq M_1\), where \(p=2\). Thus \(p\in E_+\setminus\{0\}\).

\paragraph{Individual optimality on the singular parts.}
Fix \(x\in\widehat S_0^\perp\). Then \(p(x)=0\), \(a(x)=\delta_x\oplus0\), and
\(p=2\) on \(N_1(x)\). If \(z=z_0\oplus z_1\in\mathcal B(x,p)\), then
\(0\le\langle p,z\rangle\le p(x)=0\), hence \(z_1(N_1(x))=0\).
Thus \(u_x(z)=0=u_x(a(x))\). The argument for
\(y\in\widehat S_1^\perp\) is symmetric.

\paragraph{Individual optimality on the absolutely continuous parts.}
Fix \(x\in\widehat S_0^{\mathrm{ac}}\) and set
\(c_x:=(1+\rho_0(x))^{-1}\). Since \(x\notin M_0\),
\(p(x)=\rho_0(x)/(1+\rho_0(x))=\rho_0(x)c_x\).

For \(y\in N_1(x)\), if \(y\in M_1\), then \(p(y)=2\ge c_x\). Otherwise
\(y\notin M_1\) and, since \(R_1\subseteq M_1\),
\(y\in S_1^{\mathrm{ac}}\cup S_1^\perp\). The latter alternative would imply
\(x\in N_0(S_1^\perp)\subseteq M_0\), impossible. Thus \(y\in S_1^{\mathrm{ac}}\), and the defining property of \(M_1\) gives
\(\rho_0(x)\rho_1(y)\ge1\), hence
\(p(y)=\rho_1(y)/(1+\rho_1(y))\ge(1+\rho_0(x))^{-1}=c_x\).
Therefore \(p(y)\ge c_x\) for every \(y\in N_1(x)\).

If \(z=z_0\oplus z_1\in\mathcal B(x,p)\), then
\[
u_x(z)
\le
c_x^{-1}\int_{N_1(x)}p(y)\,z_1(dy)
\le
c_x^{-1}\langle p,z\rangle
\le
c_x^{-1}p(x)
=
\rho_0(x).
\]
On the other hand, \(K_1(x,\cdot)\) is carried by
\(N_1(x)\cap S_1^{\mathrm{ac}}\setminus M_1\), and
\(\rho_0(x)\rho_1(y)=1\) for \(K_1(x,\cdot)\)-a.e.\ \(y\), so
\(p(y)=c_x\) there. Since \(K_1(x,\Omega_1)=\rho_0(x)\),
\(\langle p,a(x)\rangle=c_x\rho_0(x)=p(x)\) and
\(u_x(a(x))=K_1(x,N_1(x))=\rho_0(x)\).
Thus \(a(x)\) is affordable and utility-maximizing.

The argument for \(y\in\widehat S_1^{\mathrm{ac}}\) is symmetric. With
\(d_y:=(1+\rho_1(y))^{-1}\), every \(x\in N_0(y)\) satisfies
\(p(x)\ge d_y\), while \(p(x)=d_y\) for \(K_0(y,\cdot)\)-a.e.\ \(x\). Hence
\(u_y(z)\le\rho_1(y)\) for every \(z\in\mathcal B(y,p)\), whereas
\(\langle p,a(y)\rangle=p(y)\) and \(u_y(a(y))=\rho_1(y)\).

The sets
\(\widehat S_0^{\mathrm{ac}},\widehat S_0^\perp,
\widehat S_1^{\mathrm{ac}},\widehat S_1^\perp\)
cover \(T\) up to a \(\mu\)-null set. Thus individual optimality holds \(\mu\)-a.e.; together with feasibility and \(p\in E_+\setminus\{0\}\), this proves that \((a,p)\) is Walrasian.
\end{proof}

\begin{remark}[Mutual proportional response is not needed]
\label{rem:pr_not_needed_for_walras}
The hypotheses of Theorem~\ref{thm:refinement_pair_construction_walras} do not require
\(\pi_0\) and \(\pi_1\) to be proportional responses, and this is genuinely weaker.

Take
\(\Omega_0=\{x_1,x_2\}\), \(\Omega_1=\{y_1,y_2\}\), with discrete \(\sigma\)-algebras,
unit measures \(\nu_0(x_1)=\nu_0(x_2)=1\),
\(\nu_1(y_1)=\nu_1(y_2)=1\), and
\(\mathcal E=\Omega_0\times\Omega_1\). Define
\[
\pi_0=
\begin{pmatrix}
1&0\\
0&1
\end{pmatrix},
\qquad
\pi_1=
\begin{pmatrix}
\frac14&\frac34\\[0.3em]
\frac34&\frac14
\end{pmatrix}.
\]
Both are level-optimal maximin: for the complete bipartite relation,
\(\Fit_0(t)=\Fit_1(t)=\min\{2,2t\}=2\min\{1,t\}\), while both opposite marginals equal \((1,1)\), so
\(P_t^{(\pi_0)}(\Omega_1)=P_t^{(\pi_1)}(\Omega_0)=2\min\{1,t\}\).

They are not proportional responses: the payload of \(\pi_0\) on \(\Omega_1\) is identically \(1\), so
\(\PR_1(\pi_0)=\pi_0\neq\pi_1\); likewise
\(\PR_0(\pi_1)=\pi_1\neq\pi_0\).

Nevertheless, Construction~\ref{const:allocation_price_from_refinement_pair} yields a Walrasian equilibrium. Here
\(\rho_0^{\mathrm{ac}}\equiv1\), \(\rho_1^{\mathrm{ac}}\equiv1\), and \(p\equiv\frac12\), while
\[
a(x_1)=0\oplus\Bigl(\frac14\delta_{y_1}+\frac34\delta_{y_2}\Bigr),
\qquad
a(x_2)=0\oplus\Bigl(\frac34\delta_{y_1}+\frac14\delta_{y_2}\Bigr),
\]
\[
a(y_1)=\delta_{x_1}\oplus0,
\qquad
a(y_2)=\delta_{x_2}\oplus0.
\]
The allocation is feasible. Every agent faces price \(1/2\) and values total mass on the opposite side, so every affordable bundle has utility at most \(1\), while each assigned bundle attains \(1\). Hence \((a,p)\) is Walrasian. Thus mutual proportional response is not needed in Theorem~\ref{thm:refinement_pair_construction_walras}.
\end{remark}

\subsection{From Walrasian Equilibrium to Level-Optimal Maximin Pairs}
\label{subsec:equilibrium_to_level_optimal_pairs}

We now prove the converse. From a feasible allocation we retain utility-relevant graph-compatible cross-side consumption and assign the residual mass through measurable fallback kernels, thereby producing feasible refinements; Walrasian optimality will then force the level-optimal maximin structure.

\begin{construction}[Extracting a refinement pair from a feasible allocation]
\label{const:extract_refinement_from_allocation}
Work under Assumption~\ref{assump:closed_polish_graph_nonempty_nbhd} and in the setting of
Subsection~\ref{subsec:graph_to_continuum_economy}. Thus
\(T=\Omega_0\sqcup\Omega_1\), \(L_+=\mathcal M_+(T)\), and
\(\mathfrak E_{\mathcal E}\) is the induced continuum economy.

Let \(a:T\to\mathcal M_+(T)\) be feasible. Write
\(a(t)=a_0(t)\oplus a_1(t)\), with
\(a_\iota(t)\in\mathcal M_+(\Omega_\iota)\), and define finite kernels
\(K_\iota(t,A):=a_\iota(t)(A)\).

\paragraph{Extraction of the side-\(0\) refinement.}
Define \(\widetilde\pi_0\in\mathcal M_+(\Omega_0\times\Omega_1)\) by
\[
\widetilde\pi_0(C)
:=
\int_{\Omega_1}\int_{\Omega_0}
\mathbf1_C(x,y)\mathbf1_{\mathcal E}(x,y)\,
K_0(y,dx)\,\nu_1(dy),
\qquad
C\in\Sigma_0\otimes\Sigma_1,
\]
equivalently
\(\widetilde\pi_0(A\times B)=\int_BK_0(y,A\cap N_0(y))\,\nu_1(dy)\).
Let \(\widetilde\nu_0:=(\pr_0)_\#\widetilde\pi_0\). Feasibility gives, for \(A\in\Sigma_0\),
\[
\widetilde\nu_0(A)
\le
\int_{\Omega_1}K_0(y,A)\,\nu_1(dy)
\le
\int_Ta_t(A)\,\mu(dt)
=
\nu_0(A),
\]
so \(\widetilde\nu_0\le\nu_0\). Set \(\lambda_0:=\nu_0-\widetilde\nu_0\).

The correspondence \(x\mapsto N_1(x)\) has nonempty closed values, and for open
\(U\subseteq\Omega_1\),
\(\{x:N_1(x)\cap U\neq\emptyset\}
=\pr_0(\mathcal E\cap(\Omega_0\times U))\)
is analytic, hence measurable in the \(\lambda_0\)-completion of \(\Sigma_0\).
The Kuratowski--Ryll-Nardzewski theorem
\cite{KuratowskiRyllNardzewski1965} yields a completion-measurable selector
\(s_0:\Omega_0\to\Omega_1\) with \(s_0(x)\in N_1(x)\). Since \(\Omega_1\) is standard Borel, take a \(\Sigma_0\)-measurable version and define
\(G_0^{\mathrm{fb}}(x,B):=\mathbf1_B(s_0(x))\). Then
\(G_0^{\mathrm{fb}}(x,N_1(x))=1\) for \(\lambda_0\)-a.e.\ \(x\).

Set
\(\pi_0^{\mathrm{fb}}:=\lambda_0\kprod G_0^{\mathrm{fb}}\) and
\(\pi_0:=\widetilde\pi_0+\pi_0^{\mathrm{fb}}\).
Then \(\pi_0\) is concentrated on \(\mathcal E\) and
\((\pr_0)_\#\pi_0=\widetilde\nu_0+\lambda_0=\nu_0\), so
\(\pi_0\in\Pi_0^{\mathcal E}\).

\paragraph{Extraction of the side-\(1\) refinement.}
Define \(\widetilde\pi_1\in\mathcal M_+(\Omega_0\times\Omega_1)\) by
\[
\widetilde\pi_1(C)
:=
\int_{\Omega_0}\int_{\Omega_1}
\mathbf1_C(x,y)\mathbf1_{\mathcal E}(x,y)\,
K_1(x,dy)\,\nu_0(dx),
\qquad
C\in\Sigma_0\otimes\Sigma_1,
\]
equivalently
\(\widetilde\pi_1(A\times B)=\int_AK_1(x,B\cap N_1(x))\,\nu_0(dx)\).
Let \(\widetilde\nu_1:=(\pr_1)_\#\widetilde\pi_1\). For \(B\in\Sigma_1\),
\[
\widetilde\nu_1(B)
\le
\int_{\Omega_0}K_1(x,B)\,\nu_0(dx)
\le
\int_Ta_t(B)\,\mu(dt)
=
\nu_1(B),
\]
so \(\widetilde\nu_1\le\nu_1\). Set
\(\lambda_1:=\nu_1-\widetilde\nu_1\).

The same selection argument for \(y\mapsto N_0(y)\) gives a measurable probability kernel
\(G_1^{\mathrm{fb}}:\Omega_1\times\Sigma_0\to[0,1]\) with
\(G_1^{\mathrm{fb}}(y,N_0(y))=1\) for \(\lambda_1\)-a.e.\ \(y\).
Set
\(\pi_1^{\mathrm{fb}}:=\lambda_1\kprod G_1^{\mathrm{fb}}\) and
\(\pi_1:=\widetilde\pi_1+\pi_1^{\mathrm{fb}}\).
Then \(\pi_1\) is concentrated on \(\mathcal E\) and
\((\pr_1)_\#\pi_1=\widetilde\nu_1+\lambda_1=\nu_1\), so
\(\pi_1\in\Pi_1^{\mathcal E}\).

\paragraph{Output.}
The construction returns
\((\pi_0,\pi_1)\in\Pi^{\mathcal E}(\nu_0,\nu_1)\).
\end{construction}

The next lemma records the additional structure imposed by Walrasian optimality.

\begin{lemma}[Structure of the refinement extracted from a Walrasian equilibrium]
\label{lem:walras_extracted_refinement_structure}
Work under Assumption~\ref{assump:closed_polish_graph_nonempty_nbhd} and in the setting of
Subsection~\ref{subsec:graph_to_continuum_economy}.

Let \((a,p)\) be a Walrasian equilibrium of \(\mathfrak E_{\mathcal E}\), and let
\(\pi_1=\widetilde\pi_1+\pi_1^{\mathrm{fb}}\) be extracted by
Construction~\ref{const:extract_refinement_from_allocation}.
Let \(X_0\subseteq\Omega_0\) and \(X_1\subseteq\Omega_1\) be measurable full-measure sets on which individual optimality holds. Set
\(S_0:=\{x\in\Omega_0:p(x)>0\}\),
\(S_1:=\{y\in\Omega_1:p(y)>0\}\),
\(\rho_0(x):=a_1(x)(N_1(x))\), and
\(\rho_1(y):=a_0(y)(N_0(y))\).
The functions \(\rho_0,\rho_1\) are measurable because \(a_1,a_0\) are measurable kernels and \(\mathcal E\) is measurable.

Recall that
\(\widetilde\pi_1(A\times B)
:=\int_Aa_1(x)(B\cap N_1(x))\,\nu_0(dx)\),
\(\widetilde\nu_1:=(\pr_1)_\#\widetilde\pi_1\), and
\(\pi_1^{\mathrm{fb}}=(\nu_1-\widetilde\nu_1)\kprod G_1^{\mathrm{fb}}\).
Then:
\begin{enumerate}[(i)]
\item \textbf{Positive-price graph structure.}
One has
\(N_0(S_1^\c)\cap X_0=\emptyset\) and
\(N_1(S_0^\c)\cap X_1=\emptyset\).
Moreover, if
\((x,y)\in\mathcal E\cap(X_0\cap S_0)\times(X_1\cap S_1)\), then
\(\rho_0(x)\rho_1(y)\ge1\).

\item \textbf{The fallback part comes only from zero-price goods.}
For every measurable \(B\in\Sigma_1\) with \(B\subseteq S_1\),
\(\widetilde\nu_1(B)=\nu_1(B)\).
Equivalently, \(\nu_1-\widetilde\nu_1\) is concentrated on \(S_1^\c\).
Consequently, since the kernel \(G_1^{\mathrm{fb}}\) chosen in
Construction~\ref{const:extract_refinement_from_allocation} satisfies
\(G_1^{\mathrm{fb}}(y,N_0(y))=1\) for
\((\nu_1-\widetilde\nu_1)\)-a.e.\ \(y\), one has
\((\pr_0)_\#\pi_1^{\mathrm{fb}}\perp\nu_0\).

\item \textbf{Reciprocal identity on the carrier of the traded part.}
For \(\widetilde\pi_1\)-a.e.\ \((x,y)\),
\(\rho_0(x)\rho_1(y)=1\).
\end{enumerate}
\end{lemma}

\begin{proof}
We first record the pointwise consequences of Walrasian optimality.

\paragraph{Preliminary claim.}
For \(x\in X_0\), let \(c_1(x):=\inf_{z\in N_1(x)}p(z)\). Then:
\begin{enumerate}[(a)]
\item \(N_1(x)\subseteq S_1\).
\item If \(x\in S_0\), then \(c_1(x)>0\), the infimum is attained, and with
\(M_x:=\{z\in N_1(x):p(z)=c_1(x)\}\),
\[
a_1(x)(N_1(x)\setminus M_x)=0,
\qquad
a_1(x)((\Omega_1\setminus N_1(x))\cap S_1)=0,
\qquad
a_0(x)(S_0)=0,
\]
and \(p(x)=\rho_0(x)c_1(x)\).
\item If \(x\in S_0^\c\), then
\(a_1(x)(S_1)=a_0(x)(S_0)=0\), hence \(\rho_0(x)=0\).
\end{enumerate}
Symmetrically, for \(y\in X_1\), with
\(c_0(y):=\inf_{w\in N_0(y)}p(w)\), one has
\(N_0(y)\subseteq S_0\); if \(y\in S_1\), then \(c_0(y)>0\), the infimum is attained, and with
\(M_y:=\{w\in N_0(y):p(w)=c_0(y)\}\),
\[
a_0(y)(N_0(y)\setminus M_y)=0,
\qquad
a_0(y)((\Omega_0\setminus N_0(y))\cap S_0)=0,
\qquad
a_1(y)(S_1)=0,
\]
together with \(p(y)=\rho_1(y)c_0(y)\); if \(y\in S_1^\c\), then
\(a_0(y)(S_0)=a_1(y)(S_1)=0\) and \(\rho_1(y)=0\).

For the claim, if \(y\in N_1(x)\) and \(p(y)=0\), then
\(M\delta_y\in\mathcal B(x,p)\) has utility \(M\) for every \(M>0\), contradicting optimality. Thus \(N_1(x)\subseteq S_1\).

If \(x\in S_0\) and \(c_1(x)=0\), choose \(z_n\in N_1(x)\) with
\(p(z_n)<p(x)/n\); then \(n\delta_{z_n}\) is affordable with utility \(n\), again impossible. Thus \(c_1(x)>0\). If the infimum were not attained, every \(z\in N_1(x)\) would satisfy \(p(z)>c_1(x)\), so every positive feasible utility \(U=z_1(N_1(x))\) would satisfy
\(\int_{N_1(x)}p(w)\,z_1(dw)>c_1(x)U\), hence
\(U<p(x)/c_1(x)\). But for \(z_n\) with \(p(z_n)\downarrow c_1(x)\), the affordable bundles
\((p(x)/p(z_n))\delta_{z_n}\) have utility tending to \(p(x)/c_1(x)\), contradicting attainment of the optimum.

Fix \(z^\star\in M_x\). If any of
\(a_1(x)(N_1(x)\setminus M_x)\),
\(a_1(x)((\Omega_1\setminus N_1(x))\cap S_1)\), or
\(a_0(x)(S_0)\) were positive, the positive-price part of \(a(x)\) would cost strictly more than \(\rho_0(x)c_1(x)\). Replacing it by \(\rho_0(x)\delta_{z^\star}\) preserves utility and strictly lowers cost; the saved budget could then buy additional mass at \(z^\star\), contradicting optimality. Thus all three masses vanish and
\(\langle p,a(x)\rangle=\rho_0(x)c_1(x)\). If this were \(<p(x)\), then
\(a(x)+\varepsilon\delta_{z^\star}\) would be affordable for some \(\varepsilon>0\) and strictly preferred. Hence \(p(x)=\rho_0(x)c_1(x)\).

Finally, if \(x\in S_0^\c\), affordability and positivity of prices on \(S_0,S_1\) give
\(a_1(x)(S_1)=a_0(x)(S_0)=0\); since \(N_1(x)\subseteq S_1\), \(\rho_0(x)=0\). The symmetric assertions follow identically.
\hfill\(\diamond\)

\paragraph{Part (i).}
The claim gives
\(N_0(S_1^\c)\cap X_0=\emptyset\) and, symmetrically,
\(N_1(S_0^\c)\cap X_1=\emptyset\).
If
\((x,y)\in\mathcal E\cap(X_0\cap S_0)\times(X_1\cap S_1)\), then
\(c_1(x)\le p(y)\) and \(c_0(y)\le p(x)\). Hence
\(p(x)=\rho_0(x)c_1(x)\le\rho_0(x)p(y)\) and
\(p(y)=\rho_1(y)c_0(y)\le\rho_1(y)p(x)\), so
\(\rho_0(x)\rho_1(y)\ge1\).

\paragraph{Part (ii).}
Let \(B\in\Sigma_1\), \(B\subseteq S_1\). Feasibility gives
\[
\nu_1(B)
=
\int_{\Omega_0}a_1(x)(B)\,\nu_0(dx)
+
\int_{\Omega_1}a_1(y)(B)\,\nu_1(dy).
\]
The second integral vanishes by the symmetric claim. For \(x\in X_0\), if \(x\in S_0\), then
\(a_1(x)((\Omega_1\setminus N_1(x))\cap S_1)=0\), while if \(x\in S_0^\c\), then \(a_1(x)(S_1)=0\). Thus
\(a_1(x)(B)=a_1(x)(B\cap N_1(x))\) for \(\nu_0\)-a.e.\ \(x\), and hence
\[
\nu_1(B)
=
\int_{\Omega_0}a_1(x)(B\cap N_1(x))\,\nu_0(dx)
=
\widetilde\nu_1(B).
\]
Therefore \(\nu_1-\widetilde\nu_1\) is concentrated on \(S_1^\c\).
Since \(\pi_1^{\mathrm{fb}}
=(\nu_1-\widetilde\nu_1)\kprod G_1^{\mathrm{fb}}\) is concentrated on
\(\{(x,y):y\in S_1^\c,\ x\in N_0(y)\}\), part~(i) gives
\(N_0(S_1^\c)\subseteq\Omega_0\setminus X_0\), a \(\nu_0\)-null set. Hence
\((\pr_0)_\#\pi_1^{\mathrm{fb}}\perp\nu_0\).

\paragraph{Part (iii).}
For measurable \(A\subseteq\Omega_0\),
\[
(\pr_0)_\#\widetilde\pi_1(A)
=
\int_Aa_1(x)(N_1(x))\,\nu_0(dx)
=
\int_A\rho_0\,d\nu_0,
\]
so \((\pr_0)_\#\widetilde\pi_1=\rho_0\nu_0\).
By the claim, \(\rho_0>0\) \(\nu_0\)-a.e.\ on \(S_0\) and
\(\rho_0=0\) \(\nu_0\)-a.e.\ on \(S_0^\c\).

Part~(ii) gives agreement of
\((\pr_1)_\#\widetilde\pi_1\) with \(\nu_1\) on \(S_1\), while
\(N_1(x)\subseteq S_1\) for \(x\in X_0\) implies
\((\pr_1)_\#\widetilde\pi_1(S_1^\c)=0\). Hence
\((\pr_1)_\#\widetilde\pi_1=\nu_1\!\restriction S_1\).
Symmetrically,
\((\pr_0)_\#\widetilde\pi_0=\nu_0\!\restriction S_0\) and
\((\pr_1)_\#\widetilde\pi_0=\rho_1\nu_1\), with
\(\rho_1>0\) \(\nu_1\)-a.e.\ on \(S_1\) and \(\rho_1=0\) on \(S_1^\c\).

Disintegrate
\(\widetilde\pi_1=(\rho_0\nu_0)\kprod\Gamma\) and
\(\widetilde\pi_0=(\rho_1\nu_1)\kprod\Lambda\)
with probability kernels
\(\Gamma:\Omega_0\times\Sigma_1\to[0,1]\) and
\(\Lambda:\Omega_1\times\Sigma_0\to[0,1]\), choosing
\(\Gamma(x,N_1(x))=1\) for \((\rho_0\nu_0)\)-a.e.\ \(x\).

Define \(M\in\mathcal M_+(\Omega_0\times\Omega_0)\) by
\[
M(A\times C)
:=
\int_A\rho_0(x)
\left(
\int_{\Omega_1}\rho_1(y)\Lambda(y,C)\Gamma(x,dy)
\right)\nu_0(dx).
\]
For measurable \(C\subseteq\Omega_0\),
\[
M(\Omega_0\times C)
=
\int_{S_1}\rho_1(y)\Lambda(y,C)\,\nu_1(dy)
=
\widetilde\pi_0(C\times\Omega_1)
=
\nu_0(C\cap S_0).
\]
In particular,
\begin{equation}
\label{eq:M_total_mass_corrected_v2}
M(\Omega_0\times\Omega_0)=\nu_0(S_0).
\end{equation}
On the other hand,
\(M(A\times\Omega_0)=\int_A\Phi(x)\,\nu_0(dx)\), where
\(\Phi(x):=\rho_0(x)\int_{\Omega_1}\rho_1(y)\,\Gamma(x,dy)\).

We claim \(\Phi\ge1\) \(\nu_0\)-a.e.\ on \(S_0\). Since the two marginals of \(\widetilde\pi_1\) charge neither
\(\Omega_0\setminus X_0\) nor \(\Omega_1\setminus X_1\), one has
\(\Gamma(x,N_1(x)\cap X_1)=1\) for \(\nu_0\)-a.e.\ \(x\in S_0\).
For such \(x\in X_0\) and \(y\in N_1(x)\cap X_1\), part~(i) gives
\(\rho_0(x)\rho_1(y)\ge1\), hence \(\Phi(x)\ge1\).

By \eqref{eq:M_total_mass_corrected_v2},
\[
\int_{S_0}\Phi(x)\,\nu_0(dx)
=
\nu_0(S_0)
=
\int_{S_0}1\,\nu_0(dx).
\]
Thus \(\Phi=1\) \(\nu_0\)-a.e.\ on \(S_0\). For such \(x\), the nonnegative function
\(y\mapsto\rho_0(x)\rho_1(y)-1\) has \(\Gamma(x,\cdot)\)-integral \(0\), so
\(\rho_0(x)\rho_1(y)=1\) for \(\Gamma(x,\cdot)\)-a.e.\ \(y\). Since
\(\widetilde\pi_1(dx\,dy)=\rho_0(x)\nu_0(dx)\Gamma(x,dy)\), the identity holds
\(\widetilde\pi_1\)-a.e.
\end{proof}

The lemma provides exactly the ingredients needed for maximin: the positive-price region identifies the traded part, the fallback term is singular, and the traded mass satisfies the reciprocal relation.

\begin{theorem}[Walrasian equilibrium yields a level-optimal maximin pair]
\label{thm:walras_to_level_opt_refinement_pair}
Work under Assumption~\ref{assump:closed_polish_graph_nonempty_nbhd}, in the setting of
Subsection~\ref{subsec:graph_to_continuum_economy}, and with
Construction~\ref{const:extract_refinement_from_allocation}.

If \((a,p)\) is a Walrasian equilibrium of \(\mathfrak E_{\mathcal E}\), then the pair
\((\pi_0,\pi_1)\in\Pi^{\mathcal E}(\nu_0,\nu_1)\)
returned by Construction~\ref{const:extract_refinement_from_allocation} has both components level-optimal maximin from side \(\iota\) to side \(\bar\iota\) in the sense of Definition~\ref{defn:level_opt_maximin}.
\end{theorem}

\begin{proof}
By symmetry it suffices to treat \(\pi_1\). Choose measurable full-measure sets
\(X_0\subseteq\Omega_0\), \(X_1\subseteq\Omega_1\) on which individual optimality holds, set
\(S_0:=\{x:p(x)>0\}\), \(S_1:=\{y:p(y)>0\}\), and define
\(\rho_0(x):=a_1(x)(N_1(x))\),
\(\rho_1(y):=a_0(y)(N_0(y))\).

Set \(P_0:=(\pr_0)_\#\pi_1\). Since
\(\pi_1=\widetilde\pi_1+\pi_1^{\mathrm{fb}}\) and
\[
(\pr_0)_\#\widetilde\pi_1(A)
=
\int_Aa_1(x)(N_1(x))\,\nu_0(dx)
=
\int_A\rho_0\,d\nu_0,
\]
one has
\((\pr_0)_\#\widetilde\pi_1=\rho_0\nu_0\).
Lemma~\ref{lem:walras_extracted_refinement_structure}(ii) gives
\((\pr_0)_\#\pi_1^{\mathrm{fb}}\perp\nu_0\). Hence, for \(t\ge0\),
\(P_t^{(\pi_1)}=(\rho_0\wedge t)\nu_0\) and
\(P_t^{(\pi_1)}(\Omega_0)=\int_{\Omega_0}(\rho_0\wedge t)\,d\nu_0\).
It remains to prove
\(\Fit_0(t)=P_t^{(\pi_1)}(\Omega_0)\) for every \(t\ge0\).

\paragraph{Lower bound.}
Disintegrate
\(\widetilde\pi_1=(\rho_0\nu_0)\kprod K\) with a probability kernel
\(K:\Omega_0\times\Sigma_1\to[0,1]\). For \(t\ge0\), set
\(\tau_t:=((\rho_0\wedge t)\nu_0)\kprod K\).
Then \(\tau_t(\mathcal E^\c)=0\), because
\(\widetilde\pi_1(\mathcal E^\c)=0\) and the same kernel \(K\) is used.
Moreover,
\((\pr_0)_\#\tau_t=(\rho_0\wedge t)\nu_0\le t\nu_0\), and
\((\pr_1)_\#\tau_t\le(\pr_1)_\#\widetilde\pi_1=\widetilde\nu_1\le\nu_1\).
Thus \(\tau_t\in\Feas_0(t)\), with mass
\(\int_{\Omega_0}(\rho_0\wedge t)\,d\nu_0=P_t^{(\pi_1)}(\Omega_0)\), proving
\(\Fit_0(t)\ge P_t^{(\pi_1)}(\Omega_0)\).

\paragraph{Upper bound.}
The case \(t=0\) is immediate. Let \(t>0\) and set
\(A_t:=\{x\in X_0:\rho_0(x)\le t\}\) and
\(B_t:=\{y\in X_1\cap S_1:\rho_1(y)\ge1/t\}\).
Since \(X_0\) has full \(\nu_0\)-measure,
\(\nu_0(\Omega_0\setminus A_t)=\nu_0(\{\rho_0>t\})\).

We claim
\begin{equation}
\label{eq:N1_At_subset_Bt_clean}
N_1(A_t)\cap X_1\subseteq B_t.
\end{equation}
Indeed, if \(y\in N_1(A_t)\cap X_1\), choose \(x\in A_t\) with \(y\in N_1(x)\).
Lemma~\ref{lem:walras_extracted_refinement_structure}(i) gives
\(y\in S_1\) and \(\rho_0(x)\rho_1(y)\ge1\). Since \(\rho_0(x)\le t\),
\(\rho_1(y)\ge1/t\), proving \eqref{eq:N1_At_subset_Bt_clean}.

Lemma~\ref{lem:walras_extracted_refinement_structure}(iii) gives
\(\rho_0(x)\rho_1(y)=1\) for \(\widetilde\pi_1\)-a.e.\ \((x,y)\).
Moreover,
\(\widetilde\pi_1((\Omega_0\setminus X_0)\times\Omega_1)=0\).
On \(X_0\setminus A_t\), \(\rho_0>t\), while on \(B_t\), \(\rho_1\ge1/t\); hence reciprocity gives
\(\widetilde\pi_1((\Omega_0\setminus A_t)\times B_t)=0\).

Since \(B_t\subseteq S_1\), Lemma~\ref{lem:walras_extracted_refinement_structure}(ii) yields
\[
\nu_1(B_t)
=
\widetilde\nu_1(B_t)
=
\widetilde\pi_1(A_t\times B_t)
\le
\widetilde\pi_1(A_t\times\Omega_1)
=
\int_{A_t}\rho_0\,d\nu_0.
\]
Together with \eqref{eq:N1_At_subset_Bt_clean} and
\(\nu_1(\Omega_1\setminus X_1)=0\), this gives
\(\nu_1(N_1(A_t))\le\nu_1(B_t)\le\int_{A_t}\rho_0\,d\nu_0\).

Apply Proposition~\ref{prop:measurable_maxflow_mincut_bipartite} with capacities
\(a:=t\nu_0\) and \(b:=\nu_1\). Using the cut
\(C:=\Omega_0\setminus A_t\),
\[
\Fit_0(t)
\le
t\nu_0(\Omega_0\setminus A_t)+\nu_1(N_1(A_t))
\le
t\nu_0(\Omega_0\setminus A_t)+\int_{A_t}\rho_0\,d\nu_0
=
\int_{\Omega_0}(\rho_0\wedge t)\,d\nu_0
=
P_t^{(\pi_1)}(\Omega_0).
\]
Together with the lower bound,
\(\Fit_0(t)=P_t^{(\pi_1)}(\Omega_0)\) for every \(t\ge0\). Hence \(\pi_1\) is level-optimal maximin from side \(1\) to side \(0\). Interchanging the two sides proves the same for \(\pi_0\).
\end{proof}

Taken together, Theorems~\ref{thm:refinement_pair_construction_walras}
and~\ref{thm:walras_to_level_opt_refinement_pair} show that the continuum-economy and level-optimal maximin formulations determine the same structure: level-optimal maximin pairs yield Walrasian equilibria, and every Walrasian equilibrium yields such a pair through the extraction procedure. Thus the equilibrium viewpoint is a characterization of the same measure-theoretic object rather than an auxiliary application.

\newpage

\appendix

\section*{Appendix}

\section{Failure of Attainment on a Nonclosed Relation}
\label{app:closed-support-needed}

The next example shows why closedness of the relation constraint enters
the attainment theory, in particular
Proposition~\ref{prop:vartheta_one_sided_attainment}. The obstruction is
loss of feasibility under weak limits: a minimizing sequence
concentrates on the boundary \(y=x\), while its limit is excluded by
the nonclosed relation but becomes feasible after passage to the
closure. We add one null edge to the strict upper-triangular relation
so that every vertex has a nonempty neighborhood. Thus all parts of
Assumption~\ref{assump:closed_polish_graph_nonempty_nbhd} hold except
closedness.

Fix \(\Omega_0=\Omega_1=[0,1]\), \(\nu_0=\nu_1=\lambda\), where
\(\lambda\) denotes Lebesgue measure, and let \(\vartheta(t)=t^2\).
For measurable \(S\subseteq[0,1]^2\), write
\[
\Pi^S
:=
\left\{
\pi\in\mathcal M_+([0,1]^2):
\pi_0=\lambda,\
\pi(S)=\pi([0,1]^2)
\right\},
\]
where \(\pi_0\) and \(\pi_1\) denote the first and second marginals.
For \(\pi\in\Pi^S\), define
\[
J(\pi)
:=
\begin{cases}
\displaystyle
\int_0^1
\left(
\frac{d\pi_1}{d\lambda}(y)
\right)^2\,dy,
&
\pi_1\ll\lambda,
\\[2mm]
+\infty,
&
\text{otherwise}.
\end{cases}
\]
This is the one-sided objective associated with \(\vartheta(t)=t^2\).

When \(\pi\ll\lambda\otimes\lambda\), write
\[
\pi(dx\,dy)
=
\alpha(x,y)\mathbf 1_S(x,y)\,dx\,dy.
\]
Then \(\pi_0=\lambda\) is equivalent to
\[
\int_0^1
\alpha(x,y)\mathbf 1_S(x,y)\,dy
=
1
\qquad
\text{for }\lambda\text{-a.e. }x\in[0,1],
\]
and
\[
\frac{d\pi_1}{d\lambda}(y)
=
\int_0^1
\alpha(x,y)\mathbf 1_S(x,y)\,dx.
\]

\begin{proposition}[Failure of attainment on a nonclosed relation]
\label{prop:closed-support-needed}
Let
\[
\mathcal U
:=
\{(x,y)\in[0,1]^2:y>x\},
\qquad
\mathcal E^\circ
:=
\mathcal U\cup\{(1,0)\},
\]
and
\[
\overline{\mathcal E^\circ}
=
\{(x,y)\in[0,1]^2:y\ge x\}\cup\{(1,0)\}.
\]
Then \(\mathcal E^\circ\) is Borel and nonclosed, and every vertex on
either side has a nonempty neighborhood. Moreover:
\begin{enumerate}
\item
\[
\inf_{\pi\in\Pi^{\mathcal E^\circ}}J(\pi)=1,
\]
but the infimum is not attained.

\item
For each \(\varepsilon\in(0,\tfrac12)\), there exists
\(\pi_\varepsilon\in\Pi^{\mathcal E^\circ}\) such that
\[
J(\pi_\varepsilon)
=
1+\frac76\,\varepsilon.
\]
Moreover, \(\pi_\varepsilon\Rightarrow\pi^\Delta\) weakly as
\(\varepsilon\downarrow0\), where
\[
\pi^\Delta
:=
(\mathrm{id},\mathrm{id})_\#\lambda
\]
is the diagonal plan.

\item
The diagonal plan \(\pi^\Delta\) belongs to
\(\Pi^{\overline{\mathcal E^\circ}}\) and satisfies
\(J(\pi^\Delta)=1\). Hence
\[
\min_{\pi\in\Pi^{\overline{\mathcal E^\circ}}}J(\pi)=1.
\]
\end{enumerate}
\end{proposition}

\begin{proof}
The vertex spaces are compact Polish and \(\mathcal E^\circ\) is
Borel. It is not closed, since
\((0,1/n)\in\mathcal E^\circ\) while
\((0,1/n)\to(0,0)\notin\mathcal E^\circ\).
Every neighborhood is nonempty: if \(x<1\), then
\((x,(x+1)/2)\in\mathcal U\), while \((1,0)\in\mathcal E^\circ\);
if \(y>0\), then \((y/2,y)\in\mathcal U\), while \(y=0\) is adjacent
to \(x=1\).

We first record a universal lower bound. Let
\(S\subseteq[0,1]^2\) be measurable and \(\pi\in\Pi^S\). If
\(J(\pi)<+\infty\), write
\(p:=d\pi_1/d\lambda\). Since
\(\pi_1([0,1])=\pi([0,1]^2)=\pi_0([0,1])=1\), one has
\(\int_0^1p\,d\lambda=1\). Jensen's inequality gives
\[
J(\pi)
=
\int_0^1p(y)^2\,dy
\ge
\left(
\int_0^1p(y)\,dy
\right)^2
=
1.
\]
Equality holds if and only if \(p=1\) \(\lambda\)-a.e., equivalently
\(\pi_1=\lambda\).

Suppose now that \(\pi\in\Pi^{\mathcal E^\circ}\) satisfies
\(J(\pi)=1\). Then \(\pi_1=\lambda\). Since \(\pi_0=\lambda\),
\[
\pi\bigl(\{(1,0)\}\bigr)
\le
\pi\bigl(\{1\}\times[0,1]\bigr)
=
\pi_0(\{1\})
=
0.
\]
As
\(\mathcal E^\circ=\mathcal U\cup\{(1,0)\}\) and
\(\pi(\mathcal E^\circ)=1\), we have \(\pi(\mathcal U)=1\). Hence
\(y>x\) for \(\pi\)-a.e.\ \((x,y)\), so
\[
\int_{[0,1]^2}(y-x)\,d\pi>0.
\]
But the marginal identities give
\[
\int_{[0,1]^2}x\,d\pi
=
\int_0^1x\,dx
=
\frac12,
\qquad
\int_{[0,1]^2}y\,d\pi
=
\int_0^1y\,dy
=
\frac12,
\]
and therefore
\[
\int_{[0,1]^2}(y-x)\,d\pi=0,
\]
a contradiction. Thus the infimum on
\(\Pi^{\mathcal E^\circ}\) is not attained.

To construct a minimizing sequence, fix
\(\varepsilon\in(0,\tfrac12)\), set
\(\ell_\varepsilon(x):=\min\{\varepsilon,1-x\}\), and define
\[
\alpha_\varepsilon(x,y)
:=
\begin{cases}
\displaystyle
\frac{1}{\ell_\varepsilon(x)}
\mathbf 1_{\{x<y<x+\ell_\varepsilon(x)\}},
&
x<1,
\\[2mm]
0,
&
x=1.
\end{cases}
\]
Let
\[
\pi_\varepsilon(dx\,dy)
:=
\alpha_\varepsilon(x,y)\,dx\,dy.
\]
Then \(\pi_\varepsilon\) is concentrated on
\(\mathcal U\subseteq\mathcal E^\circ\), and
\[
\int_0^1\alpha_\varepsilon(x,y)\,dy
=
1
\qquad
\text{for }\lambda\text{-a.e. }x.
\]
Hence \(\pi_\varepsilon\in\Pi^{\mathcal E^\circ}\).

Let \(p_\varepsilon:=d(\pi_\varepsilon)_1/d\lambda\). A direct
calculation gives
\[
p_\varepsilon(y)
=
\begin{cases}
\displaystyle
\frac{y}{\varepsilon},
&
0<y<\varepsilon,
\\[2mm]
1,
&
\varepsilon\le y\le1-\varepsilon,
\\[2mm]
\displaystyle
\frac{1-y}{\varepsilon}
-
\log\left(\frac{1-y}{\varepsilon}\right),
&
1-\varepsilon<y<1.
\end{cases}
\]
Indeed, for \(0<y<\varepsilon\), only \(x\in(0,y)\) contributes, so
\(p_\varepsilon(y)=\int_0^y dx/\varepsilon=y/\varepsilon\).
For \(\varepsilon\le y\le1-\varepsilon\), the contributing interval
is \(x\in(y-\varepsilon,y)\), so \(p_\varepsilon(y)=1\). Finally, if
\(1-\varepsilon<y<1\),
\[
\begin{aligned}
p_\varepsilon(y)
&=
\int_{y-\varepsilon}^{1-\varepsilon}
\frac{dx}{\varepsilon}
+
\int_{1-\varepsilon}^{y}
\frac{dx}{1-x}
\\
&=
\frac{1-y}{\varepsilon}
-
\log\left(\frac{1-y}{\varepsilon}\right).
\end{aligned}
\]

Consequently,
\[
\begin{aligned}
J(\pi_\varepsilon)
&=
\int_0^\varepsilon
\left(\frac{y}{\varepsilon}\right)^2\,dy
+
\int_\varepsilon^{1-\varepsilon}1\,dy
+
\int_{1-\varepsilon}^1
\left[
\frac{1-y}{\varepsilon}
-
\log\left(\frac{1-y}{\varepsilon}\right)
\right]^2\,dy
\\
&=
\frac{\varepsilon}{3}
+
(1-2\varepsilon)
+
\varepsilon
\int_0^1(u-\log u)^2\,du.
\end{aligned}
\]
Since
\[
\int_0^1(u-\log u)^2\,du
=
\int_0^1u^2\,du
-
2\int_0^1u\log u\,du
+
\int_0^1(\log u)^2\,du
=
\frac{17}{6},
\]
we obtain
\[
J(\pi_\varepsilon)
=
\frac{\varepsilon}{3}
+
(1-2\varepsilon)
+
\frac{17}{6}\varepsilon
=
1+\frac76\,\varepsilon.
\]
Thus \(J(\pi_\varepsilon)\downarrow1\), and the lower bound yields
\[
\inf_{\pi\in\Pi^{\mathcal E^\circ}}J(\pi)=1.
\]

The same family exhibits boundary concentration. Let
\(f\in C([0,1]^2)\), with modulus of continuity \(\omega_f\).
Since
\[
0<y-x<\ell_\varepsilon(x)\le\varepsilon
\]
for \(\pi_\varepsilon\)-a.e.\ \((x,y)\),
\[
\left|
\int_{[0,1]^2}f\,d\pi_\varepsilon
-
\int_0^1f(x,x)\,dx
\right|
\le
\omega_f(\varepsilon).
\]
Hence
\[
\pi_\varepsilon\Rightarrow\pi^\Delta
\qquad
\text{as }\varepsilon\downarrow0.
\]

Finally, the diagonal is contained in
\(\overline{\mathcal E^\circ}\), and both marginals of
\(\pi^\Delta\) equal \(\lambda\). Hence
\[
\pi^\Delta\in\Pi^{\overline{\mathcal E^\circ}},
\qquad
J(\pi^\Delta)=1.
\]
The universal lower bound shows that this is the minimum over
\(\Pi^{\overline{\mathcal E^\circ}}\).
\end{proof}

\begin{remark}[Absolute continuity remains too restrictive]
\label{rem:closed-support-needed-density}
The preceding proposition concerns the unrestricted plan problem.
If one instead restricts the problem on
\(\overline{\mathcal E^\circ}\) to plans absolutely continuous with
respect to \(\lambda\otimes\lambda\), nonattainment persists.

Indeed, let
\[
\pi(dx\,dy)
=
\alpha(x,y)
\mathbf 1_{\overline{\mathcal E^\circ}}(x,y)\,dx\,dy
\]
be such a plan. If \(J(\pi)=1\), the lower-bound argument gives
\(\pi_1=\lambda\). Since \(\pi\ll\lambda\otimes\lambda\) and
\((\lambda\otimes\lambda)(\{(x,y):x=y\})=0\), one has
\(\pi(\{(x,y):x=y\})=0\). Moreover,
\(\pi(\{(1,0)\})=0\), either by absolute continuity or from
\(\pi_0=\lambda\). Since
\[
\overline{\mathcal E^\circ}
=
\mathcal U
\cup
\{(x,y):x=y\}
\cup
\{(1,0)\}
\]
and \(\pi(\overline{\mathcal E^\circ})=1\), it follows that
\(\pi(\mathcal U)=1\). Hence \(y>x\) for
\(\pi\)-a.e.\ \((x,y)\), and the same first-moment comparison yields
a contradiction. Thus the problem over plans in
\(\Pi^{\overline{\mathcal E^\circ}}\) that are absolutely continuous
with respect to \(\lambda\otimes\lambda\) still has infimum \(1\), but
no minimizer.

The family
\((\pi_\varepsilon)_{0<\varepsilon<1/2}\) constructed in the proof
belongs to this restricted class and still satisfies
\[
J(\pi_\varepsilon)\downarrow1.
\]
\end{remark}

\newpage

\section*{Declaration of AI usage}

During the preparation of this work, the author used generative AI for assistance with mathematical proofreading and consistency checks of author-written arguments, literature and reference checks, equation typesetting, prose presentation, and the correction of typographical and linguistic errors. After using this tool, the author reviewed and edited the content as needed, independently verified the mathematical arguments and references, and takes full responsibility for the content of the article.

\bibliographystyle{alpha}
\bibliography{ref}

\end{document}